\documentclass[reqno,centertags,11pt]{amsart}
\usepackage{amsmath,amsthm,amscd,amssymb}
\usepackage{latexsym}
\usepackage{enumerate, verbatim}
\usepackage{mathtools}  

\usepackage[colorlinks=true]{hyperref}

\newcommand{\Hmm}[1]{\leavevmode{\marginpar{\tiny%
$\hbox to 0mm{\hspace*{-0.5mm}$\leftarrow$\hss}%
\vcenter{\vrule depth 0.1mm height 0.1mm width \the\marginparwidth}%
\hbox to 0mm{\hss$\rightarrow$\hspace*{-0.5mm}}$\\\relax\raggedright #1}}}

\newcommand{\N}{{\mathbb{N}}}

\newcommand{\R}{{\mathbb{R}}}

\newcommand{\C}{{\mathbb{C}}}

\newcommand{\f}{\frac}

\newcommand{\ol}{\overline}

\newcommand{\wti}{\widetilde}

\newcommand{\hatt}{\widehat}
\newcommand{\beq}{\begin{equation}}
\newcommand{\eeq}{\end{equation}}
\newcommand{\bdm}{\begin{displaymath}}
\newcommand{\edm}{\end{displaymath}}
\newcommand{\ba}{\begin{align}}
\newcommand{\ea}{\end{align}}
\newcommand{\bpf}{\begin{proof}}
\newcommand{\epf}{\end{proof}}

\newcommand{\la}{\langle}
\newcommand{\ra}{\rangle}

\newcommand{\supp}{\mathrm{supp}\, }               
\newcommand{\dist}{\mathrm{dist}}               

\newcommand{\veps}{\varepsilon}

\newcommand{\re}{\mathrm{Re}}
\newcommand{\im}{\mathrm{Im}}

\newcommand{\dav}{{d_{\mathrm{av}}}}

\newcommand{\sgn}[1]{{\mathrm{sgn}(#1)}}

\newcommand{\id}{\mathbf{1}}                

\allowdisplaybreaks

\newcommand{\calF}{\mathcal{F}}
\newcommand{\calG}{\mathcal{G}}

\newtheorem{theorem}{Theorem}
\newtheorem{proposition}[theorem]{Proposition}
\newtheorem{lemma}[theorem]{Lemma}
\newtheorem{corollary}[theorem]{Corollary}

\theoremstyle{definition}

\newtheorem{definition}[theorem]{Definition}

\newtheorem{remark}[theorem]{Remark}
\newtheorem{remarks}[theorem]{Remarks}

\newcounter{theoremi}[theorem]

\numberwithin{theorem}{section}
\numberwithin{equation}{section}

\newcounter{smalllist}

\newcounter{listi}
\newenvironment{theoremlist}{\begin{list}{{\rm(\roman{listi})}}{%
\setlength{\topsep}{0mm}\setlength{\parsep}{0mm}\setlength{\itemsep}{0mm}%
\setlength{\labelwidth}{1.5em}\setlength{\leftmargin}{1.7em}\usecounter{listi}%
}}{\end{list}}

\newcounter{smallenum}

\newcounter{assumptions}
\newenvironment{assumptions}{\begin{list}{\textbf{A\rm\arabic{assumptions}})}{%
\setlength{\topsep}{0mm}\setlength{\parsep}{0mm}\setlength{\itemsep}{0mm}%
\setlength{\labelwidth}{2em}\setlength{\leftmargin}{2em}\usecounter{assumptions}%
}}{\end{list}}

\begin{document}

\title[Thresholds for existence of DMS for general nonlinearities]{Thresholds for existence of dispersion management solitons for general nonlinearities}
\author[M.-R.~Choi, D.~Hundertmark, Y.-R.~Lee]{Mi-Ran Choi, Dirk Hundertmark, Young-Ran~Lee}
\address{Department of Mathematics, Sogang University, 35 Baekbeom-ro,
    Mapo-gu, Seoul 04107, Korea.}%
\email{rani9030@sogang.ac.kr}
\address{Department of Mathematics, Institute for Analysis, Karlsruhe Institute of Technology, 76128 Karlsruhe, Germany.}%
\email{dirk.hundertmark@kit.edu}%
\address{Department of Mathematics, Sogang University,  35 Baekbeom-ro,
    Mapo-gu, Seoul 04107, South Korea.}%
\email{younglee@sogang.ac.kr}

\thanks{\copyright 2016 by the authors. Faithful reproduction of this article,
        in its entirety, by any means is permitted for non-commercial purposes}
\date{\today, version continuous-general-nonlinearity-7-3.tex}

\begin{abstract}
We prove a threshold phenomenon for the existence of solitary solutions of the dispersion management equation for positive and zero average dispersion for a large class of nonlinearities. These solutions are found as minimizers of
nonlinear and nonlocal variational problems which are invariant under
a large non--compact group. There exists a threshold such that minimizers exist when the power of the solitons is bigger than the threshold. Our proof of existence of minimizers is
rather direct and avoids the use of Lions' concentration compactness
argument. The existence of dispersion managed
solitons is shown under very mild conditions on the
dispersion profile and the nonlinear polarization of optical active medium, which cover all physically relevant cases for the dispersion profile and a large class of nonlinear polarizations, for example, they are allowed to change sign.
\end{abstract}

\maketitle
{\hypersetup{linkcolor=black}
\tableofcontents}

\section{Introduction}\label{introduction}
\subsection{The variational problems}
\label{sec:intro-1}
We show the existence of minimizers for a family of nonlocal and nonlinear variational problems
 \beq\label{eq:min}
    E_\lambda^{\dav}\coloneqq \inf \left\{ H(f): \|f\|^2=\lambda\right\},
 \eeq
where $\lambda>0$, the average dispersion $\dav\ge 0$, $\|f\|^2=\int_\R |f|^2\, dx$, the Hamiltonian takes the form
 \beq\label{eq:energy}
    H(f)\coloneqq\f{d_{\text{av}}}{2}\|f'\|^2-N(f) ,
 \eeq
and the nonlocal nonlinearity is given by
 \beq\label{eq:N(f)}
   N(f)\coloneqq\iint_{\R^2} V(|T_r f(x)|)dx \psi(r) dr .
 \eeq
 Here $V:[0,\infty)\to \R$ is a suitable nonlinear potential and $T_r=e^{ir\partial_x^2}$ is the solution operator of the free Schr\"{o}dinger equation in one dimension.
 The function $\psi$ is assumed to be in suitable $L^p$-spaces.

 If $\dav>0$ then, strictly speaking, the infimum in \eqref{eq:min} is taken over all $f$ with additionally $f\in H^1(\R)$, the usual Sobolev space of square integrable functions whose distributional derivative $f'$ is also square integrable.
 One can recover our formulation \eqref{eq:min} by setting
 $\|f'\|\coloneqq\infty$ if $f\in L^2\setminus H^1$.

 Our interest in these variational problems stems from the fact that the minimizers of \eqref{eq:min} are the building blocks for (quasi-)periodic breather type solutions, the so-called dispersion management solitons, of the dispersion managed nonlinear Schr\"odinger equation.
 The function $\psi$ is the density of a probability measure related to the local dispersion profile $d_0$, see the discussion in Section \ref{sec:connection-nonlinear-optics},  especially Lemma \ref{lem:properties-mu} and formula \eqref{eq:density-psi}.
 The dispersion management solitons have attracted a lot of interest in the development of ultrafast longhaul optical data transmission fibers. So far, it has mainly been studied for a Kerr-type nonlinearity, i.e., the special case where $V(a)=a^4$. The purpose of this work is to extend our previous existence results from \cite{HuLee2012} to  a large class of nonlinearities $V$ and also to positive average dispersion. We address the connection of the above variational problems with nonlinear optics later in Section \ref{sec:connection-nonlinear-optics}.

 The standard approach to show the existence of a minimizer of \eqref{eq:min} is to identify it as the strong limit of a suitable minimizing sequence, that is, a sequence
 $(f_n)_{n\in\N}\subset L^2(\R)$, with $\|f_n\|^2 =\lambda$ and
 $E^{\dav}_\lambda = \lim_{n\to \infty} H(f_n)$. The catch is that the above  variational problem is invariant under translations of $L^2(\R)$ if $\dav>0$ and under translations and boosts, that is, shifts in Fourier space, if $\dav=0$. This invariance under a large non-compact group of transformations leads to a loss of compactness since minimizing sequences can easily converge weakly to zero. The usual strategy to compensate for such a loss of compactness is Lions' concentration compactness method. In a previous paper, \cite{HuLee2012}, we used an alternative approach, which directly showed that  modulo the natural symmetries of the problem, minimizing sequences stay compact. The tools were very much tailored to the special type of Kerr nonlinearity. This paper extends our approach from \cite{HuLee2012} to a much more general setting. This extension is by no means straightforward, see Section \ref{sec:nonlinear-estimates} and Remark \ref{rem:nonlinearity}.

Our main assumptions on the nonlinear potential $V:\R_+\to\R $ are
\begin{assumptions}
\item \label{ass:A1}  $V$ is continuous on $\R_+=[0,\infty)$ and differentiable on $(0,\infty)$ with $V(0)=0$. There exist $2 \le  \gamma_1\le \gamma_2 <\infty$ such that
	\begin{align*} 
		|V'(a)|\lesssim a^{\gamma_1-1} + a^{\gamma_2-1} \quad\text{ for all } a>0.
	\end{align*}
\item \label{ass:A2} $V$ is continuous on $\R_+$ and differentiable on $(0,\infty)$ with $V(0)=0$. There exists $\gamma_0>2$ such that
	\begin{align*}
		V'(a)a \ge \gamma_0 V(a)\quad\text{ for all } a>0.
	\end{align*}
\item \label{ass:A3} There exists $a_0>0$ such that $V(a_0)>0$.
\end{assumptions}
\vspace{2pt}
Above, we use the convention $f\lesssim g$, if there exists a finite constant $C>0$ such that $f\le Cg$.

These three assumptions above are our main requirements on the nonlinear potential. For our existence results, depending on whether $\dav=0$ or $\dav>0$, we will have to pose some additional restrictions on the range of $\gamma_1\le\gamma_2$. For example, if $\dav>0$, we will need that $2< \gamma_1\le\gamma_2<10$ and we will also have some additional $L^p$ conditions on $\psi$ to ensure the existence of minimizers for \eqref{eq:min}. We will see that these conditions yield a threshold phenomenon
, that is, $f\in H^1(\R)$, respectively $f \in L^2(\R)$, if $\dav=0$, with $\|f\|^2=\lambda$ and $H(f) =E^\dav_\lambda$   exists at least for large enough power $\lambda$.
In order to guarantee the existence of solutions for arbitrarily small
$\lambda>0$, we need to strengthen assumption \ref{ass:A3} to
\vspace{2pt}

\begin{assumptions}
\setcounter{assumptions}{3}
\item \label{ass:A4} If $\dav>0$ we assume that there exist $\veps>0 $
	 and $2< \kappa_0 < 6$ such that
	\begin{align*}
		V(a)\gtrsim a^{\kappa_0} \quad\text{ for all } 0< a\le\veps.
	\end{align*}
 If $\dav = 0$ we assume that there exists $\veps>0 $ such that $V(a)>0$ for all $0<a\le \veps$.
%
\end{assumptions}

\begin{remarks}
  \begin{theoremlist}
 	\item\label{rem:simple} An integration shows that \ref{ass:A1} implies
		\begin{align*}
		|V(a)|\lesssim a^{\gamma_1} + a^{\gamma_2}.
	\end{align*}
	Much more important for us is the fact that \ref{ass:A1} allows us to control the \emph{nonlocal} nonlinearity $N$ under \emph{splitting}, see Lemma \ref{lem:V-splitting} and the discussion in  Section \ref{subsec:Vsplitting}.
  \item Examples of nonlinearities obeying assumptions \ref{ass:A1} through \ref{ass:A3} are given by
  \begin{align*}
  	V(a) = \sum_{j=1}^J c_j a^{s_j}
  \end{align*}  	
  with $c_j > 0$, $2<s_j<10$ if $\dav>0$, respectively, $2<s_j<6$ if $\dav=0$,
  and $J\in\N$, but our assumptions also allow nonlinear potentials which can become negative, for example, for $\dav>0$ we can allow for
  \begin{align*}
  	V(a) = -a^4 +a^8\quad \text{for } a\ge 0 .
  \end{align*}
  It certainly fulfills \ref{ass:A1}. Since
  \begin{align*}
  	V'(a)a = -4a^4 +8 a^8 = 4(-a^4+a^8) +4a^8 \ge 4V(a),
  \end{align*}
  it also obeys \ref{ass:A2}. Moreover,  $V(a_0)>0$ for all large enough $a_0$, so \ref{ass:A3} holds.

  Similarly, if $\dav=0$ we can allow for
 \begin{align*}
  	V(a) = -a^3 +a^5\quad \text{for } a\ge 0 .
  \end{align*}
  If we did not assume \ref{ass:A3}, then the nonlinearities could also be strictly negative for all $a>0$, for example, $V(a)= -a^4-a^6$ obeys \ref{ass:A1} and because of
  \begin{align*}
  	V'(a)a = -4a^4 -6 a^6 = 6(-\frac{4}{6}a^4 - a^6) \ge 6 V(a)
  \end{align*}
  also \ref{ass:A2}, but then the critical threshold $\lambda^\dav_{\mathrm{cr}}$ given in Theorems \ref{thm:existence+} and \ref{thm:existence0} would be infinite. The threshold is finite once there exists $f\in H^1(\R) $ with $N(f)>0$, see 
  Theorem \ref{thm:threshold-phenomena}.\ref{thm:threshold-phenomena-5}.
 \end{theoremlist}
\end{remarks}

Our first main result is

\begin{theorem}[Thresholds for existence for positive average dispersion]\label{thm:existence+}
  Assume $\dav>0$, $V$ obeys assumptions \ref{ass:A1} through \ref{ass:A3} for some  $2< \gamma_1\le \gamma_2<10$, and $\psi\in L^{\alpha_\delta}$ has compact support for some $\delta>0$, where $\alpha_\delta:=\alpha_\delta(\gamma_2):=\max\{1,\f{4}{10-\gamma_2}+\delta\}$. Then
  \begin{theoremlist}
  	\item	
		There exists a threshold $0\le\lambda^\dav_\mathrm{cr}<\infty$ such that $E^{\dav}_\lambda =0$ for $0< \lambda < \lambda^\dav_{\mathrm{cr}}$ and $-\infty<E^{\dav}_\lambda<0$ for  $\lambda> \lambda^\dav_{\mathrm{cr}}$.
	\item
		If $0<\lambda<\lambda_{\mathrm{cr}}^\dav$, then no minimizer for the constrained  minimization problem  \eqref{eq:min} exists.
		If $\gamma_1\ge 6$, then $\lambda_{\mathrm{cr}}^\dav>0$.
	\item
		If $\lambda>\lambda_{\mathrm{cr}}^\dav$, then any minimizing sequence for  \eqref{eq:min} is up to translations relatively compact in $L^2(\R)$, in particular,  there exists a minimizer for \eqref{eq:min}.
		This minimizer is also a weak solution of the dispersion management equation \eqref{eq:GT} for some Lagrange multiplier $\omega< 2E^{\dav}_\lambda/\lambda<0$.
	\item
		If $V$ obeys in addition \ref{ass:A4},  then $\lambda_{\mathrm{cr}}^\dav=0$.
  \end{theoremlist}
\end{theorem}
\begin{remark}
	If $\gamma_2<6$ then $\f{4}{10-\gamma_2}<1$ and thus $\alpha_\delta=1$ for all small enough $\delta>0$. So if $2<\gamma_1\le \gamma_2<6$ we only need $\psi\in L^1$ with compact support for Theorem \ref{thm:existence+} to hold.

	If $6\le \gamma_2<10$ the density  $\psi$ has to have  compact support and to be in $L^p$ with $p$ slightly larger than $\f{4}{10-\gamma_2}$. 	With the bound \eqref{eq:L-p-bound-psi} this translates into some slightly more restrictive integrability bound
	on $1/d_0$, which still covers all physically relevant local dispersion profiles.
\end{remark}

We have a similar existence result in the case of $d_{\text{av}}=0 $ under slightly different $L^p$ assumptions on the density $\psi$.

\begin{theorem}[Threshold for existence for zero average dispersion] \label{thm:existence0}
Assume $\dav=0$ and $V$ obeys assumptions \ref{ass:A1} through \ref{ass:A3} with $2< \gamma_1\le \gamma_2< 6$, and that  the density $\psi$ has compact support and
	$\psi\in L^{\f{4}{6-\gamma_2}+\delta}$ for some  $\delta>0$.
  \begin{theoremlist}
	\item 	
		There exists a threshold $0\le\lambda^0_\mathrm{cr}<\infty$ such that $E^{0}_\lambda =0$ for $0< \lambda < \lambda^0_{\mathrm{cr}}$ and $-\infty<E^{0}_\lambda<0$ for  $\lambda> \lambda^0_{\mathrm{cr}}$.
	\item
		If $\lambda>\lambda^0_{\mathrm{cr}}$, then any minimizing sequence for  \eqref{eq:min} is up to translations and boosts, that is, translations in Fourier space, relatively compact in $L^2(\R)$, in particular,  there exists a minimizer for \eqref{eq:min}.
		This minimizer is also a weak solution of the dispersion management equation \eqref{eq:GT} for some Lagrange multiplier $\omega< 2E^{0}_\lambda/\lambda<0$.
	\item
		If $V$ obeys in addition \ref{ass:A4},  then $\lambda^0_{\mathrm{cr}}=0$.
  \end{theoremlist}	
\end{theorem}

\begin{remarks}\label{rem:nonlinearity}
 \begin{theoremlist}
  \item Concerning the existence, Theorems \ref{thm:existence+} and \ref{thm:existence0} are sharp, see Theorem \ref{thm:non-existence} below.
  \item  In the application to nonlinear optics, $\psi$ is the density of the probability measure $\mu$ in \eqref{eq:GT-time-version-2} below, which in turn is naturally related to the dispersion profile $d_0$ in dispersion management cables, see the discussion just above Lemma \ref{lem:properties-mu}. It turns out that this probability measure always has compact support as soon as $d_0$ is integrable over the period $[0,L]$ and it has a density
  $\psi$ once the zero set of the local dispersion profile $d_0$ has zero Lebesgue measure, that is,
	$|\{s\in [0,L]: d_0(s)=0\}| =0$. This is a reasonable assumption on $d_0$ since otherwise $\psi$ would have some delta distribution component and the nonlinearity	 $N$ would be rather singular.

    The so far most studied case is the model case where
    \begin{align*}
    d_0= \id_{[0,1)} - \id_{[1,2)}	\, ,
    \end{align*}
	that is, two pieces of glass fiber cables with exactly opposite dispersion are concatenated together and this is repeated periodically with period $2$. In this case $\psi= \id_{[0,1]}$. Our existence results are valid for a considerably larger class of probability densities $\psi$, and thus a \emph{very large} class of dispersion profiles $d_0$. The $L^p$ conditions on $\psi$ in Theorems \ref{thm:existence+} and \ref{thm:existence0} translate to conditions on
    the dispersion profile $d_0^{1-p}$ via Lemma \ref{lem:properties-mu} below.

    In particular,  for $\dav>0$ and $\gamma_2<6$, we can allow for the
    \emph{largest possible} class of local dispersion profiles $d_0$, they only have to change sign finitely many times and their zero set has to have Lebesgue measure zero.

    Even in the case of a Kerr nonlinearity, where $V(a)\sim a^4$, i.e., $\gamma_1=\gamma_2=4$, the above two theorems strongly improve our result in \cite{HuLee2012} in terms of scales of $L^p$ spaces: In \cite{HuLee2012}, we needed in addition that $\psi\in L^4$, whereas now with $\gamma_2=4$, one sees that $\psi\in L^1$ is enough for strictly positive average dispersion and for vanishing average dispersion we only need $L^{2+\delta} $ for arbitrarily small $\delta>0$.
 \item For the Kerr nonlinearity, the smoothness and decay of the minimizers has been studied in \cite{EHL2009} and \cite{HuLee2009} for the simplest case of an  alternating dispersion profile given by $d_0= \id_{[0,1)} - \id_{[1,2)}$ and extended to more general dispersion profiles in \cite{GH}. In the more general setting discussed in this paper the smoothness and decay of solitary solutions is an open problem.
 \end{theoremlist}   	
\end{remarks}

Concerning the question whether the range of exponents in Theorems \ref{thm:existence+} and \ref{thm:existence0} is optimal, we note the following
\begin{theorem}\label{thm:non-existence}
  	Let $V$ be a pure power law nonlinearity given by $V(a)=ca^\gamma$ for $a\ge 0$
  	and some coefficients $c,\gamma>0$.
  \begin{theoremlist}
  	\item Assume further that $\dav>0$ and $\psi$ is a probability density with compact support which is strictly positive in a possibly one sided neighborhood of zero.
  	Then $H(f)$ is unbounded from below on $H^1$ for any fixed
  	$\|f\|^2=\lambda>0$, if $\gamma>10$.
  If $\gamma=10$, then  $H(f)$ is unbounded from below for any fixed
  	$\|f\|^2=\lambda>0$  as long as $c$ is large enough.
  \item
    If $\dav=0$ and $\gamma>6$, then $H(f)$ is unbounded from below on $L^2$ for 
    any fixed
  	$\|f\|^2=\lambda>0$, if $\psi$ is a probability density with compact support which is strictly positive in a possibly one sided neighborhood of zero. \\ 
  	If $\dav=0$ and $\gamma=6$, then no minimizers exist in the model case $\psi=\id_{[0,1]}$. 	
  \end{theoremlist}
\end{theorem}
\begin{remark}
	As the lower bound \eqref{eq:psi-lower} from Remark \ref{rem:psi-lower-bound} shows, the assumption of the first part of Theorem \ref{thm:non-existence} is fulfilled if the dispersion profile $d_0$ is bounded from above close to zero, which includes all physically relevant dispersion profiles.
\end{remark}

 The strategy of the proofs of our Existence Theorems \ref{thm:existence+} and \ref{thm:existence0} is as follows: The main observation which shows that $E^\dav_\lambda<0$ is \emph{equivalent} to gaining compactness is done in Theorem \ref{thm:existence}.
 The necessary space--time bounds which prevent splitting of minimizing sequences as soon as $E^\dav_\lambda<0$ are done in Section \ref{subsec:fractional-linear} and their consequences for the nonlinear and nonlocal potential in Section \ref{subsec:Vsplitting}. The main building blocks, for which one has to develop suitable space-time bounds, turn out to be of the form given in Definition \ref{def:building block} and are motivated by Lemma \ref{lem:V-splitting}.
 Our proofs for strictly positive average dispersion rely on some very useful space-time bounds for coherent states, see Lemma \ref{lem:gaussian coherent states}, which are new and proven in Appendix \ref{sec:Galilei}.
 Strict subadditivity of the energy is done in Section \ref{sec:concavity} and the necessary tightness bound, modulo the symmetries of the problem are established in Section \ref{sec:existence}.
 That there exists a threshold which distinguishes between $E^\dav_\lambda=0$ and
 $E^\dav_\lambda <0$ is the content of Theorem \ref{thm:threshold-phenomena}
 . At the end of Section \ref{sec:threshold} the proofs of Theorems \ref{thm:existence+} and  \ref{thm:existence0} are given.
 Theorem \ref{thm:non-existence} is proven in Section \ref{sec:non-existence}.

\subsection{The connection with nonlinear optics}\label{sec:connection-nonlinear-optics}

Our main motivation for studying \eqref{eq:min}
comes from the fact that the minimizer of the variational problem is related to
breather-type solutions of the dispersion managed nonlinear Schr\"odinger equation
 \beq\label{eq:NLS-gen}
      i \partial_t u = -d(t) \partial_x^2 u - p(|u|)u,
 \eeq
where the dispersion $d(t)$ is parametrically modulated and $P(u)=p(|u|)u$ is the nonlinear interaction due to the polarizability of the glass fiber cable.
In nonlinear optics \eqref{eq:NLS-gen} describes the evolution of a pulse in a frame moving
with the group velocity of the signal through a glass fiber cable, see \cite{SulemSulem}.
As a \emph{warning}: with our choice of notation the variable $t$ denotes the position along the glass fiber cable  and $x$ the (retarded) time. Hence $d(t)$ is \emph{not varying in time} but  denotes indeed a dispersion \emph{varying along} the optical cable. For physical reasons it would not be a strong restriction to assume that $d$ is piecewise constant, but we will not make this assumption in this paper.
By symmetry, one assumes that $P$ is odd and $P(0)=0$ can always be enforced by adding a constant term. Most often one makes a Taylor series expansion, keeping just the lowest order nontrivial term leads to $P(u)\sim |u|^2u$, the Kerr nonlinearity, but we will not make this approximation.

The dispersion management idea, i.e., the possibility to periodically manage
the dispersion by putting alternating sections with positive and negative dispersion together in an
optical glass-fiber cable to compensate for dispersion of the signal was predicted by Lin, Kogelnik,
and Cohen already in 1980, see \cite{LKC80}, and then implemented by Chraplyvy and Tkach for which they
received the Marconi prize in 2009. See the reviews \cite{TDNMSF99, Tetal03} and the references cited in \cite{HuLee2012} for a discussion of the dispersion management technique.

The periodic modulation of the dispersion can be modeled by the ansatz
 \beq\label{eq:dispersion}
    d(t) = \veps^{-1} d_0(t/\veps) + \dav .
 \eeq
Here  $\dav \ge 0$ is the average component and $d_0$ its mean zero part which we assume to have period $L$.
For small $\veps$, equation \eqref{eq:dispersion} describes a fast strongly varying dispersion which corresponds to  the regime of \emph{strong} dispersion management.

A technical complication is the fact that \eqref{eq:NLS-gen} is a non-autonomous equation. We seek to rewrite \eqref{eq:NLS-gen} into a more convenient form in order to find breather type solutions.
Let  $D(t)= \int_{0}^t d_0(r)\, dr$ and note that as long as  $d_0$ is locally
integrable and has period $L$ with mean zero, $D$ is also periodic with period $L$.
Furthermore,  $T_r= e^{ir\partial_x^2}$ is a unitary operator and thus the
unitary family $t\mapsto T_{D(t/\veps)}$ is periodic with period $\veps L$.
Making the ansatz $u(t,x)= (T_{D(t/\veps)}v(t,\cdot))(x)$ in \eqref{eq:NLS-gen}, a short calculation
shows
    \bdm
        i\partial_t v= -d_{\text{av}}\partial_x^2 v - T_{D(t/\veps)}^{-1}\big[P(T_{D(t/\veps)} v)\big]
    \edm
which is equivalent to \eqref{eq:NLS-gen} and still a non-autonomous equation.

For small $\veps$, that is, in the regime of strong dispersion
management, $T_{D(t/\veps)}$ is fast oscillating in the variable $t$, hence the
solution $v$ is expected to evolve on two widely separated time-scales, a slowly evolving part
$v_{\text{slow}}$ and a fast, oscillating part with a small amplitude.
Analogously to Kapitza's treatment of the unstable pendulum which is stabilized by fast
oscillations of the pivot, see \cite{LandauLifshitz}, the effective equation for the slow part
$v_{\text{slow}}$ was derived by Gabitov and Turitsyn \cite{GT96a,GT96b} for the special case of a Kerr nonlinearity. It is given by integrating the fast
oscillating term containing $T_{D(t/\veps)}$ over one period in $t$,
    \beq
    \begin{split}\label{eq:GT-time-version-1}
        i\partial_t v_{\text{slow}}
        &= - d_{\text{av}} \partial_x^2  v_{\text{slow}}
            - \frac{1}{\veps L}\int_{0}^{\veps L}
                T_{D(r/\veps)}^{-1}\big[P(T_{D(r/\veps)} v_{\text{slow}})\big]
              \, dr \\
        &= - d_{\text{av}} \partial_x^2  v_{\text{slow}}
            - \frac{1}{L}\int_{0}^{L}
                T_{D(r)}^{-1}\big[P(T_{D(r)} v_{\text{slow}})\big]
              \, dr .
    \end{split}
    \eeq
This averaging procedure leading to \eqref{eq:GT-time-version-1} was rigorously justified in \cite{ZGJT01}
for suitable dispersion profiles $d_0$ in the case of a Kerr nonlinearity.
The averaged equation is autonomous and stationary solutions of \eqref{eq:GT-time-version-1} can be found by making the ansatz
 \beq\label{eq:stationary}
    v_{\text{slow}}(t,x) = e^{-i\omega t} f(x) .
 \eeq
Before doing so, it turns out to be advantageous to rewrite the nonlocal nonlinear term in \eqref{eq:GT-time-version-1}: Define a measure $\mu(B)$ by setting $\mu(B)\coloneqq \tfrac{1}{L}\int_{0}^L \id_B(D(r))\, dr$
for any Lebesgue measurable set $B\subset \R$. Since $\mu(B)\ge 0$ and
$\mu(\R) = \tfrac{1}{L}\int_0^L \id_\R(D(r))\, dr = \tfrac{1}{L}\int_0^L dr=1$,
one sees that $\mu$ is a probability measure. Since $\mu$ is the image measure of normalized Lebesgue measure on $[0,L]$ under $D$, we can rewrite \eqref{eq:GT-time-version-1} as
    \beq\label{eq:GT-time-version-2}
        i\partial_t v_{\text{slow}}
        = - d_{\text{av}} \partial_x^2  v_{\text{slow}}
            - \int_\R
                T_{r}^{-1}\big[P(T_{r} v_{\text{slow}})\big]
              \, \mu(dr) .
    \eeq
One can easily check that the simplest case of dispersion management, $L=2$, 
$d_0= \id_{[0,1)}-\id_{[1,2)}$, which is
the case most studied in the literature, corresponds to the measure $\mu$ having density $\id_{[0,1]}$, the uniform distribution on $[0,1]$, see formula \eqref{eq:density-psi}.
For the general case, we gather some basic properties of the probability measure $\mu$ in the following

\begin{lemma}[Lemma 1.4 in \cite{HuLee2012}]\label{lem:properties-mu} Assume that the dispersion profile $d_0$ is locally integrable. Then the following holds.
\begin{theoremlist}
 \item  The probability measure $\mu$ has compact support.
 \item If the set $\{d_0=0\}$ has zero Lebesgue measure, then $\mu$ is absolutely continuous with respect to Lebesgue measure.
\item If furthermore $d_0$ changes sign finitely many times on $[0,L]$ and is bounded away from zero then $\mu$ has a bounded density $\psi$.
\item Moreover, if $d_0$ changes sign finitely many times on $[0,L]$ and for some $p>1$
 \bdm
  \int_{0}^L |d_0(s)|^{1-p}\, ds <\infty,
 \edm
 then $\mu$ has a density $\psi\in L^p$. More precisely, we have the bound
  \beq\label{eq:L-p-bound-psi}
  \|\psi\|_{L^p}
  \lesssim
  \Big(\int_{0}^L |d_0(s)|^{1-p}\, ds\Big)^\frac{1}{p}
 \eeq
 where the implicit constant depends only on the number of sign changes of $d_0$ and the period $L$.
\end{theoremlist}
\end{lemma}
\begin{remarks}\label{rem:psi-lower}
 \begin{theoremlist}
  \item  As explained in \cite{HuLee2012}, the bound \eqref{eq:L-p-bound-psi} is quite natural and sharp. It translates $L^p$ restrictions on $\psi$ into integrability conditions on $d_0^{1-p}$. The extreme case $p=\infty$ yields that $\psi$ is bounded once $d_0$ is bounded away from zero, and the case $p=1$ poses the weak additional restriction that the set where $d_0$ is zero has Lebesgue measure zero.
\item\label{rem:psi-lower-bound} Without working too hard, one can derive a formula for the density $\psi$
	of the probability measure $\mu$. We give the short argument from \cite{HuLee2012},
	for the reader's convenience, since it has an important consequence for
	Theorem \ref{thm:non-existence}: We assume that the measure of the set $\{d_0=0\}$
	is zero, otherwise $\mu$ will have a Dirac point mass component, and that $d_0$
	changes sign only finitely many times on $[0,L]$. Let $\{A_j\}$ be a collection of
	disjoint half--open intervals with $\cup_j A_j= [0,L)$ such that, on each $A_j$, the dispersion profile $d_0$ has a
	fixed sign and so $D$ is strictly monotone.
Then  by the definition of $\mu$,
 \begin{align*}
  \int F(r)\, &\mu(dr)
  =
  \sum_j \frac{1}{L}\int_{A_j} F(D(s))\, ds
  =
  \sum_j \frac{1}{L}\int_{A_j} F(D(s)) |D'(s)| |D'(s)|^{-1} \,ds \\
 &=
  \sum_j \frac{1}{L}\int_{D(A_j)} F(r) \frac{1}{|d_0(D^{-1}(r))|}\, dr
  =
  \int\limits_{\supp(\mu)} F(r) \frac{1}{L}\sum_{s\in D^{-1}(\{r\})} |d_0(s)|^{-1}\, dr.
 \end{align*}
 In the third equality we used a change of variables $r= D(s)$ and that in each $A_j$ there is a unique $s_j\in A_j$ such that $D(s_j)=r$ and for the last equality we set $D^{-1}(\{r\})=\{s\in[0,L)\vert\, D(s)= r \}$, the set of pre-images of $r$ within $[0,L)$.
 Thus we have the formula
 \beq\label{eq:density-psi}
  \psi(r)=\frac{1}{L}\sum_{s\in D^{-1}(\{r\})} |d_0(s)|^{-1}
 \eeq
 for the density $\psi$ of $\mu$ in terms of the dispersion profile $d_0$. Since $D(r)= \int_0^r d_0(s)\, ds$ and $d_0$ is locally integrable, $D$ is continuous and we
 can use \eqref{eq:density-psi} to
 get a lower bound for  all $r$ in the support
 of $\psi$ close enough to zero, as long as  $d_0$ does not behave too wildly:
 If $d_0$ is bounded close to zero there exists $r_0>0$ such that with
 $m\coloneqq L^{-1}\inf\{|d_0(D(r))|^{-1}: 0\le r\le r_0\}>0$ one has the lower bound
 \begin{align}\label{eq:psi-lower}
   \psi \ge m\id_{[0,r_0]} \text{ or }  \psi \ge m\id_{[-r_0,0]} 	
 \end{align}
 which one of the above two lower bounds holds depends on the sign of $d_0$ close to zero.
 \end{theoremlist}
\end{remarks}

Coming back to our discussion of the dispersion management equation,
plugging \eqref{eq:stationary} into \eqref{eq:GT-time-version-2}, we see that
$f$ should solve
 \beq\label{eq:GT}
    \omega f
    = -d_{\text{av}}  f'' - \int_\R
                T_{r}^{-1}\big[P(T_{r} f)\big]
              \, \mu(dr)   ,
 \eeq
which is a nonlocal nonlinear eigenvalue equation for $f$. Testing \eqref{eq:GT} with suitable test functions $h$ one gets the weak formulation
 \bdm
    \omega \langle h, f\rangle
    =
    d_\text{av}\langle h', f' \rangle - \la h, \int_\R
                T_{r}^{-1}\big[P(T_{r} f)\big]
              \, \mu(dr) \ra
 \edm
where $\la h_1,h_2 \ra$ is the scalar product on $L^2(\R)$ given by
$\int_\R \ol{h_1(x)} h_2(x)\, dx$.
Exchanging integrals, a formal calculation, using the unicity of $T_r$, yields
    \begin{align*}
        \la h, \int_\R
                T_{r}^{-1}\big[P(T_{r} f)\big]
              \, \mu(dr )\ra
        = \int_\R \la T_r h , P(T_{r} f)\ra  \, \mu(dr)
    \end{align*}
and one arrives at the weak formulation of \eqref{eq:GT} in the form
    \beq\label{eq:GT-weak}
        \omega \langle h, f\rangle
        =
        d_\text{av}\langle h', f' \rangle - \int_\R \la T_r h , P(T_{r} f)\ra  \, \mu(dr)  ,
    \eeq
supposed to hold for any $h$ in the Sobolev
space $H^1(\R)$.

Using the formula from Lemma \ref{lem:differentiability} for the derivative of the nonlocal nonlinearity $N(f)$ from \eqref{eq:N(f)}, one sees that \eqref{eq:GT-weak} is the weak form of the Euler-\-Lagrange equation associated to the energy $H(f)$ given in \eqref{eq:energy}, as long as $V'(|T_rf|) \sgn{T_rf} = P(T_rf)$. This is the case if
 \begin{align*}
	V'(a) = p(a)a = P(a)\quad \text{for all } a>0,
 \end{align*}
 i.e., $V$ is the antiderivative of the polarizability $P$,
 \begin{align*}
 	V(a)\coloneqq \int_0^a P(s)\, ds .
 \end{align*}
In this case it is, up to some technicalities, clear that any minimizer of the associated constrained minimization problem
\eqref{eq:min} will be a weak solution of \eqref{eq:GT} for some choice of Lagrange multiplier $\omega$, as long as the variational problem \eqref{eq:min} admits minimizers.
In particular,  combining Theorems \ref{thm:existence+} and \ref{thm:existence0} with Lemma \ref{lem:properties-mu} one sees that \eqref{eq:GT} has a non trivial weak solution $f$  under the condition that the assumptions \ref{ass:A1}--\ref{ass:A3} hold, at least
for large enough power $\lambda=\|f\|^2$, or for arbitrary power, if additionally \ref{ass:A4}  holds  for the antiderivative of $P$ and that the dispersion profile $d_0$ changes signs finitely many times and $1/d_0$ obeys some mild integrability conditions given by the right hand side of \eqref{eq:L-p-bound-psi}.
This allows for a large class of dispersion profiles $d_0$,  covering all physically relevant cases.

\section{Nonlinear estimates} \label{sec:nonlinear-estimates}
\subsection{Fractional Bilinear Estimates}\label{subsec:fractional-linear}
In this paper, the nonlocal nonlinearity is not a pure power, thus the multilinear estimates from \cite{HuLee2012} cannot be used anymore.
First, we gather the estimates which will be used in the proof of fat--tail Propositions \ref{prop:fat-tail-bound+} and \ref{prop:fat-tail-bound0}, which are crucial for the existence proof in this paper.
The core of the argument will be suitable splitting bounds on the nonlocal nonlinearity $N(f)$ from \eqref{eq:N(f)} given in Proposition \ref{prop:N-splitting}. For this, inspired by the splitting Lemma \ref{lem:V-splitting} for $V$, one needs  certain \emph{fractional linear} bounds on the building blocks from Definition \ref{def:building block}.
\\
Since $T_r=e^{ir\partial_x^2}$ is the solution operator for the free Schr\"{o}dinger equation in dimension one, we can express $T_rf$ for any nice $f$, for example, in the Schwartz class, as follows:
\begin{align}
T_rf(x)
&=\f{1}{\sqrt{4\pi i r }}\int_{\R} e^{i\f{|x-y|^2}{4r}}f(y)dy \label{eq:explicit form of the sol}\\
&=\f{1}{\sqrt{2\pi}}\int_\R e^{ix\eta}e^{-ir\eta^2}\hat{f}(\eta)d\eta,\label{eq:Fourier representation}
\end{align}
where $\hat{f}$ is the Fourier transform of $f$ given by
$$
\hat{f}(\eta)=\f{1}{\sqrt{2\pi}}\int_\R e^{-ix\eta} f(x)dx.
$$

As a first step, we note that, for $\psi$ in suitable $L^p$ spaces, certain space time norms of $T_rf$ are bounded.
\begin{lemma}\label{lem:boundedness}
Let $f\in L^2(\R)$, $2\le q\le 6$ and $\psi\in L^{\f{4}{6-q}}(\R)$. Then
\beq\label{eq:boundedness}
	\|T_r f\|_{L^{q}(\R^2,dx\psi dr)}
	\lesssim
		\|\psi\|_{L^{\f{4}{6-q}}(\R)}\|f\|.
\eeq
\end{lemma}
\begin{proof}
Using the H\"older inequality, we get
\begin{align*}
	\iint_{\R^2}|T_r f |^{q}dx \psi dr
	&=
		\iint_{\R^2} \left( |T_r f|^{\frac{2(6-q)}{4}}\psi \right) \left(  |T_r f|^{\frac{6(q-2)}{4}}\right) dxdr\notag\\
	&\leq \left(\iint_{\R^2}|T_r f|^2\psi^{\f{4}{6-q}}\,dxdr\right)^{\f{6-q}{4}}\left(\iint_{\R^2}|T_r f|^6\,dxdr\right)^{\f{q-2}{4}}.
\end{align*}
Since $T_r$ is unitary on $L^2(\R)$,
\begin{align*}
	\iint_{\R^2}|T_r f|^2\psi^{\f{4}{6-q}}\,dxdr = \|f\|^2 \int_\R \psi^{\f{4}{6-q}}\,dr
\end{align*}
and the Strichartz inequality \cite{GV85,Strichartz}, needed here only in one dimension, gives
\begin{align*}
	\iint_{\R^2}|T_r f|^6\,dxdr \le S_1^6 \|f\|^6
\end{align*}
and so \eqref{eq:boundedness} follows.
\end{proof}
\begin{remark}
	The sharp value of the constant in the one-dimensional Strichartz inequality
	is known to be $S_1=12^{-1/12}$, the two dimensional sharp constant is known, too, and it is also known that Gaussians are the only maximizers in the Strichartz inequality in one and two space dimensions, see \cite{Foschi} and \cite{HZ06}.
	In recent years there has been a renewed interest in establishing existence of maximizers for certain space time norms of solutions of linear evolution equations, like the Schr\"odinger or wave equation, see, for example,
	\cite{Christ-Shao-1, Christ-Shao-2, Fanelli-Vega-Visciglia, HuShao}.
\end{remark}
To take advantage of the fact that an interaction term containing the product of two terms of the form $T_rf_1$ and $T_rf_2$ is typically small if the functions $\hat{f_1}$ and $\hat{f_2}$ have separated supports, we need

\begin{lemma}[Fractional bilinear estimate] \label{lem:fractional-bilinear}
 Let $2 \leq p< 3$ and $f_1,\ f_2 \in L^2(\R)$ whose Fourier transforms have separated supports, say
 $s= \dist (\supp \hat{f_1}, \supp  \hat{f_2})>0$.
 Then
 \beq \label{eq:fractional-bilinear}
   \|T_r f_1 T_r f_2\|_{L^p(\R^2,\,dx dr)} \lesssim \f{1}{s^{(3-p)/p}} \|f_1\| \|f_2\|.
 \eeq
\end{lemma}

\begin{remark}
	The bound \eqref{eq:fractional-bilinear} is a well-known bilinear
	estimate for $p=2$, see \cite{Bourgain}. For readers' convenience,
	we give a proof of \eqref{eq:fractional-bilinear} for any $2\le p < 3$.
	As the proof shows, \eqref{eq:fractional-bilinear} holds also
	for $p=3$, without any support condition on $\hat{f_1}$ and $\hat{f_2}$.
\end{remark}

\begin{proof}
Using \eqref{eq:Fourier representation}, we get
\begin{align*}
	T_r f_1(x) T_r f_2(x)
	= \f{1}{2\pi}\iint_{\R^2}e^{ix(\eta_1+\eta_2)-ir(\eta_1^2+\eta_2^2)}
	  \hat{f_1}(\eta_1)\hat{f_2}(\eta_2)d\eta_1d\eta_2 .
\end{align*}
Doing the change of variables $a=\eta_1+\eta_2$, $b=\eta_1^2+\eta_2^2$, with Jacobian $J=\f{\partial(a,b)}{\partial(\eta_1,\eta_2)}=2(\eta_2-\eta_1)$ and introducing
\begin{align*}
	F(a,b)\coloneqq \frac{1}{|J|}\hat{f_1}(\eta_1(a,b))\hat{f_2}(\eta_2(a,b))\id_{[0,\infty)}(b)
\end{align*}
one sees
\begin{align*}
	T_r f_1(x) T_r f_2(x)
	= \f{1}{2\pi}\iint_{\R^2}e^{ixa-irb}
	  F(a,b)\, dadb ,
\end{align*}
that is, up to sign in one of the variables, $T_r f_1(x) T_r f_2(x)$ is the space-time Fourier transform of $F$. Since $p\ge 2$, one can apply the Hausdorff-Young inequality, see,  e.g., \cite{LiebLoss}, which reduces to Plancherel's identity for $p=2$, to get
\begin{align*}
  \|T_r f_1 T_r f_2\|_{L^p(\R\times\R,dxdr)}
  &\le
    \|F\|_{L^{p'}(\R^2,dadb)}
\end{align*}
with $p'$ the dual index to $p$. Undoing the above change of variables, one sees
\begin{align}\label{eq:great}
    \|F\|_{L^{p'}(\R^2,dadb)}
   =2^{-1/p}\left(\iint_{\R^2}\f{1}{|\eta_2-\eta_1|^{p'-1}}|\hat{f_1}(\eta_1) \hat{f_2}(\eta_2)|^{p'}\,d\eta_1d\eta_2\right)^{1/p'}.
\end{align}
If $p=p'=2$, we use $|\eta_2-\eta_1|\ge s$ on the support of the product $\hat{f_1}\hat{f_2}$ to get
\begin{align*}
  	\|F\|_{L^{2}(\R^2,dadb)} \lesssim \frac{1}{\sqrt{s}} \|\hat{f_1}\|\|\hat{f_2}\|
\end{align*}
which concludes the proof for $p=2$, since the Fourier transform is an isometry on $L^2$.

Since $3/2 < p'<2$, one can use the Hardy-Littlewood-Sobolev inequality to see
\begin{align*}
\eqref{eq:great} &\leq \f{1}{s^{2-3/p'}}\left(\iint_{\R^2}\f{|\hat{f_1}(\eta_1)|^{p'} |\hat{f_2}(\eta_2)|^{p'}}{|\eta_2-\eta_1|^{2-p'}}d\eta_1d\eta_2\right)^{\f{1}{p'}}\\
&\lesssim
\f{1}{s^{(3-p)/p}}\|\hat{f_1}\| \| \hat{f_2}\|
\end{align*}
which yields \eqref{eq:fractional-bilinear} for $2<p<3$.
\end{proof}


The following will be the building blocks for our bounds on the nonlocal nonlinear potential, see \eqref{eq:V-splitting-2}.
\begin{definition}\label{def:building block}
For any $\gamma \ge 2$, define
	\begin{align*}
		M_\psi^\gamma(f_1,f_2)\coloneqq
		\iint_{\R^2} |T_rf_1||T_rf_2|(|T_rf_1|+|T_rf_2|)^{\gamma-2} \, dx \psi dr .
	\end{align*}
\end{definition}
\begin{remark}
 At first $M_\psi^\gamma(f_1,f_2)$ is defined only when $f_1,f_2$ are Schwartz functions. Using Proposition \ref{prop:M-bounded} below one sees that for all
 $\gamma\ge 2$ and $\psi\in L^1$ one can extend $M_\psi^\gamma(f_1,f_2)$ to all of $H^1$, and even to all of $L^2$ if $2\le \gamma\le 6$ and $\psi$ in certain $L^p$ spaces,  by density of the Schwartz functions. 	
\end{remark}®

\begin{proposition}\label{prop:M-bounded}
  \begin{theoremlist}
  	\item
		Let $2\le \gamma\le 6$ and $\psi\in L^{\f{4}{6-\gamma}}$. Then
		\begin{align}\label{eq:M-bounded-1}
			M_\psi^\gamma(f_1,f_2)
			\lesssim \|f_1\|\|f_2\|(\|f_1\|+\|f_2\|)^{\gamma-2}
	\end{align}
	where the implicit constant depends only on $\gamma$ and the
	$L^{\f{4}{6-\gamma}}$ norm of $\psi$.
\item Let $2\le\gamma<\infty$ and $\psi\in L^1$. Then
\begin{align}\label{eq:M-bounded-2}
	M_\psi^\gamma(f_1,f_2)
	\lesssim \|f_1\|\|f_2\|(\|f_1\|_{H^1}+\|f_2\|_{H^1})^{\gamma-2}
\end{align}
where the implicit constant depends only on $\gamma$ and the $L^{1}$ norm of $\psi$.
  \end{theoremlist}
\end{proposition}
\begin{proof}
	Using H\"older's inequality for 3 functions with exponents $\gamma$, $\gamma$, and $\gamma/(\gamma-2)$ one has
	\begin{align*}
		M_\psi^\gamma(f_1,f_2)\le
		 \|T_rf_1\|_{L^\gamma(\R^2, dx\psi dr)}
		 \|T_rf_2\|_{L^\gamma(\R^2, dx\psi dr)}
		 \||T_rf_1|+|T_rf_2|\|_{L^\gamma(\R^2, dx\psi dr)}^{\gamma-2}.
	\end{align*}
	Applying the triangle inequality and Lemma \ref{lem:boundedness} completes the proof of \eqref{eq:M-bounded-1}.
	
	Similarly, using H\"older's inequality with exponents $2,2$, and $\infty$ shows
	\begin{align}
		M_\psi^\gamma(f_1,f_2)
		&\le
		 \|T_rf_1\|_{L^2(\R^2, dx\psi dr)}
		 \|T_rf_2\|_{L^2(\R^2, dx\psi dr)}
		 \sup_{r\in\R }\left(\|T_rf_1\|_{L^\infty}+\|T_rf_2\|_{L^\infty}\right)^{\gamma-2}
		 \|\psi\|_{L^1}^{\gamma-2} \nonumber\\
		&\le
		 \|\psi\|_{L^1}^\gamma \|f_1\|\|f_2\| \sup_{r\in\R }\left(\|T_rf_1\|_{L^\infty}+\|T_rf_2\|_{L^\infty}\right)^{\gamma-2} \label{eq:Linfty}
	\end{align}
	where we also used Lemma \ref{lem:boundedness}. Using the simple bound
\begin{align}\label{eq:Kato}
	\|h\|_{L^\infty}^2 \le \|h\| \|h'\|,
\end{align}
whose proof we postpone to the end of this proof, and the fact that the derivative and $T_r$ commute and $T_r$ is unitary on $L^2(\R)$, one gets
\begin{align*}
  \sup_{r\in\R }\|T_rf_1\|_{L^\infty}
	\le (\|f_1\| \|f_1'\|)^{1/2} \le \|f_1\|_{H^1}
\end{align*}
and similar for $T_rf_2$. Thus the second factor in \eqref{eq:Linfty} is bounded by
\begin{align}\label{eq:nice2}
	\sup_{r\in\R }(\|T_rf_1\|_{L^\infty}+\|T_rf_2\|_{L^{\infty}})^{\gamma-2}
	\le (\|f_1\|_{H^1} + \|f_2\|_{H^1})^{\gamma-2}
\end{align}
which finishes the proof of \eqref{eq:M-bounded-2}.

It remains to give an argument for \eqref{eq:Kato}. This is well--known, but we give the short proof for convenience of the reader. It is enough to assume that $h\in C^\infty_0(\R)$. Then
\begin{align*}
  |h(x)|^2	=\int_{-\infty}^x 2\re(\ol{h(s)}h'(s))\, ds = - \int_{x}^{\infty} 2\re(\ol{h(s)}h'(s))\, ds .
\end{align*}
So
\begin{align*}
	|h(x)|^2 \le \int_{\R } \left|h(s)h'(s)\right|\, ds
	\le \|h\|\|h'\|
\end{align*}
using the Cauchy-Schwarz inequality.
\end{proof}

\begin{proposition}\label{prop:M-fourier} Let $s= \dist (\supp \hat{f_1}, \supp  \hat{f_2})>0$. \\
If $2<\gamma< 6$, $\tau > 1$ and $\psi\in L^{\beta(\gamma, \tau)}$, then
 	\begin{align*}
	M_\psi^\gamma(f_1,f_2)
	\lesssim s^{-\alpha(\gamma, \tau)}\|f_1\|\|f_2\|(\|f_1\|+\|f_2\|)^{\gamma-2},
	\end{align*}
where $\alpha(\gamma, \tau) \coloneqq \min \{\f{\gamma-2}{6\tau}, \f{6-\gamma}{2\tau}\}$ and $\beta(\gamma, \tau) \coloneqq \f{4}{6-\gamma -2 \alpha(\gamma, \tau)}$.
\end{proposition}
\begin{remark}
 Note that $\beta(\gamma, \tau)$ is only slightly bigger than
 $\f{4}{6-\gamma}$ since $\alpha(\gamma, \tau)>0$ tends to zero as $\tau \to \infty$ and that it is increasing in $\gamma$.  	
 So we loose only an epsilon, by choosing $\tau$ large enough,  with respect to
 the bound from Proposition \ref{prop:M-bounded}.
\end{remark}

\begin{proof}
 Let $0<\alpha<\f{1}{2}$ to be chosen later and write
 \begin{align*}
 	M_\psi^\gamma(f_1,f_2)
 	=
		\iint_{\R^2} \left\{(|T_rf_1||T_rf_2|)^{1-2\alpha}\psi\right\}
		\left\{|T_rf_1||T_rf_2|\right\}^{2\alpha}\left\{(|T_rf_1|+|T_rf_2|)^{\gamma-2}\right\} \, dx dr.
 \end{align*}
 Now use H\"older's inequality for 3 functions with exponents
  $p_1$, $\f{1}{\alpha}$, and, $\f{6}{\gamma-2}$, where
 \begin{align*}
 	\f{1}{p_1}= 1- \alpha -\f{\gamma-2}{6} = \f{8-\gamma-6\alpha}{6}
 \end{align*}
 to see that
  \begin{align*}
 	M_\psi^\gamma(f_1,f_2)
 	&\le \left(\iint_{\R^2}|T_rf_1 T_rf_2|^{\f{6(1-2\alpha)}{8-\gamma-6\alpha}} \psi^{\f{6}{8-\gamma-6\alpha}}\, dxdr
 	     \right)^{\f{8-\gamma-6\alpha}{6}} \\
 	&\phantom{\le}
 		\|T_rf_1 T_rf_2\|_{L^{2}(\R^2, dxdr)}^{2\alpha}
		 \||T_rf_1|+|T_rf_2|\|_{L^{6}(\R^2, dxdr)}^{\gamma-2}.
 \end{align*}
 Up to a constant, the third factor is bounded by $(\|f_1\|+\|f_2\|)^{\gamma-2}$, using the triangle and Strichartz inequalities.
 Using Lemma 2.2, the second factor is bounded by
 \begin{align*}
 	\|T_rf_1 T_rf_2\|_{L^{2}(\R^2, dxdr)}^{2\alpha}
 	\lesssim s^{-\alpha} \|f_1\|^{2\alpha} \|f_2\|^{2\alpha} .
 \end{align*}
 For the first factor, we  note that with the help of the Cauchy-Schwarz inequality one gets
 \begin{align*}
 	\iint_{\R^2}& |T_rf_1 T_rf_2|^{\f{6(1-2\alpha)}{8-\gamma-6\alpha}} \psi^{\f{6}{8-\gamma-6\alpha}}\, dxdr
     \\
     & \le \left( \iint_{\R^2}|T_rf_1|^{\f{12(1-2\alpha)}{8-\gamma-6\alpha}} \psi^{\f{6}{8-\gamma-6\alpha}}\, dxdr \right)^{1/2}
     \left( \iint_{\R^2}|T_rf_2|^{\f{12(1-2\alpha)}{8-\gamma-6\alpha}} \psi^{\f{6}{8-\gamma-6\alpha}}\, dxdr \right)^{1/2}.
 \end{align*}

  In order to use Lemma \ref{lem:boundedness} for this, we need to have
  $2\le q\le 6$ with $q= \f{12(1-2\alpha)}{8-\gamma-6\alpha}$.
  This is equivalent to $6\alpha<8-\gamma$, $6\alpha\le \gamma-2$ and $2\alpha\le 6-\gamma$.

  Moreover, we need
  \begin{align*}
  	\psi^{\f{6}{8-\gamma-6\alpha}} \in L^{\f{4}{6-q}}
    = L^{\f{4(8-\gamma-6\alpha)}{6(6-\gamma-2\alpha)}}
  \end{align*}
  hence
  \begin{align*}
  	\psi\in L^{\f{4}{6-\gamma-2\alpha}} .
  \end{align*}

 Now we come to the choice of $\alpha$: In order to guarantee that $0<\alpha<1$,  $6\alpha<8-\gamma$, $6\alpha\le \gamma-2$, and $2\alpha\le 6-\gamma$, we take any $\tau > 1$ and put $\alpha\coloneqq\alpha(\gamma, \tau)$.
 Then one checks that $\alpha$ obeys the above bounds to finish the proof.
\end{proof}

\begin{lemma}[Duality]\label{lem:duality}
Define
\begin{align*}
	\wti{\psi}(s)\coloneqq \f{1}{(2|s|)^{\f{6-\gamma}{2}}}\psi\big(-\f{1}{4s}\big)	
\end{align*}
for $s\not=0$. Then
 \beq \label{eq:duality}
   M_\psi^\gamma(f_1,f_2) = M_{\wti{\psi}}^\gamma(\check{f}_1,\check{f}_2)
 \eeq
where $\check{f}$ is the inverse Fourier transform of $f$.
\end{lemma}

\begin{remark}
 Of course, the definition of $\wti{\psi}$ depends on
 $\gamma$, but we drop this dependence in our notation, for simplicity.
 For $2\le \gamma\le 6$, Proposition  \ref{prop:M-bounded} yields a natural a priori bound on $M_\psi^\gamma(f_1,f_2)$ which depends on the
 $L^{\f{4}{6-\gamma}}$ norm of $\psi$. It is an easy exercise to check that
 $\|\psi \|_{L^{\f{4}{6-\gamma}}}= \|\wti{\psi} \|_{L^{\f{4}{6-\gamma}}}$, so Proposition \ref{prop:M-bounded} and the duality expressed in \eqref{eq:duality} are consistent.
\end{remark}

\begin{proof}
 Without loss of generality, assume that $f_1$ and $f_2$ are Schwartz functions for the calculations below. Defining $u_j(r,x)\coloneqq (T_rf_j)(x)$ and
 $\check{u}_j(r,x)\coloneqq (T_r\check{f_j})(x)$, $j=1,2$, using the explicit form of the free time evolution \eqref{eq:explicit form of the sol} for $u_j(r,x)$, and expanding the square, one sees
\begin{align*}
	u_j(r,x) =
	\frac{1}{\sqrt{2ir}} e^{i\frac{x^2}{4r}}
	 \check{u}_j\Big(\f{-1}{4r},\f{-x}{2r}\Big)
\end{align*}
which is often called pseudo-conformal invariance of the free Schr\"odinger evolution. Then
\begin{align}\label{eq:duality1}
	&M_\psi^\gamma(f_1,f_2) \nonumber\\
	&= \iint_{\R^2}\frac{\left|\check{u}_1\Big(\f{-1}{4r},\f{-x}{2r}\Big)\right|
	     \left|\check{u}_2\Big(\f{-1}{4r},\f{-x}{2r}\Big)\right|
	     \left( \left|\check{u}_1\Big(\f{-1}{4r},\f{-x}{2r}\Big)\right|
	     +\left|\check{u}_2\Big(\f{-1}{4r},\f{-x}{2r}\Big)\right| \right)^{\gamma-2}}{(2|r|)^{\gamma/2}}
	     \,dx \psi(r)dr .
\end{align}
Doing first the change of variables $x=-2ry$, $dx=2|r|dy$ and then $r=-1/(4s)$ with $dr= (2|s|)^{-2}\, ds$, yields
\begin{align*}
	\eqref{eq:duality1}
	& = \iint_{\R^2}\frac{\left|\check{u}_1(s,y)\right|
	     \left|\check{u}_2(s,y)\right|
	     \left( \left|\check{u}_1(s,y)\right|
	     +\left|\check{u}_2(s,y)\right| \right)^{\gamma-2}}{(2|s|)^{\f{6-\gamma}{2}}}
	     \,dy \psi(-\f{1}{4s})ds
\end{align*}
which completes the proof.
\end{proof}

This duality is a convenient tool in the proof of the analogue of  Proposition \ref{prop:M-fourier} when the functions $f_1$ and $f_2$ have separated supports.
\begin{proposition}\label{prop:M-time}
Let $s= \dist (\supp {f_1}, \supp  †{f_2})>0$.
If $2<\gamma< 6$, $\tau > 1$ and
$\psi\in L^{\beta(\gamma, \tau)}(|r|^{\alpha(\gamma,\tau)\beta(\gamma, \tau)}dr)$, then
 	\begin{align}\label{eq:M-time}
	M_\psi^\gamma(f_1,f_2)
	\lesssim s^{-\alpha(\gamma, \tau)}\|f_1\|\|f_2\|(\|f_1\|+\|f_2\|)^{\gamma-2}.
	\end{align}
\end{proposition}
\begin{proof}
	Given the duality expressed in Lemma \ref{lem:duality} this is now  simple: We have
	\begin{align*}
		M_\psi^\gamma(f_1,f_2) = M_{\wti{\psi}}^\gamma(\check{f}_1,\check{f}_2)
	\end{align*}
	and note that the assumption on the separation of the supports of $f_1$ and $f_2$ means, of course, that $\check{f}_1$ and $\check{f}_2$ have separated Fourier support, so Proposition \ref{prop:M-fourier} applies to $M_{\wti{\psi}}^\gamma(\check{f}_1,\check{f}_2)$ as long as $\wti{\psi}$ is in the correct $L^p$ space.
	A short calculation shows
	\begin{align*}
		\|\wti{\psi}\|_{L^p(dr)}^p = \int_\R (2|r|)^{\f{p(6-\gamma)}{2}-2} |\psi(r)|^p\, dr
	\end{align*}
	and \eqref{eq:M-time} follows by choosing
	$p=\beta(\gamma,\tau)$.
\end{proof}

To handle the cases with $6\le\gamma<10$ for positive average dispersion, we need a fractional bilinear estimate for $M_\psi^\gamma$ in $H^1$ as follows.
\begin{proposition}[$H^1$ bilinear estimate] \label{prop:M-H1-bilinear}
  Let $\gamma\ge 2$ and $\psi\in L^{1}(\R)$ with compact support. Then for any $f_1, f_2 \in H^1(\R)$ with
  $s= \dist (\supp f_1, \supp f_2)>0$,
  \bdm 
    M_\psi^\gamma(f_1,f_2)
      \lesssim s^{-1}\|f_1\|_{H^1} \|f_2\|_{H^1}(\|f_1\|_{H^1}+\|f_2\|_{H^1})^{\gamma-2},
  \edm
  where the implicit constant depends only on the support and the $L^{1}$ norm of $\psi$.
\end{proposition}

\begin{proof}
{}From the definition of $M_\psi^\gamma(f_1,f_2)$ one sees
\begin{align}\label{eq:split1}
	M_\psi^\gamma(f_1,f_2)
	\le
	  \|T_rf_1 T_rf_2\|_{L^{1}(\R^2, dx\psi dr)}
	  \sup_{r\in\R }(\|T_rf_1\|_{L^\infty}+\|T_rf_2\|_{L^{\infty}})^{\gamma-2} .
\end{align}
We use \eqref{eq:nice2} to bound the second factor in \eqref{eq:split1} by
$(\|f_1\|_{H^1} + \|f_2\|_{H^1})^{\gamma-2}$.

To bound the first factor, we use the positive operators $P^{\scriptscriptstyle{\le}}_L$ and $P^{\scriptscriptstyle{>}}_L$ from Lemma \ref{lem:gaussian coherent states} for suitably chosen $L>0$.
Although they are not projection operators, we think of $P^{\scriptscriptstyle{\le}}_L $ as `projecting' onto frequencies localized to $\lesssim L$ and $P^{\scriptscriptstyle{>}}_L$ as `projecting' onto large frequencies $\gtrsim L$.
At the same time, the supports of $P^{\scriptscriptstyle{\le}}_Lf_1$ and $P^{\scriptscriptstyle{>}}_L f_2$ will still be essentially separated.
See Lemma \ref{lem:completeness} and \ref{lem:gaussian coherent states} in Appendix \ref{sec:Galilei} for the properties of $P^{\scriptscriptstyle{\le}}_L$ and $P^{\scriptscriptstyle{>}}_L$ which we will need.

Since $P^{\scriptscriptstyle{\le}}_L+P^{\scriptscriptstyle{>}}_L=\id$ on $L^2(\R)$ by Lemma \ref{lem:completeness}, we can use the triangle inequality and the linearity of $T_r$  to split
 \begin{equation}\label{eq:split3}
 \begin{split}
     \|T_r f_1 T_r f_2 & \|_{L^{1}(\R^2,\, dx\psi dr)}
      \\
    &\le \|T_r P^{\scriptscriptstyle{>}}_L f_1 T_r f_2\|_{L^{1}(\R^2,\, dx\psi dr)}
    + \|T_r P^{\scriptscriptstyle{\le}}_L f_1 T_r P^{\scriptscriptstyle{>}}_L f_2\|_{L^{1}(\R^2,\, dx\psi dr)}  \\
    &\phantom{\le} + \|T_r P^{\scriptscriptstyle{\le}}_L f_1 T_r P^{\scriptscriptstyle{\le}}_L f_2\|_{L^{1}(\R^2,\, dx\psi dr)}.
 \end{split}
 \end{equation}
The Cauchy--Schwarz  inequality and Lemma \ref{lem:boundedness} yield
 \begin{align*}
   \|T_r P^{\scriptscriptstyle{>}}_L f_1 T_r f_2\|_{L^{1}(\R^2,\, dx\psi dr)}
   &\le \|T_r P^{\scriptscriptstyle{>}}_L f_1\|_{L^{2}(\R^2,\, dx\psi dr)}
   \|T_r f_2\|_{L^{2}(\R^2,\, dx\psi dr)}\\
   &\lesssim \|P^{\scriptscriptstyle{>}}_L f_1\| \|f_2\|
     \lesssim L^{-1} \|f_1\|_{H^1} \|f_2\|,
  \end{align*}
  where we also use  \eqref{eq:high momenta} in the last bound. Note that the implicit constant from Lemma \ref{lem:boundedness} depends only on the $L^1$ norm
  of $\psi$.   Switching the roles of $f_1$ and $f_2$, using in addition that $P^{\scriptscriptstyle{\le}}_L\le \id$, shows
 \begin{align*}
   \|T_r P^{\scriptscriptstyle{\le}}_L f_1 T_r P^{\scriptscriptstyle{>}}_L f_2\|_{L^{1}(\R^2,\, dx\psi dr)}
  \lesssim L^{-1} \| f_1\| \| f_2\|_{H^1}.
 \end{align*}
To bound the last term of the right hand side in \eqref{eq:split3}, we use the bound \eqref{eq:low momenta bilinear}  to get
 \begin{align*}
     \|T_r P^{\scriptscriptstyle{\le}}_L f_1 T_r P^{\scriptscriptstyle{\le}}_L f_2\|_{L^{1}(\R^2,\, dx\psi dr)}
     &\le \|\psi\|_{L^1} \sup_{|r|\le R} \|T_r P^{\scriptscriptstyle{\le}}_L f_1 T_r P^{\scriptscriptstyle{\le}}_L f_2\|_{L^{1}(\R, dx)} \\
     &\le \|\psi\|_{L^1} A_R L^2\, e^{L^2 - B_{1,R}s^2}\, \|f_1\| \|f_2\|,
 \end{align*}
 with $R>0$ chosen such that $\supp\ \psi \subset [-R,R]$ and the constants
 $A_R$ and $B_{1,R}$ from Lemma \ref{lem:gaussian coherent states}.
Therefore
 \bdm
     \|T_r f_1 T_r f_2  \|_{L^{1}(\R^2,\, dx\psi dr)}
       \lesssim
          \Big[L^2 e^{L^2- B_{1,R}s^2} + L^{-1}\Big]\|f_1\|_{H^1} \|f_2\|_{H^1}
 \edm
for any $L \ge 0$. Choosing $2L^2= B_{1,R}s^2$, we get
\begin{align*}
	\|T_r f_1 T_r f_2\|_{L^{1}(\R^2,\, dx\psi dr)} \lesssim s^{-1}\|f_1\|_{H^1} \|f_2\|_{H^1},
\end{align*}
and using this in \eqref{eq:split1} proves Proposition \ref{prop:M-H1-bilinear}.
\end{proof}

\subsection{Splitting the nonlocal nonlinearity }\label{subsec:Vsplitting}

For the nonlinear potential $V:\R_+\to\R $ our assumption \ref{ass:A1} guarantees a simple bound which is central for our existence proofs.

\begin{lemma}\label{lem:V-splitting}	
	Assume that $V$ obeys \ref{ass:A1}
. Then
	\begin{align}\label{eq:V-bound}
		|V(a)|\lesssim  	a^{\gamma_1} + a^{\gamma_2}
	\end{align}
	for all $a\ge 0$. Moreover,
		\begin{align}\label{eq:V-splitting-1}
		|V(|z+w|)- V(|z|)|
		\lesssim |w|\left( (|z|+|w|)^{\gamma_1-1} + (|z|+|w|)^{\gamma_2-1} \right)
	\end{align}
	and
	\begin{align}\label{eq:V-splitting-2}
		|V(|z+w|)- V(|z|) -V(|w|)|
		\lesssim |z||w|\left( (|z|+|w|)^{\gamma_1-2} + (|z|+|w|)^{\gamma_2-2} \right)
	\end{align}
	for all $z,w\in \C $.
\end{lemma}
\begin{proof}
  As already observed in Remark \ref{rem:simple},  \eqref{eq:V-bound} follows from integrating the bound for $V'$.
For the second claim, 	let $c=\min \{|z|,|z+w|\}$ and $d=\max \{|z|,|z+w|\}\le |z|+|w|$. Then $d-c=||z+w|-|z||\le |w|$ and using the triangle inequality and \ref{ass:A1}, we have
	\begin{align*}
	 \left| V(|z+w|) -V(|z|) \right|
		&\leq
			\int\limits_{c}^{d} |V'(a)|\, da
			\lesssim  (d^{\gamma_1-1}+d^{\gamma_2-1})(d-c) \nonumber\\
		&\le ((|z|+|w|)^{\gamma_1-1}+(|z|+|w|)^{\gamma_2-1})|w| .
	\end{align*}
For the last claim note that
since $V(0)=0$, we have $V(|z+w|)-V(|z|)-V(|w|)=0$ if at least one of $z$
and $w$ equals zero. So assume $z,w\neq 0$ in the following. Then
	\begin{equation}
	\begin{split}\label{eq:diff1}
		V(|z+w|) -V(|z|) - V(|w|)
		&=
			 \left[ \frac{1}{|z|+|w|}V(|z+w|) - \frac{1}{|z|}V(|z|) \right] |z| \\
		&\phantom{=} \, + \left[ \frac{1}{|z|+|w|}V(|z+w|) - \frac{1}{|w|}V(|w|) \right] |w| .
	\end{split}
	\end{equation}
	Moreover,
	\begin{align} \label{eq:diff2}
		\frac{1}{|z|+|w|}V(|z+w|) - \frac{1}{|z|}V(|z|)
			&=
				\frac{1}{|z|+|w|} \left( V(|z+w|) - V(|z|) \right) - \frac{|w|}{(|z|+|w|)|z|} V(|z|) .
	\end{align}
	Using \eqref{eq:V-bound} we have
	\begin{align*}
		\frac{|w|}{(|z|+|w|)|z|}|V(|z|)|
		\lesssim
			\frac{|w|}{|z|+|w|}(|z|^{\gamma_1-1}+ |z|^{\gamma_2-1})
		\le |w| ((|z|+|w|)^{\gamma_1-2}+ (|z|+|w|)^{\gamma_2-2}) ,
	\end{align*}
	which together with \eqref{eq:V-splitting-1} in  \eqref{eq:diff2} gives
	\begin{align*}
		\left|\frac{1}{|z|+|w|}V(|z+w|) - \frac{1}{|z|}V(|z|)\right|
		\lesssim ((|z|+|w|)^{\gamma_1-2}+(|z|+|w|)^{\gamma_2-2})|w|
	\end{align*}
	and a similar inequality holds when we switch $z$ and $w$, so  \eqref{eq:diff1} and \eqref{eq:diff2} imply \eqref{eq:V-splitting-2}.
	\end{proof}
Recall that the nonlocal nonlinearity is given by
\begin{align*}
  N(f) = \iint_{\R^2} V(|T_rf(x)|)\, dx\psi dr 	.
\end{align*}

\begin{proposition}[Boundedness]\label{prop:N-boundedness}
	Assume that $V$ obeys assumption \ref{ass:A1}
. Furthermore, for $j=1,2$ choose $\kappa_j$ with\footnote{Here $(x)_+=\max\{x,0\}$ is the positive part of $x\in\R$.}
	$(\gamma_j-6)_+\le \kappa_j\le \gamma_j-2$
	and assume that $\psi\in L^{\f{4}{6-\gamma_1+\kappa_1}}\cap L^{\f{4}{6-\gamma_2+\kappa_2}}$.
	Then for all $f\in H^1(\R)$
	\begin{align}\label{eq:N-boundedness}
		|N(f)|\lesssim \|f'\|^{\f{\kappa_1}{2}}\|f\|^{\gamma_1-\f{\kappa_1}{2}}
		+ \|f'\|^{\f{\kappa_2}{2}}\|f\|^{\gamma_2-\f{\kappa_2}{2}},
	\end{align}
	where the implicit constant depends only on the $L^{\f{4}{6-\gamma_1+\kappa_1}}$ and
	$L^{\f{4}{6-\gamma_2+\kappa_2}}$ norms of $\psi$. 	
	Moreover, if $2\le \gamma_1\le\gamma_2\le 6$ and $\kappa_1=\kappa_2=0$, then the above bound extends to all $f\in L^2(\R)$.
\end{proposition}
\begin{remark}
	As the condition in Proposition \ref{prop:N-boundedness} indicates, we need $\kappa_j>0$ only when $\gamma_j>6$ for some $j=1,2$. If $2\le \gamma_1\le \gamma_2\le 6$, we can bound $N(f)$ solely in terms of the $L^2$ norm of $f$. This is not possible anymore if some exponent $\gamma_j$ is bigger than $6$. In this case one has to use an $L^\infty$ bound and \eqref{eq:Kato} to extract some excess power and for this one has to pay the price that the bound then contains the $L^2$ norm of the
	derivative of $f$, but this allows to go beyond the exponent $6$ in the existence results for $\dav>0$.
	
	Note also that the choice $\kappa_j=\gamma_j-2$
	is allowed. In this case Proposition \ref{prop:N-boundedness} shows that for arbitrary
	$\psi\in L^1$ the nonlinearity is bounded by
	\begin{align}\label{eq:consequence1}
		|N(f)|\lesssim \|f'\|^{\f{\gamma_1-2}{2}}\|f\|^{\f{\gamma_1+2}{2}}
		+ \|f'\|^{\f{\gamma_2-2}{2}}\|f\|^{\f{\gamma_2+2}{2}} .
	\end{align}
  Moreover, using $\|f'\|\|f\|\le \|f\|_{H^1}^2$, \eqref{eq:consequence1} also gives the bound
   \begin{align}\label{eq:consequence2}
		|N(f)|\lesssim \|f\|^{2}
		\left( \|f\|_{H^1}^{\gamma_1-2} + \|f\|_{H^1}^{\gamma_2-2}\right)
	\end{align}
	where the implicit constant only depends on the $L^1$ norm of $\psi$.
\end{remark}
\begin{proof}[Proof of Proposition \ref{prop:N-boundedness}:]
  Take an arbitrary $f\in H^1(\R)$. As in the proof of Proposition \ref{prop:M-H1-bilinear} we can use \eqref{eq:Kato} to get
  \begin{align*}
  	\sup_{r\in\R } \|T_rf\|_{L^\infty} \le \|f'\|\|f\| .
  \end{align*}
  Thus, for any $\gamma\ge 2$ and $\kappa\ge 0$ with $\gamma-\kappa > 0$, we have
  \begin{align*}
  	\iint_{\R^2} |T_rf(x)|^\gamma \, dx \psi dr
  	&\le \sup_{r\in\R }\|T_rf\|_{L^\infty}^\kappa
  		\iint_{\R^2} |T_rf(x)|^{\gamma-\kappa} \, dx \psi dr \\
  	&\le \|f'\|^{\f{\kappa}{2}} \|f\|^{\f{\kappa}{2}}
  		\iint_{\R^2} |T_rf(x)|^{\gamma-\kappa} \, dx \psi dr .
  \end{align*}
  If, in addition, $2\le \gamma-\kappa\le 6$ and $\psi\in L^{\f{4}{6-\gamma+\kappa}}(\R)$, then we can use Lemma \ref{lem:boundedness} to see
  \begin{align*}
  	\iint_{\R^2} |T_rf(x)|^{\gamma-\kappa} \, dx \psi dr \lesssim \|f\|^{\gamma-\kappa},
  \end{align*}
  where the implicit constant depends only on the $L^{\f{4}{6-\gamma+\kappa}}$ norm of $\psi$.
  Thus,
  \begin{align*}
  	\iint_{\R^2} |T_rf(x)|^\gamma \, dx \psi dr
  	\lesssim \|f'\|^{\f{\kappa}{2}}\|f\|^{\gamma-\f{\kappa}{2}}
  \end{align*}
  for all $(\gamma-6)_+\le \kappa\le \gamma-2$.
  With the bound \eqref{eq:V-bound} 
   and the definition of $N(f)$ this proves \eqref{eq:N-boundedness} under the assumption that
  $\psi\in L^{\f{4}{6-\gamma_1+\kappa_1}}\cap L^{\f{4}{6-\gamma_2+\kappa_2}}$.
\end{proof}

\begin{proposition}[Splitting $N$]\label{prop:N-splitting}
Assume that $V$ obeys assumption \ref{ass:A1}.
  \begin{theoremlist}
  	\item
    If $2\le \gamma_1\le \gamma_2\le 6$ and $\psi\in L^1\cap L^{\f{4}{6-\gamma_2}}$, then
  \begin{align} \label{eq:N-splitting1}
     |N(f_1+f_2)- N(f_1) - N(f_2)|
     \lesssim
       \|f_1\|\|f_2\| \left(1+\|f_1\|^4+\|f_2\|^{4}\right) .
  \end{align}
	\item
 	If  $2<\gamma_1\le \gamma_2< 6$ and $\tau>1$, then with $\alpha(\gamma_1,\tau)$ and $\beta(\gamma_2,\tau)$ as in Proposition \ref{prop:M-fourier},
    \begin{equation}\label{eq:N-splitting2}
    \begin{split}
       |N(f_1+f_2)- N(f_1) - N(f_2)| \lesssim &
        ~s^{-\min\{\alpha(\gamma_1,\tau),\alpha (\gamma_2,\tau)\}} \|f_1\|\|f_2\|
        \left(1+\|f_1\|^4+\|f_2\|^{4}\right)
	\end{split}
	\end{equation}
 	if  $\psi\in L^1 \cap L^{\beta(\gamma_2,\tau)}$ and $s= \dist (\supp \hat{f_1}, \supp  \hat{f_2})>0$, or $\psi\in L^{\beta(\gamma_2, \tau)}$ has compact support and $s= \dist (\supp f_1, \supp  f_2)>0$.
 \item If  $2\le\gamma_1\le\gamma_2<\infty$ and $\psi\in L^1$, then
		\begin{align}\label{eq:N-splitting3}
		|N(f_1+f_2)- N(f_1) - N(f_2)|
		\lesssim \|f_1\|\|f_2\| \left(1+ \|f_1\|_{H^1}^{\gamma_2-2}+\|f_2\|_{H^1}^{\gamma_2-2}\right).
	\end{align}
 \item
  If  $2\le\gamma_1\le \gamma_2 < \infty$ and $\psi\in L^1$ has compact support, then
    \begin{equation}\label{eq:N-splitting4}
    \begin{split}
       |N(f_1+f_2)- N(f_1) - N(f_2)|
       \lesssim ~
	     s^{-1}\|f_1\|_{H^1}\|f_2\|_{H^1}
	     \left(1+\|f_1\|_{H^1}^{\gamma_2-2}+\|f_2\|_{H^1}^{\gamma_2-2}\right)
	\end{split}
	\end{equation}
 with  $s= \dist (\supp f_1, \supp  f_2)>0$.
  \end{theoremlist}
\end{proposition}
\begin{proof}
 Because of Lemma \ref{lem:V-splitting} and the Definition \ref{def:building block} of $M^\gamma_\psi$, we have
 \begin{align}
 \begin{split}\label{eq:N-splitting0}
   	|N(f_1+f_2) &- N(f_1) - N(f_2)| \\
   	& \le \iint_{\R^2} \Bigl| V(|T_rf_1(x)+T_rf_2(x)|) -V(|T_rf_1(x)|) -V(|T_rf_2(x)|) \Bigr|\, dx \psi dr \\
   	&\lesssim
   	  M_\psi^{\gamma_1}(f_1,f_2) + M_\psi^{\gamma_2}(f_1,f_2).
 \end{split}
 \end{align}
 So \eqref{eq:N-splitting3} follows from Proposition \ref{prop:M-bounded} and
 \eqref{eq:N-splitting4} follows from Proposition \ref{prop:M-H1-bilinear}, noting also that
 \begin{align*}
 	(a+b)^{\gamma_1-2} + (a+b)^{\gamma_2-2} \lesssim  1+a^{\gamma_2-2} + b^{\gamma_2-2} ,
 \end{align*}
 for all $a,b\ge 0$.
 Similarly, \eqref{eq:N-splitting1} follows from Proposition \ref{prop:M-bounded} as long as
 $\psi\in L^{\f{4}{6-\gamma_1}}\cap L^{\f{4}{6-\gamma_2}}$. Since we also assume $\psi\in L^1$ for convenience, this condition reduces to $\psi\in L^1\cap L^{\f{4}{6-\gamma_2}}$.

 For the proof of \eqref{eq:N-splitting2}, we first assume
 $s= \dist (\supp \hat{f_1}, \supp  \hat{f_2})>0$.
 Then Proposition \ref{prop:M-fourier} shows
  \begin{align*}
   	M_\psi^{\gamma}(f_1,f_2)
   	\lesssim s^{-\alpha(\gamma,\tau)}\|f_1\|\|f_2\|(\|f_1\|+\|f_2\|)^{\gamma-2}
   \end{align*}
 for any $2<\gamma<6$ and $\tau>1$,
 as long as $\psi\in L^{\beta(\gamma,\tau)}$. \\
 Thus \eqref{eq:N-splitting2} follows from \eqref{eq:N-splitting0} as long as
 $\psi\in L^{\beta(\gamma_1,\tau)}\cap L^{\beta(\gamma_2, \tau)}$.
 Noting
 $$
 1<\beta(\gamma_1, \tau) \le \beta(\gamma_2, \tau)
 \quad\text{ and }\quad
 	L^1\cap L^{\beta(\gamma_2, \tau)}
 	  \subset L^{\beta(\gamma_1,\tau)}\cap L^{\beta(\gamma_2, \tau)}
 $$
 finishes the proof of \eqref{eq:N-splitting2} when $\hat{f}_1$ and $\hat{f}_2$ have separated supports.

 If $s= \dist (\supp {f_1}, \supp  {f_2})>0$, we make the simple observation that
 for any compactly supported $\psi$ one has
 \begin{align*}
 	  \psi\in L^p \Rightarrow \psi\in L^p(|r|^a\, dr)\cap L^1
 \end{align*}
 for any weight $|r|^a$ with $a\ge 0$ and  $p\ge 1$.  With this observation, the above proofs carry over to the case that the functions $f_1$ and $f_2$ have separated supports, using now Proposition \ref{prop:M-time} instead of Proposition \ref{prop:M-fourier}.
\end{proof}

%

\section{Strict subadditivity of the ground state energy} \label{sec:concavity}

Recall that for $\dav \ge 0$
\begin{align*}
	H(f)= \f{\dav}{2}\|f'\|^2 - N(f)
\end{align*}
and
\begin{align*}
	E_\lambda^{\dav}= \inf \left\{ H(f): \|f\|^2=\lambda\right\}\, ,
\end{align*}
where,  if $f\in L^2\setminus H^1$,  we set $\|f'\|=\infty$, so the infimum in the definition of $E^{\dav}_\lambda$ is over all $f\in H^1$ with fixed $L^2$ norm if $\dav>0$.

In this section, we will give an a-priori bound on the ground-state energy which will be an essential ingredient in the construction of strongly convergent minimizing sequences. Recall also the definition of $\alpha_\delta=\max \{1,\f{4}{10-\gamma_2}+\delta \}$ for $\delta\ge0$ from Theorem \ref{thm:existence+}.
\begin{lemma} \label{lem:E-boundedness}
  Assume that $V$ obeys assumption \ref{ass:A1}.
  \begin{theoremlist}
  \item
   If  $\dav=0$, $2< \gamma_1\le \gamma_2\le 6$ and $\psi\in L^1\cap L^{\f{4}{6-\gamma_2}}$, then for every $\lambda >  0$
  \begin{align*}
  	-\infty< E^0_\lambda \le 0,
  \end{align*}
  in particular, the variational problem \eqref{eq:min} is well--posed.
  \item
  If $\dav>0$, $2< \gamma_1\le \gamma_2<10$ and $\psi\in L^1\cap L^{\alpha_\delta}$ for some $\delta>0$, then the energy functional $H(f)$ 
  is coercive in $\|f'\|$ for fixed $\|f\|$, that is,
  \begin{align}\label{eq:blow-up of H(f)}
  \lim_{\|f'\|\to\infty}  H(f)=\infty 	
  \end{align}
  for fixed $\|f\|^2=\lambda>0$.
  Also
   \begin{align*}
  	-\infty< E^\dav_\lambda \le 0
  \end{align*}
  and thus the variational problem \eqref{eq:min} is well--posed and  any minimizing sequence $(f_n)_n\subset H^1(\R)$ for $E^\dav_\lambda$ is bounded
  in the $H^1$ norm.
  \item
  If $V$ obeys assumption \ref{ass:A4}, then $E^\dav_\lambda<0$ for any $\lambda>0$ and
  $\dav\ge 0$.
  \end{theoremlist}
\end{lemma}

\begin{proof} If $2< \gamma_1\le \gamma_2\le 6$ we choose $\kappa_1=\kappa_2=0$ in Proposition \ref{prop:N-boundedness} to see that for any $\psi\in L^1\cap L^{\f{4}{6-\gamma_2}}$ one has
\begin{align*}
	|N(f)|\lesssim \|f\|^{\gamma_1} + \|f\|^{\gamma_2} .
\end{align*}
Thus for $\dav=0$ we have
\begin{align*}
	E^0_\lambda = -\sup_{\|f\|^2=\lambda}N(f)\gtrsim - (\lambda^{\f{\gamma_1}{2}} + \lambda^{\f{\gamma_2}{2}}) >-\infty .
\end{align*}

To get a finite lower bound for $E^\dav_\lambda$ for $\dav>0$ and $2< \gamma_1\le\gamma_2<10$ we have to do a little bit of  numerology first: If $\alpha_\delta=1$, simply set $\kappa_j\coloneqq\gamma_j-2$. Since $\psi\in L^1$ the conditions on $\psi$ from Proposition \ref{prop:N-boundedness} are clearly satisfied.  Note that if $\alpha_\delta=1$, then necessarily $\gamma_2<6$, thus also
	$\gamma_1\le\gamma_2<6$, and hence
	\begin{align*}
		(\gamma_j-6)_+=0\le \kappa_j= \gamma_j-2.
	\end{align*}
  This shows that the condition on $\kappa_j$ from Proposition \ref{prop:N-boundedness} are fulfilled and it also shows that $\kappa_j\le \gamma_2-2<4$ in this case.
	
	If $\alpha_\delta>1$, pick $\beta_j>0$ such that
	$\f{4}{10-\gamma_j-\beta_j}= \alpha_\delta >1$. Setting $\kappa_j\coloneqq (4-\beta_j)_+$ we certainly have $0\le\kappa_j<4$. Also, since $10-\gamma_j-\beta_j>0$, we have $\beta_j<10-\gamma_j$ and this implies
	\begin{align*}
		\kappa_j=(4-\beta_j)_+\ge (\gamma_j-6)_+\,\,.
	\end{align*}
	Also, since $\f{4}{10-\gamma_j-\beta_j}>1$, we have
	\begin{align*}
		\kappa_j<\gamma_j-2 .
	\end{align*}
 So again, the conditions on $\kappa_j$ from Proposition \ref{prop:N-boundedness} are fulfilled. So from  \eqref{eq:N-boundedness} we get
 \begin{align*}
 	|N(f)|\lesssim \|f'\|^{\f{\kappa_1}{2}}\|f\|^{\gamma_1-\f{\kappa_1}{2}}
		+ \|f'\|^{\f{\kappa_2}{2}}\|f\|^{\gamma_2-\f{\kappa_2}{2}}
 \end{align*}
 and there exists a constant $C>0$ depending only on the $L^{\f{4}{6-\gamma_1+\kappa_1}}$ and
	$L^{\f{4}{6-\gamma_2+\kappa_2}}$ norms of $\psi$ such that
	\begin{align}\label{eq:coercive1}
	  H(f)\ge \f{\dav}{2} \|f'\|^2 - C\left( \|f'\|^{\f{\kappa_1}{2}}\|f\|^{\gamma_1-\f{\kappa_1}{2}}
		+ \|f'\|^{\f{\kappa_2}{2}}\|f\|^{\gamma_2-\f{\kappa_2}{2}} \right)	.
	\end{align}
 Since $\kappa_j=(4-\beta_j)_+$, for $\alpha_\delta>1$,  we have $6-\gamma_j+\kappa_j \ge 10-\gamma_j-\beta_j$, so
 \begin{align*}
 	1< \f{4}{6-\gamma_j+\kappa_j}\le \f{4}{10-\gamma_j-\beta_j} =\alpha_\delta
 \end{align*}
 and by interpolating, or simply H\"older's inequality, the constant $C$ in \eqref{eq:coercive1} can be made to depend only on the $L^1$ and $L^{\alpha_\delta}$ norms of $\psi$.

 If we fix $\|f\|^2=\lambda>0$, then we can rewrite \eqref{eq:coercive1} with $\|f'\|=t$ as
 \begin{align}\label{eq:coercive2}
 	H(f)\ge \f{\dav}{2}t^2 - C\left( t^{\f{\kappa_1}{2}}\lambda^{\f{\gamma_1-\f{\kappa_1}{2}}{2}}
		+ t^{\f{\kappa_2}{2}}\lambda^{\f{\gamma_2-\f{\kappa_2}{2}}{2}} \right)	.
 \end{align}
 Since $\kappa_j<4$ for $j=1,2$, this immediately implies \eqref{eq:blow-up of H(f)}. 

  The lower bound \eqref{eq:coercive1} also shows that
  \begin{align*}
  	E^\dav_\lambda =\inf\big\{H(f):\, \|f\|^2=\lambda\big\}
  	\ge \inf_{t>0}\left(\f{\dav}{2}t^2 - C\left( t^{\f{\kappa_1}{2}}\lambda^{\f{\gamma_1-\f{\kappa_1}{2}}{2}}
		+ t^{\f{\kappa_2}{2}}\lambda^{\f{\gamma_2-\f{\kappa_2}{2}}{2}} \right)\right)
		>-\infty \, .
  \end{align*}
  The coercivity expressed in \eqref{eq:coercive2} also makes it easy to see that
  any minimizing sequence $(f_n)_n\subset H^1(\R)$ for $E^\dav_\lambda$ is bounded
  in the $H^1$ norm. Indeed, if $f_n$ is such that $\|f_n\|^2=\lambda>0$ and $H(f_n)\to E^\dav_\lambda>-\infty$ as $n\to\infty$, then the lower bound \eqref{eq:coercive2} shows that $\|f_n'\|$ stays bounded and hence also $\|f_n\|_{H^1}^2= \|f_n\|^2+\|f_n'\|^2$
  stays bounded.

  To finish the proof of Lemma \ref{lem:E-boundedness}, we have to show that for $\lambda>0$ and $\dav\ge0$ one has $E^\dav_\lambda\le 0$ and even $E^\dav_\lambda< 0$, if, in addition,  assumption \ref{ass:A4} on $V$ holds. We will do this by computing the energy of suitable Gaussians.

 For this, we let $g_{\sigma_0}$ be the centered Gaussian from \eqref{eq:Gauss}
 with $\sigma_0>0$. Then $\|g_{\sigma_0}\|^2=\lambda>0$,
 $\|g_{\sigma_0}'\|^2=\lambda/\sigma_0$, and its time evolution is given in Lemma \ref{lem:Gaussians} by
   \begin{align*}
  	T_rg_{\sigma_0}(x) =  \left(\frac{2 \lambda^2}{\pi \sigma_0}\right)^{1/4} \left( \f{\sigma_0}{\sigma(r)} \right)^{1/2} e^{-\f{x^2}{\sigma(r)}}
  \end{align*}
 with $\sigma(r)=\sigma_0+4ir$.
 The first bound from Lemma \ref{lem:V-splitting} shows
 \begin{align*}
 	|N(g_{\sigma_0})|\le  \iint_{\R^2} |V(|T_rg_{\sigma_0}(x)|)|\, dx\psi dr \lesssim
 	  \|\psi\|_{L^1} \left(\|T_rg_{\sigma_0}\|_{L^{\gamma_1}}^{\gamma_1} + \|T_rg_{\sigma_0}\|_{L^{\gamma_2}}^{\gamma_2}\right)
 \end{align*}
 and Lemma \ref{lem:Gaussians} gives
 \begin{align*}
 	\|T_rg_{\sigma_0}\|_{L^\gamma}^{\gamma}
 	= \left( \f{\pi}{\gamma} \right)^{1/2} \left( \f{2\lambda^2}{\pi} \right)^{\gamma/4}
 	  	\sigma_0^{-\f{\gamma-2}{4}}
  			 \left( \f{|\sigma_0|}{|\sigma(r)|} \right)^{\f{\gamma-2}{2}}
 	\le \left( \f{\pi}{\gamma} \right)^{1/2} \left( \f{2\lambda^2}{\pi} \right)^{\gamma/4}
 	  	\sigma_0^{-\f{\gamma-2}{4}}
\, .
 \end{align*}
Thus,
 \begin{align*}
 	H(g_{\sigma_0})&= \f{\dav}{2}\|g_{\sigma_0}'\|^2 -N(g_{\sigma_0}) \\
 	&\le \f{\dav}{2}\cdot\f{\lambda}{\sigma_0} + C\|\psi\|_{L^1} \left( \f{\pi}{\gamma} \right)^{1/2}
    \left[\left( \f{2\lambda^2}{\pi} \right)^{\gamma_1/4} \sigma_0^{-\f{\gamma_1-2}{4}}
    +\left( \f{2\lambda^2}{\pi} \right)^{\gamma_2/4} \sigma_0^{-\f{\gamma_2-2}{4}}
  \right]
 \end{align*}
  for some constant $C>0$. Since $2<\gamma_1\le \gamma_2$ we can let $\sigma_0\to\infty$ to see
 \begin{align*}
   \lim_{\sigma_0\to\infty} H(g_{\sigma_0}) =0
 \end{align*}
 which clearly implies $E^\dav_\lambda\le 0$.

To apply \ref{ass:A4} when $\dav>0$, we consider $\sigma_0$ large enough so that
$$
 |T_rg_{\sigma_0}(x)| \le \left(\frac{2 \lambda^2}{\pi \sigma_0}\right)^{1/4} \left( \f{|\sigma_0|}{|\sigma(r)|} \right)^{1/2}\le \left(\f{2\lambda^2}{\pi\sigma_0}\right)^{1/4}<\epsilon .
$$
Then \ref{ass:A4} implies the lower bound
\begin{align*}
 N(g_{\sigma_0})&=\iint_{\R^2}V(|T_r g_{\sigma_0}(x)|)dx \psi dr
 \gtrsim \iint_{\R^2}|T_r g_{\sigma_0}(x)|^{\kappa_0} dx \psi dr \\
 &= \left( \f{\pi}{\kappa_0} \right)^{1/2} \left( \f{2\lambda^2}{\pi}
 	\right)^{\f{\kappa_0}{4}}
 	  	\sigma_0^{\f{2-\kappa_0 }{4}}
 	  	\int_{\R} \left( \f{\sigma_0}{|\sigma(r)|} \right)^{\f{\kappa_0 -2}{2}}\, \psi(r)dr\\
 &= \left( \f{\pi}{\kappa_0} \right)^{1/2} \left( \f{2\lambda^2}{\pi}
 		\right)^{\f{\kappa_0}{4}}
 	  	\sigma_0^{\f{2-\kappa_0}{4}}
 	\int_{\R}\f{\psi(r)}{[1+(4r/\sigma_0)^2]^{\f{\kappa_0-2}{4}}}\,dr,
 \end{align*}
 where in the second line we used \eqref{eq:Gaussian-Lp-norm}.
Thus the energy of this Gaussian test function is bounded above by
 \bdm
     H(g_{\sigma_0})
     \leq
     \frac{d_{\text{av}}\lambda}{2\sigma_0}
     \left[
     1- \f{C}{\dav \lambda}\left( \f{\pi}{\kappa_0} \right)^{1/2} \left( \f{2\lambda^2}{\pi}
 		\right)^{\f{\kappa_0}{4}}
 	  	\sigma_0^{\f{6-\kappa_0 }{4}}
 	  	\int_{\R}\f{\psi(r)}{(1+(4r/\sigma_0)^2)^{\f{\kappa_0-2}{4}}}\,dr
     \right]
 \edm
for some constant $C$. So, using a large enough $ \sigma_0$, we get $H(g_{\sigma_0})<0$ since $2 < \kappa_0<6$ and
 \bdm
 \int_{\R}\f{\psi(r)}{[1+(4r/\sigma_0)^2]^{\f{\kappa_0-2}{4}}}\,dr \to \|\psi \|_{L^1}
 \edm
as $\sigma_0 \to \infty$ by Lebesgue's dominated convergence theorem.

If $\dav=0$, we again use the Gaussian $g_{\sigma_0}$ with $\sigma$ so large that
$0<|T_rg_{\sigma_0}|\le \veps$. Then \ref{ass:A4} implies
$
  H(g_{\sigma_0}) = -N(g_{\sigma_0}) <0,	
$
so $E^0_\lambda<0$.
\end{proof}

For our proof of a quantitative version of strict subadditivity of the energy  we need one more ingredient.
\begin{lemma}\label{lem:V-lower-bound}
	$V$ obeys \ref{ass:A2} if and only if for all $ t \ge 1 $ we have
	\begin{align}\label{eq:V-lower-bound}
		V(ta) \ge t^{\gamma_0} V(a) \quad \text{for all } a>0.
	\end{align}
\end{lemma}
\begin{proof}
	Assume that $V$ obeys \ref{ass:A2}. Then
	\begin{align*}
		\frac{d}{dt} V(ta) = V'(ta)a \ge \frac{\gamma_0}{t} V(ta)
	\end{align*}
	for all $a>0$ and $t>1$. Thus
	\begin{align*}
		\frac{d}{dt} (t^{-\gamma_0} V(ta)) \ge 0	
	\end{align*}
	and integrating this yields \eqref{eq:V-lower-bound}.
	Conversely, since \eqref{eq:V-lower-bound} is an equality for $t=1$, we can differentiate it at $t=1$ to get \ref{ass:A2}.
\end{proof}

The lower bound from Lemma \ref{lem:V-lower-bound} will be the main input for the following quantitative version of strict subadditivity of $E_\lambda^\dav$, which in turn will be crucial in the proof of Propositions \ref{prop:fat-tail-bound+} and  \ref{prop:fat-tail-bound0}.

 \begin{proposition}[Strict Subadditivity]\label{prop:strict-subadditivity}
 Under assumptions \ref{ass:A1} and \ref{ass:A2} and
for any  $\lambda>0$, $0<\delta<\lambda/2$, and $ \lambda_1,\ \lambda_2\geq \delta$ with  $\lambda_1+\lambda_2\leq \lambda$, we have
 \bdm
    E_{\lambda_1}^\dav+E_{\lambda_2}^\dav\geq \left[1-(2^{\f{\gamma_0}{2}}-2)\left(\frac{\delta}{\lambda}\right)^{\f{\gamma_0}{2}}\right]E_{\lambda}^\dav,
 \edm
 for $\gamma_0>2$ as in assumption \ref{ass:A2}.
\end{proposition}
\begin{remark}
	Since $1-(2^{\f{\gamma_0}{2}}-2)\left(\frac{\delta}{\lambda}\right)^{\f{\gamma_0}{2}}<1$ for any $\delta>0$ and $\gamma_0>2$, the energy is strictly subadditive whenever $E^\dav_\lambda<0$, since then
	\begin{align*}
		E_{\lambda_1}^\dav+E_{\lambda_2}^\dav> E_{\lambda}
	\end{align*}
	for all $\lambda_1,\lambda_2>0$ with $\lambda_1+\lambda_2=\lambda>0$, by Proposition \ref{prop:strict-subadditivity}.
\end{remark}	
\begin{proof}[Proof of Proposition \ref{prop:strict-subadditivity}:]
First we show that for all $\lambda> 0$ and $0<\mu\leq 1$
\begin{equation} \label{eq:scaling}
E_{\mu\lambda}^\dav\geq\mu^{\f{\gamma_0}{2}} E_{\lambda}^\dav.
\end{equation}
Setting $\widetilde{\lambda}=\mu\lambda$ and $\mu=\rho^{-1}$, one sees that the inequality \eqref{eq:scaling} is equivalent to
\begin{align}\label{eq:scaling-1}
E_{\rho\tilde{\lambda}}^\dav\leq \rho^{\f{\gamma_0}{2}} E_{\tilde{\lambda}} ^\dav   &\quad  \text{ for all } \rho \geq 1 ,\  \tilde{\lambda} > 0
\end{align}
which we are going to prove now:
 Given $f\in H^1(\R)$, or $f\in L^2(\R)$ if $\dav=0$,  with $\|f\|^2=\lambda$ and  $ \rho \geq 1$, we get from Lemma \ref{lem:V-lower-bound}
 \begin{align*}
 N(\rho^{1/2}f)&=\iint _{\R^2} V(\rho^{1/2}|T_r f(x)|)\, dx \psi dr\geq \rho^{\f{\gamma_0}{2}}N(f) .
 \end{align*}
Since $\|\rho^{1/2}f\|^2=\rho \lambda$, $\rho\ge1$ and $\gamma_0>2$ we get
 \begin{align*}
     H(\rho^{1/2}f)\
     \leq\rho \frac{d_{\text{av}}}{2}\|f'\|^2-\rho ^{\f{\gamma_0}{2}}N(f)
     \leq \rho^{\f{\gamma_0}{2}}H(f),
 \end{align*}
which proves \eqref{eq:scaling-1}.

Now let $\lambda_1=\mu_1\lambda$ and $\lambda_2=\mu_2\lambda$ with
$\mu_1+\mu_2\leq 1$ and $\mu_1,\ \mu_2\geq \delta/\lambda$.
Using \eqref{eq:scaling}, we get
 \beq\label{eq:strict-subadditivity-1}
     E_{\lambda_1}^\dav+E_{\lambda_2}^\dav
     =E_{\mu_1\lambda}^\dav+E_{\mu_2\lambda}^\dav
     \geq(\mu_1^{\f{\gamma_0}{2}}+\mu_2^{\f{\gamma_0}{2}})E_\lambda^\dav.
 \eeq
Without loss of generality, we may assume that
$\delta\le \mu_1\leq\mu_2$. Using this and $\mu_1+\mu_2\le 1$ one sees
\begin{equation}\label{eq:nice}
\begin{split}
  \mu_1^{\f{\gamma_0}{2}}+\mu_2^{\f{\gamma_0}{2}}
  &= (\mu_1+\mu_2)^{\f{\gamma_0}{2}}-\left((\mu_1+\mu_2)^\f{\gamma_0}{2} -\mu_1^{\f{\gamma_0}{2}}-\mu_2^{\f{\gamma_0}{2}}\right)\\
  &= (\mu_1+\mu_2)^{\f{\gamma_0}{2}}-\mu_1^{\f{\gamma_0}{2}}\left(\left(1+\f{\mu_2}{\mu_1}\right)^\f{\gamma_0}{2} -1-\left(\f{\mu_2}{\mu_1}\right)^{\f{\gamma_0}{2}}\right)\\
  &\le 1- \mu_1^{\f{\gamma_0}{2}}\left(2^\f{\gamma_0}{2} -2\right)
  	\le 1- \left(\f{\delta}{\lambda}\right)^{\f{\gamma_0}{2}}\left(2^\f{\gamma_0}{2} -2\right)
\end{split}
\end{equation}
where we have also used that the function $t\mapsto(1+t)^{\f{\gamma_0}{2}}-1-t^{\f{\gamma_0}{2}}$ is increasing on $[1,\infty)$.

 Since by Lemma \ref{lem:E-boundedness} we always have $E^\dav_\lambda\le 0$, we can use \eqref{eq:nice} in \eqref{eq:strict-subadditivity-1} to get
 \begin{align*}
     E_{\lambda_1}^\dav+E_{\lambda_2}^\dav
     \geq \left[ 1- \left(\f{\delta}{\lambda}\right)^{\f{\gamma_0}{2}}\left(2^\f{\gamma_0}{2} -2\right) \right] E_\lambda^\dav
 \end{align*}
 which completes the proof.
\end{proof}

\section{The existence proof} \label{sec:existence}
In this section we will characterize when minimizing sequences are precompact modulo tranlations and boosts. Recall the definition of the exponent $\alpha_\delta=\alpha_\delta(\gamma_2)=\frac{4}{10-\gamma_2}+\delta$.
\begin{theorem}\label{thm:existence}
	Let $\lambda>0$ and assume that $V$ obeys \ref{ass:A1} and \ref{ass:A2} and that the density $\psi$ has compact support.
	\begin{theoremlist}
	\item
	If $\dav>0$, $2<\gamma_1\le\gamma_2<10$, and $\psi\in L^{\alpha_\delta}$ for some
	$\delta>0$,  then every minimizing sequence for the variational problem \eqref{eq:min} is precompact modulo translations if and only if
	$E^{\dav}_\lambda<0$.
	\item
	If $\dav =0$, $2<\gamma_1\le\gamma_2<6$, and $\psi\in L^{\f{4}{6-\gamma_2}+\delta}$ for some $\delta>0$, then every minimizing sequence for the variational problem \eqref{eq:min} is precompact modulo translations and boosts if and only if
	$E^{0}_\lambda<0$.
	\end{theoremlist}
	In both cases minimizers of \eqref{eq:min} exist  if $E^{\dav}_\lambda<0$, and these miniminzers are solutions of the dispersion management equation \eqref{eq:GT} for some Lagrange multiplier
	$\omega< 2E^{\dav}_\lambda/\lambda<0$.
\end{theorem}
\begin{remark}
	This theorem shows that compactness modulo translation, respectively modulo translations and boost, for minimizing sequences is equivalent to strict negativity of the energy.
\end{remark}
Key for our proof of Theorem \ref{thm:existence} are the following propositions, which will help to eliminate splitting of minimizing sequences.
First, we introduce notations.
 For $s>0$ and $0<\alpha \leq1$, define
 \beq\label{eq:G}
     G_\alpha(s)\coloneqq\left[(s+1)^{\frac{2\alpha}{1+2\alpha}}-1\right]^{-1/2}.
 \eeq
Note that $G_\alpha$  is a decreasing function on $(0,\infty)$ which vanishes at infinity, which is important for us, and
 \beq\label{eq:G-limit}
     \lim _{s\rightarrow 0^+}G_\alpha(s)=\infty
 \eeq
 which is of less importance.
Moreover, for $x\in\R$, let $x_+\coloneqq\text{max}\{x,0\}$.\\

\begin{proposition}[Fat-tail for positive average dispersion] \label{prop:fat-tail-bound+}
 Assume $V$ obeys \ref{ass:A1} with $2<\gamma_1 \le \gamma_2 <10$ and \ref{ass:A2}, $\dav>0$ and $\psi\in L^{1}$ has compact support.
 Let $\lambda>0$, $f\in H^1$ with $\|f\|^2=\lambda$, and $0<\delta<\lambda/2$, and choose any $a,b\in\R$ with
 \beq\label{eq3}
    \int_{-\infty}^{a}|f(x)|^2dx \geq \delta\;\; \text{and} \;\;\int_{b}^{\infty}|f(x)|^2dx \geq  \delta
 \eeq
then
 \beq\label{eq:fat-tail-bound+}
    H(f)\geq
    \left[1-(2^{\f{\gamma_0}{2}}-2)\left(\f{\delta}{\lambda}\right)^{\f{\gamma_0}{2}}\right]
    E^\dav_\lambda - C\|f\|_{H^1}^2\left(1+\|f\|_{H^1}^8\right) G_{1}\left((b-a-1)_+ \right) ,
 \eeq
 where the constant $C$ depends only on the support and the $L^{1}$ norm of $\psi$.
\end{proposition}

We have a similar bound in the case of vanishing average dispersion.

\begin{proposition}[Fat-tail for zero average dispersion] \label{prop:fat-tail-bound0}
 Assume $V$ obeys \ref{ass:A1} with $2< \gamma_1\le \gamma_2< 6$ and \ref{ass:A2}, $\dav=0$ and $\psi\in L^{\beta(\gamma_2,\tau)}$ has compact support\footnote{Recall the definition of $\alpha(\gamma,\tau)$ and $\beta(\gamma,\tau)$ from Proposition \ref{prop:M-fourier}.}. Let $\lambda>0$, $f\in L^2$ with $\|f\|^2=\lambda$, and $0<\delta<\lambda/2$, and $a,b\in\R$ with either
 \beq\label{eq:time-fat-tail}
    \int_{-\infty}^{a}|f(x)|^2dx \geq \delta\;\; \text{and} \;\;\int_{b}^{\infty}|f(x)|^2dx \geq  \delta
 \eeq
 or
 \beq\label{eq:fourier-fat-tail}
   \int_{-\infty}^a |\hatt{f}(\eta)|^2\, d\eta \ge \delta \text{ and }
   \int_b^\infty |\hatt{f}(\eta)|^2\, d\eta    \ge \delta,
 \eeq
then
 \beq\label{eq:fat-tail-0}
    H(f)\geq \left[1-(2^{\f{\gamma_0}{2}}-2)\left(\f{\delta}{\lambda}\right)^{\f{\gamma_0}{2}}\right]E^0_\lambda - C \lambda(1+\lambda^2)
     G_{\min\{\alpha(\gamma_1,\tau),\alpha ( \gamma_2,\tau)\}}\left((b-a-1)_+ \right)
 \eeq
  where the constant $C$ depends only on the support and the $L^{\beta(\gamma_2,\tau)}$ norm of $\psi$.
\end{proposition}


\bpf[Proof of Proposition \ref{prop:fat-tail-bound+}]
If $b-a\leq 1$, $\eqref{eq:fat-tail-bound+}$ holds immediately since its right hand side is $-\infty$ by \eqref{eq:G-limit}. So now we assume that $b-a>1$. Let $a'$ and $b'$ be arbitrary numbers satisfying $a\leq a'<b'\leq b$ and $b'-a'\ge 1$, which we will suitably choose later.
The estimate of $\|f'\|^2$ is based on a one-dimensional version of the well-known IMS localization formula
 \beq \label{IMS}
    \|f'\|^2=\sum_j\langle(\xi_j f)',(\xi_j f)'\rangle - \sum_j\langle f,|\xi_j'|^2f\rangle
 \eeq
for any collection of functions $\{\xi_j\}$ which are smooth, $0\leq \xi_j\leq 1$, and $\sum_j\xi_j^2=1$. To construct such a partition which suits our needs, consider smooth functions $\{\chi_{j}\}$ that satisfy

\begin{enumerate}[i)]
 \item  $0\leq\chi_j\leq1 $    for  $j=-1,0,1$.
 \item $\displaystyle{\sum_{j=-1}^{1}\chi^2_{j}=1}$.
 \item $\supp \chi_{0} \subset [-\frac{1}{2},\frac{1}{2}],\;\; \chi_0 = 1\;\; \hbox{on} \;\; [-\frac{1}{4},\frac{1}{4}], $ \\
    $\supp \chi_{-1} \subset (-\infty,-\frac{1}{4}], \;\; \chi_{-1} = 1\;\;  \hbox{on} \;\; (-\infty,-\frac{1}{2}], $\\
     $\supp  \chi_{1} \subset [\frac{1}{4},\infty), \;\; \chi_{1} = 1 \;\;\hbox{on} \;\; [\frac{1}{2},\infty).$
\end{enumerate}
Let
 \begin{align*}
    \xi_j(x)=\chi_j \left(\frac{x-\frac{1}{2}(a'+b')}{b'-a'} \right)\ \  \text{for}\  j=-1,0,1.
 \end{align*}
Since $\chi_j'$ is bounded, we see that for some constant $C_1>0$
 \begin{align*}\label{eq:loc-error}
 	  \sum_{j=-1}^1|\xi'_j|^2 \le \frac{C_1 }{(b'-a')^2}.
 \end{align*}
Plugging this into $\eqref{IMS}$ yields
 \beq\label{eq:kinetic}
 \begin{aligned}
    \|f'\|^2\ &\geq\|(\xi_{-1}f)'\|^2+\|(\xi_{0}f)'\|^2+\|(\xi_{1}f)'\|^2-\frac{C_1 \|f\|^2}{(b'-a')^2}\\
    &\geq \|(\xi_{-1}f)'\|^2+\|(\xi_{1}f)'\|^2-\frac{C_1 \|f\|^2}{(b'-a')^2} .
 \end{aligned}
 \eeq
 Now we set  $f_{j}\coloneqq\xi_{j}f$ for $j=-1,1$ and  $f_0\coloneqq f-f_1-f_{-1}=(1-\xi_{-1}-\xi_{1})f $, where we note that $f_0$ is defined differently from $f_{-1}$ and $f_{1}$!

Obviously, $\|f_j\|\leq\|f\|$ for  $j=-1,1$, and since the supports of $\xi_{-1}$ and $\xi_1$ are disjoint also $|f_0|\le |f|$, hence  $\|f_0\|\le \|f\|$.

Set $h\coloneqq f_{-1}+f_1$. Then $f=f_0+h$ and  the bound \eqref{eq:N-splitting3} from Proposition \ref{prop:N-splitting} shows
 \begin{align*}
 	N(f) -  N(f_0) - N(h) \lesssim \|f_0\|\|h\| \left(1+\|f_0\|_{H^1}^8 + \|h\|_{H^1}^8\right)
 \end{align*}
 where we also used $ 1+a^{\gamma_2-2}\lesssim 1+a^8$ for all $a\ge 0$ and $\gamma_2<10$.
 Using Proposition \ref{prop:N-boundedness}, more precisely equation \eqref{eq:consequence2}, which is one of its consequences, we have
 \begin{align*}
 	N(f_0) \lesssim \|f_0\|^2 \left( \|f_0\|_{H^1}^{\gamma_1-2} + \|f_0\|_{H^1}^{\gamma_2-2} \right)
 	\lesssim \|f_0\|^2 \left( 1+\|f_0\|_{H^1}^{8} \right),
 \end{align*}
 and combining the above two bounds we arrive at
 \begin{align}\label{eq:main-split-1}
 	N(f) - N(h) \lesssim \|f_0\|\|f\|(1+\|f\|_{H^1}^8),
 \end{align}
 where used $\|f_0\|,\|h\|\le \|f\|$ and also $\|f_0\|_{H^1}, \|h\|_{H^1}\lesssim \|f\|_{H^1}$, the latter holds because of our smoothness assumptions on the cut-off
 functions $\xi_j$ uniformly in  $b'-a'\ge 1$.

 Since $f_{-1}$ and $f_1$ have supports separated by at least $(b'-a')/2$, \eqref{eq:N-splitting4} gives
 \begin{align}
 	N(h) - N(f_{-1}) - N(f_1)
 	& \lesssim (b'-a')^{-1}\|f_{-1}\|_{H^1}\|f_{1}\|_{H^1} \left(1+\|f_{-1}\|_{H^1}^8 + \|f_{1}\|_{H^1}^8\right) \nonumber \\
 	&\lesssim (b'-a')^{-1}\|f\|_{H^1}^2 \left(1+\|f\|_{H^1}^8\right)\label{eq:main-split-2}
 \end{align}
 where we again  used that, because of our assumption that $b'-a'\ge 1$, the bound $\|f_j\|_{H^1}\lesssim \|f\|_{H^1}$ holds, where the implicit constant does not depend on $a'$ and $b'$.

Combining \eqref{eq:main-split-1} and \eqref{eq:main-split-2}, we get
 \begin{align*}
 	N(f) - N(f_{-1}) - N(f_1)
 	 \lesssim  \left( \|f_0\|\|f\| +  \f{\|f\|_{H^1}^2}{b'-a'} \right)
 	  \left(1+\|f\|_{H^1}^8\right)
 \end{align*}
so when combined with \eqref{eq:kinetic}, this yields
 \begin{align} \label{eq:energy-estimate}
     H(f)-H(f_{-1})-H(f_1)
     \gtrsim
      -\left[ \frac{\|f\|^2}{(b'-a')^2}+\left(\|f_0\|\|f\| +  \f{\|f\|_{H^1}^2}{b'-a'} \right)
 	  \left(1+\|f\|_{H^1}^8\right)\right] .
 \end{align}
To choose $a'$ and $b'$, we use a continuous version of the pigeon hole principle, as in our previous work \cite{HuLee2012}: Let $1\le l \le b-a$ and note that 
 \beq\label{mvt}
    \int_a^{b-l}\int_y^{y+l}|f(x)|^2dxdy \leq \int_a^b \int_{x-l}^x|f(x)|^2dy dx \leq l\|f\|^2.
 \eeq
Moreover, by the mean value theorem, there exists $y'\in (a,b-l)$ such that
 $$
     (b-a-l)\int_{y'}^{y'+l}|f(x)|^2dx=\int_a^{b-l}\int_{y}^{y+l}|f(x)|^2dxd\eta.
 $$
Thus, since $f_0$ has support in $[a',b']$ and $|f_0|\le |f|$, choosing $a'=y'$ and $b'=y'+l$ in the previous identity together with \eqref{mvt} gives $l=b'-a'$ and
 $$
     \|f_0\|^2\leq \|f\id_{[a',b']}\|^2\leq \frac{l}{b-a-l}\|f\|^2.
 $$
Plugging this into  $\eqref{eq:energy-estimate}$ yields
 \begin{align*}
     H(f)- H(f_{-1})-H(f_1)
     &\gtrsim -\left[ \frac{\|f\|^2}{l^2}+\left(\left(\frac{l}{b-a-l}\right)^{1/2}\|f\|^2 +  \f{\|f\|_{H^1}^2}{l} \right)
 	  \left(1+\|f\|_{H^1}^8\right)\right]\\
 	 &\ge -\|f\|_{H^1}^2\left(1+\|f\|_{H^1}^8\right)
       \left[ \frac{1}{l^2}
       +\left(\frac{l}{b-a-l}\right)^{1/2} +  \f{1}{l} \right].
 \end{align*}
Since $\|f\|^2=\lambda$, $\lambda_j=\|f_j\|^2 \geq \delta,\  j=-1,1$ and $\lambda_1+\lambda_2=\|f_{-1}\|^2+\|f_1\|^2 \leq \lambda$, we certainly have
\begin{align*}
	H(f_{-1})+ H(f_1) \ge E^\dav_{\lambda_1} + E^\dav_{\lambda_2}
	\ge \left[1-(2^{\f{\gamma_0}{2}}-2)\left(\f{\delta}{\lambda}\right)^{\f{\gamma_0}{2}}\right]
    E_{\lambda}^\dav
\end{align*}
by Proposition \ref{prop:strict-subadditivity}. Thus we arrive at the bound
 \begin{equation*}
 \begin{split}	
     H(f)- &\left[1-(2^{\f{\gamma_0}{2}}-2)\left(\f{\delta}{\lambda}\right)^{\f{\gamma_0}{2}}\right]
    E_{\lambda}^\dav
     \gtrsim - \|f\|_{H^1}^2\left(1+\|f\|_{H^1}^8\right)
       \left[ \frac{1}{l^2}
       +\left(\frac{l}{b-a-l}\right)^{1/2} +  \f{1}{l} \right]
 \end{split}
 \end{equation*}
for any $0<\delta<\lambda/2$ and all $1\le l\le b-a$.
Now we choose $l=\sqrt[3]{b-a}$. Then $1\le l\leq b-a$ since $b-a\ge 1$, and
 $$
     \max\left\{ \f{1}{l^2},\left(\frac{l}{b-a-l}\right)^{1/2},\f{1}{l} \right\}
     = \left(\frac{l}{b-a-l}\right)^{1/2}
    =\left(\frac{1}{(b-a)^{2/3}-1}\right)^{1/2} = G_{1}((b-a-1)_+)
 $$
which completes the proof.
\end{proof}

\begin{proof}[Proof of Proposition \ref{prop:fat-tail-bound0}]
 Since its proof is very analogous to that of Proposition
 \ref{prop:fat-tail-bound+}, let us mention only the things which need to be changed: In the case of zero average dispersion, the energy contains no
 $\|f'\|^2$ term, hence we do not need to use smooth cut-offs, that is,  we can use $f=f_{-1}+f_0+f_1$ where we set
 $f_{-1}=f\id_{(-\infty,a')}, f_0=f\id_{[a',b']}$ and $f_1=f\id_{(b',\infty)}$, and similarly for $\hat{f}$.

 We can then simply repeat the argument in the proof of \eqref{eq:energy-estimate}, again using \eqref{eq:N-splitting1} but now combined with \eqref{eq:N-splitting2}  instead of \eqref{eq:N-splitting4}, to see that
 \begin{align} \label{eq:energy-estimate-0}
     H(f)-H(f_{-1})-H(f_1)
     & \gtrsim
      -\left(\|f_0\|\|f\| +  \f{\|f\|^2}{(b'-a')^{\min\{\alpha(\gamma_1,\tau),\alpha ( \gamma_2,\tau)\}}} \right)
 	  \left(1+\|f\|^4\right) \nonumber\\
 	 &\ge
 	   - \lambda \left(1+\lambda^2\right)
         \left[ \left(\frac{l}{b-a-l}\right)^{1/2}
                + \f{1}{l^{\min\{\alpha(\gamma_1,\tau),\alpha ( \gamma_2,\tau)\}}}          \right]
 \end{align}
with the only restriction that $l=b'-a' \ge 1$.

If $0<b-a\le 1$, we note that \eqref{eq:fat-tail-0} trivially holds since the right hand side equals $-\infty$. So let $b-a>1$. We choose $l\coloneqq(b-a)^{\f{1}{1+2\min\{\alpha(\gamma_1,\tau),\alpha ( \gamma_2,\tau)\} }}$.
 Then $1<l<b-a$ and $l^{1+2\min\{\alpha(\gamma_1,\tau),\alpha ( \gamma_2,\tau)\}}=b-a>b-a-l>0 $ hence
\begin{align*}
	\left(\frac{l}{b-a-l}\right)^{1/2}
    \ge \f{1}{l^{\min\{\alpha(\gamma_1,\tau),\alpha ( \gamma_2,\tau)\}}} .
\end{align*}
This together with \eqref{eq:energy-estimate-0} and our choice of $G_{\min\{\alpha(\gamma_1,\tau),\alpha ( \gamma_2,\tau)\}} ((b-a-1)_+)$, which satisfies $0<{\min\{\alpha(\gamma_1,\tau),\alpha ( \gamma_2,\tau)\}} \leq1 $, finishes the proof.
\end{proof}

Since the function $G_{\alpha}$ is decreasing on $\R_+$ and vanishes at infinity, similar results to Proposition 2.4 in \cite{HuLee2012} follow from Propositions \ref{prop:fat-tail-bound+} and \ref{prop:fat-tail-bound0}.

\begin{proposition}[Tightness for Positive Average Dispersion]\label{prop:tightness for positive average}
Under the conditions of Theorem \ref{thm:existence} on $V$, $\gamma_1,\gamma_2$, and $\psi$,
let $(f_n)_n\subset H^1(\R)$ be a minimizing sequence for the variational problem \eqref{eq:min} for $\dav>0$ with $\lambda=\|f_n\|^2>0$ and assume
$E_\lambda^\dav <0$. Then there exists $K<\infty$ such that, for any $L>0$,
  \beq\label{eq:tightFourier}
    \sup_{n\in\N} \int_{|\eta|>L} |\hat{f}_n(\eta)|^2\, d\eta \le \frac{K}{L^2}
  \eeq
i.e., the sequence is tight in Fourier space. Moreover, there exist shifts $y_n$ such that
 \beq\label{eq:tightReal}
     \lim_{R\rightarrow\infty} \sup_{n\in{\mathbb N}}\int_{|x|>R}|f_n(x-y_n)|^2dx=0,
 \eeq
 i.e., the shifted sequence is also tight.
\end{proposition}

\begin{proposition}[Tightness for Zero Average Dispersion]\label{prop:tightness for zero average}
Under the conditions of Theorem \ref{thm:existence} on $V$, $\gamma_1,\gamma_2$, and $\psi$, let $(f_n)_n\subset L^2(\R)$ be a minimizing sequence for the variational problem \eqref{eq:min} for $\dav=0$  with $\lambda=\|f_n\|^2>0$
and assume $E_\lambda^0 <0$. Then there exist shifts $y_n$ and boosts
$\xi_n$ such that
 \beq\label{eq:min-seq-tight-fourier}
  \lim_{L\to\infty}\sup_{n\in\N} \int_{|\eta -\xi_n|>L} |\hatt{f}_n(\eta)|^2\, d\eta = 0
 \eeq
and
 \beq\label{eq:min-seq-tight}
  \lim_{R\to\infty}\sup_{n\in\N} \int_{|x -y_n|>R} |f_n(x)|^2\, dx = 0 .
 \eeq
\end{proposition}

\bpf[Proof of Proposition \ref{prop:tightness for positive average}]
 Let $(f_n)_n$ be a minimizing sequence. Lemma \ref{lem:E-boundedness} shows that $\|f_n'\|$ is bounded, that is,
 \bdm
   K\coloneqq  \sup_{n\in\N} \|f_n'\|^2  <\infty.
 \edm
Thus, for every $n \in \N$ and $L>0$, we obtain
 \bdm
   \int_{|\eta|>L} |\hat{f}_n(\eta)|^2\, d\eta \le \int_{|\eta|>L} \frac{|\eta|^2}{L^2} |\hat{f}_n(\eta)|^2\, d\eta \le \int_{\R} \frac{|\eta|^2}{L^2} |\hat{f}_n(\eta)|^2\, d\eta \le \frac{K}{L^2}
 \edm
which is \eqref{eq:tightFourier}.

 To prove the second bound, we follow the argument of \cite{HuLee2012} closely. We give some details for the readers' convenience. Define $a_{n,\delta}$ and $ b_{n,\delta}$ by
 \bdm
   a_{n,\delta}\coloneqq \inf\left\{a\in \R: \int_{-\infty}^{a} |f_n(x)|^2\, dx \ge \delta \right\}
 \edm
 and
 \bdm
   b_{n,\delta}\coloneqq \sup\left\{ b\in \R: \int_b^\infty |f_n(x)|^2\, dx \ge \delta \right\} .
 \edm
 Note that the measure $|f_n(x)|^2\, dx$ is absolutely continuous with respect to Lebesgue measure and hence
 \bdm
  \int_{-\infty}^{a_{n,\delta}} |f_n(x)|^2\, dx = \delta \quad\text{and }
    \int_{b_{n,\delta}}^{\infty} |f_n(x)|^2\, dx = \delta.
 \edm
 Furthermore $\delta\mapsto a_{n,\delta}$ and $\delta\mapsto b_{n,\delta}$ are monotone, more precisely,
 for $0<\delta_2<\delta_1<\lambda/2$ one has $a_{n,\delta_2}\le a_{n,\delta_1}$ and
 $b_{n,\delta_2}\ge b_{n,\delta_1}$.
 Let $R_{n,\delta}\coloneqq b_{n,\delta}-a_{n,\delta}$ and note that the above monotonicity yields $R_{n,\delta_2}\ge R_{n,\delta_1}$ for  $0<\delta_2<\delta_1<\lambda/2$.  Lastly, for some fixed $0<\delta_0<\lambda/2$ put
 \bdm
   y_n\coloneqq \frac{b_{n,\delta_0}+a_{n,\delta_0}}{2} \in [a_{n,\delta_0}, b_{n,\delta_0}].
 \edm
 In particular, $a_{n,\delta}\le a_{n,\delta_0}\le y_n\le b_{n,\delta_0}\le b_{n,\delta}$ for all $0<\delta\le \delta_0$.
 This implies $b_{n,\delta} - y_n  \le b_{n,\delta} - a_{n,\delta} = R_{n,\delta}$ and
 $ y_n- a_{n,\delta} \le b_{n,\delta} - a_{n,\delta} = R_{n,\delta}$, hence we are guaranteed that
 \beq\label{eq:inclusion}
   [a_{n,\delta}, b_{n,\delta}]\subset [y_n-R_{n\delta}, y_n+R_{n,\delta}].
 \eeq
 Now assume that
 \beq\label{eq:Rfinite}
   R_\delta \coloneqq \sup_{n\in\N} R_{n,\delta} <\infty
 \eeq
 for $0<\delta\le \delta_0$ and put $R_\delta\coloneqq R_{\delta_0}$ for $\delta_0<\delta<\lambda/2$. Then
 \eqref{eq:inclusion} yields
 \bdm
   \int_{|x-y_n|>R_\delta} |f_n(x)|^2\, dx
   \le \int_{-\infty}^{a_{n,\delta}} |f_n(x)|^2\, dx
        + \int_{b_{n,\delta}}^{\infty} |f_n(x)|^2\, dx  =2\delta
 \edm
 for all $0<\delta\le \delta_0$ but the same bound also holds when $\delta_0<\delta<\lambda/2$ since in this case
  \bdm
   \int_{|x-y_n|>R_\delta} |f_n(x)|^2\, dx =  \int_{|x-y_n|>R_{\delta_0}} |f_n(x)|^2\, dx \le 2\delta_0 <2\delta.
  \edm

 It remains to show \eqref{eq:Rfinite}: Recall that $K\coloneqq  \sup_{n\in\N} \|f_n'\|^2  <\infty$,  set $\wti{K}= \sqrt{\lambda +K}$, and note that  the minimizing sequence $f_n$ obeys
 $\|f_n\|_{H^1}\le \wti{K}$ for all $n\in\N$.
Using $b= b_{n,\delta}$ and $a= a_{n,\delta}$, rearranging \eqref{eq:fat-tail-bound+} from
 Proposition \ref{prop:fat-tail-bound+} yields
 \bdm
 E_\lambda^{\dav}-(2^{\f{\gamma_0}{2}}-2)\left(\f{\delta}{\lambda}\right)^{\f{\gamma_0}{2}}E_\lambda^{\dav} -H(f_n) \le C\wti{K}^2(1+\wti{K}^8) G_{1}\left((R_{n,\delta}-1)_+ \right)  .
 \edm
 Thus, since $H(f_n)\to E_{\lambda}^\dav<0$,
 \bdm
 0 < -(2^{\f{\gamma_0}{2}}-2)\left(\f{\delta}{\lambda}\right)^{\f{\gamma_0}{2}}E_\lambda^{\dav} \le C\wti{K}^2(1+\wti{K}^8)\liminf_{n\to\infty}G_{1}\left((R_{n,\delta}-1)_+\right).
 \edm
 Since $G_1$ is monotone decreasing, we get
 \bdm
   G_1((\limsup_{n\to\infty}R_{n,\delta}-1)_+) = \liminf_{n\to\infty}G_1((R_{n,\delta}-1)_+) >0
 \edm
 and so
 \bdm
   \limsup_{n\to\infty} R_{n,\delta}<\infty.
 \edm
 Hence \eqref{eq:Rfinite} holds.
\end{proof}

\begin{proof}[Proof of Proposition \ref{prop:tightness for zero average}]
Using the fact that the function $G_{\alpha}$ is monotone decreasing, the proof is virtually identical to the proof of the bound \eqref{eq:tightReal} of
Proposition \ref{prop:tightness for positive average}.
\end{proof}

To prove Theorem \ref{thm:existence}, we need one more result for the continuity of the nonlinear functional $N(f)$.

\begin{lemma} \label{lem:continuity}
\begin{theoremlist}
  \item If $2\le \gamma_1\le \gamma_2\le 6$ and $\psi\in L^{\f{4}{6-\gamma_2}}$ then the nonlinear nonlocal functional $N: L^2(\R)\to \R$ given by
\begin{align*}
L^2(\R)\ni f\mapsto N(f)=
\iint_{\R^2}V(|T_r f|)dx\psi dr
\end{align*}
 is locally Lipshitz continuous on $L^2$ in the sense that
 \begin{align*}
   	|N(f_1)-N(f_2)| \lesssim \left(1+ \|f_1\|^{\gamma_2-1} + \|f_2\|^{\gamma_2-1}\right)
   	\|f_1-f_2\|
 \end{align*}
  where the implicit constant depends only on the
  $L^{\f{4}{6-\gamma_2}}$ norm of $\psi$.
 \item If $2\le\gamma_1\le\gamma_2<\infty$ and $\psi\in L^1$, then the nonlinear nonlocal functional $N: H^1(\R)\to \R$ given by
\begin{align*}
H^1(\R)\ni f\mapsto N(f)=
\iint_{\R^2}V(|T_r f|)dx\psi dr
\end{align*}is locally Lipschitz continuous in the sense that
 \begin{align*}
 |N(f_1)-N(f_2)| \lesssim \left(1+ \|f_1\|_{H^1}^{\gamma_2-2} + \|f_2\|_{H^1}^{\gamma_2-2}\right)\left( \|f_1\|+ \|f_2\|\right)
 \|f_1-f_2\| .
 \end{align*}
\end{theoremlist}
\end{lemma}
\begin{remark}
	Note that the second part of Lemma \ref{lem:continuity} shows that the Lipschitz constant of $N$ on $H^1$ depends on the $H^1$ norm, however, if $f_1 $ and $f_2$ are bounded in $H^1$, then the difference $N(f_1)-N(f_2)$ is small whenever $f_1$ is close to $f_2$ in the much weaker $L^2$ norm!
\end{remark}
\begin{proof}
  We always have
  \begin{align*}
  	|N(f_1)-N(f_2)| \le \iint_{\R^2} \left| V(|T_rf_1|) - V(|T_rf_2|) \right| \, dx \psi dr
  \end{align*}
  and from \eqref{eq:V-splitting-1} we see that
  \begin{align*}
    	| V(|T_rf_1|) - V(|T_rf_2|)| \lesssim
    	|T_r(f_1-f_2)| \left( (|T_rf_1|+|T_rf_2|)^{\gamma_1-1} + (|T_rf_1|+|T_rf_2|)^{\gamma_2-1} \right)
  \end{align*}
 so
 \begin{equation}
 \begin{split}\label{eq:basic}
 	|N(f_1)-N(f_2)|
 	\lesssim &  ~~\|T_r(f_1-f_2) (|T_rf_1|+|T_rf_2|)^{\gamma_1-1}\|_{L^1(\R^2,\,dx\psi dr)}\\
 	&~~ + \|T_r(f_1-f_2) (|T_rf_1|+|T_rf_2|)^{\gamma_2-1}\|_{L^1(\R^2,\,dx\psi dr)} .
 \end{split}
 \end{equation}
If $2\le \gamma\le 6$ we can  use H\"older's inequality for 2 functions with exponents
$\gamma$, and $\gamma/(\gamma-1)$ to bound
	\begin{align*}
		\|T_r(f_1-f_2) (|T_rf_1|+ &|T_rf_2|)^{\gamma-1}\|_{L^1(\R^2,\,dx\psi dr)}
		\le \\
		&
		 \|T_r(f_1-f_2)\|_{L^\gamma(\R^2, dx\psi dr)}
		 \||T_rf_1|+|T_rf_2|\|_{L^\gamma(\R^2, dx\psi dr)}^{\gamma-1}.
	\end{align*}
	Applying the triangle inequality and Lemma \ref{lem:boundedness} then yields
	\begin{align*}
		\|T_r(f_1-f_2) (|T_rf_1|+ |T_rf_2|)^{\gamma-1}\|_{L^1(\R^2,\,dx\psi dr)}
		\lesssim
		 \|f_1-f_2\| \left(\|f_1\|^{\gamma-1}+\|f_2\|^{\gamma-1}\right)
	\end{align*}
  where the implicit constant depends only on the $L^{\f{4}{6-\gamma}}$ norm of $\psi$.
  Using these bounds in \eqref{eq:basic} shows that
  \begin{align*}
  	|N(f_1)-N(f_2)|&\lesssim \left( \|f_1\|^{\gamma_1-1}+\|f_2\|^{\gamma_1-1} +\|f_1\|^{\gamma_2-1}+\|f_2\|^{\gamma_2-1}\right) \|f_1-f_2\| \\
  	&\lesssim \left( 1+\|f_1\|^{\gamma_2-1}+\|f_2\|^{\gamma_2-1}\right) \|f_1-f_2\|
  \end{align*}
  where we also used that $ a^{\gamma_1-1}+a^{\gamma_2-1}\lesssim 1+a^{\gamma_2-1}$ for all $a\ge 0$ when $2\le\gamma_1\le\gamma_2$.
  This proves the first part of Lemma \ref{lem:continuity} when $2\le \gamma_1\le\gamma_2\le 6$.

  To prove the second part of the Lemma, we have to bound the terms in
  \eqref{eq:basic} slightly differently. For $\gamma\ge 2$ we use the Cauchy--Schwarz inequality to get
  \begin{align*}
  	\|T_r(f_1-f_2) (|T_rf_1|+&|T_rf_2|)^{\gamma-1}\|_{L^1(\R^2,\,dx\psi dr)}
  	\\
  	&\le \|T_r(f_1-f_2)\|_{L^2(\R^2,\,dx\psi dr)} \|(|T_rf_1|+|T_rf_2|)^{2(\gamma-1)}\|_{L^1(\R^2,\,dx\psi dr)}^{1/2} \\
  	&\le \|\psi\|_{L^1}\|f_1-f_2\| \|(|T_rf_1|+|T_rf_2|)^{2(\gamma-1)}\|_{L^1(\R^2,\,dx\psi dr)}^{1/2}
  \end{align*}
 using Lemma \ref{lem:boundedness}. Since $\gamma\ge 2$, we can further split
 \begin{align*}
 	\|(|T_rf_1|+|T_rf_2|)^{2(\gamma-1)}&\|_{L^1(\R^2,\,dx\psi dr)}^{1/2}
 	\\
 	&\le \sup_{r\in\R }\left( \|T_rf_1\|_{L^\infty} + \|T_rf_2\|_{L^\infty} \right)^{\gamma-2}
 		\||T_rf_1|+|T_rf_2|\|_{L^2(\R^2,\,dx\psi dr)} \,  .
 \end{align*}
 Because of   \eqref{eq:nice2} the first factor is bounded by $(\|f_1\|_{H^1} + \|f_2\|_{H^1})^{\gamma-2}$ and using Lemma \ref{lem:boundedness} and the triangle inequality, the second factor is bounded by $\|\psi\|_{L^1}(\|f_1\|+\|f_2\|)$.
 Using this in \eqref{eq:basic} proves the second part of the lemma for $2\le\gamma_1\le\gamma_2<\infty$.
\end{proof}
\begin{lemma} \label{lem:differentiability}
If $2\le \gamma_1\le\gamma_2\le 6$ and $\psi\in L^1\cap L^{\f{4}{6-\gamma_2}}$, respectively if $2\le\gamma_1\le\gamma_2<\infty$ and $\psi\in L^1$, then for any $f,h \in L^2(\R)$, respectively $f,h\in H^1(\R)$, the functional $N$ as above has directional derivative
given by
 \begin{align*}
  D_hN(f)=
 	   \int_\R \re\left\langle T_r h,V'(|T_rf|)\sgn{T_rf} \right\rangle \,\psi dr.
 \end{align*}
\end{lemma}

\begin{proof}
Let $f \in L^2(\R)$ and $\epsilon \neq 0$. Fix any $h\in L^2(\R)$ and  the quotient of $N$ is
\begin{align}\label{eq:derivative}
\f{N(f+\epsilon h)-N(f)}{\epsilon}&=\f{1}{\epsilon}\left[\iint_{\R^2}V(|T_r(f+\epsilon h)|)-V(|T_r f|)dx \psi dr\right]\notag\\
& =\f{1}{\epsilon} \iint_{\R^2}\int _0^1 \f{d}{ds} V(|T_r(f+s\epsilon h)|) ds dx \psi dr.
\end{align}
By straightforward calculations, we obtain
\begin{align}
\f{d}{ds} V(|T_r(f+s\epsilon h)|)=V ' (|T_r (f+s \epsilon h )|) \f{\epsilon (T_rf \overline{T_rh}+T_rh \overline{T_rf}+2s\epsilon |T_r h|^2)}{2|T_r (f+s\epsilon h)|}\notag
\end{align}
and thus
\begin{align*}
\eqref{eq:derivative}=\iint_{\R^2} \int _0^1 V ' (|T_r (f+s \epsilon h )|) \f{T_rf \overline{T_rh}+T_rh \overline{T_rf}+2s\epsilon |T_r h|^2}{2|T_r (f+s\epsilon h)|}  ds dx \psi dr.
\end{align*}
By Lebesgue's dominated convergence theorem, letting $\epsilon \to 0$, we get
\begin{align*}
D_hN(f)=\iint_{\R^2} \int _0^1 V ' (|T_r f|) \f{\re (T_rf \overline{T_rh})}{|T_r f|}  ds dx \psi dr
=\iint_{\R^2} V ' (|T_r f|) \f{\re (T_rf \overline{T_rh})}{|T_r f|}  dx \psi dr
\end{align*}
which completes the proof when $2\le \gamma_1\le\gamma_2\le6$. The case $2\le\gamma_1\le\gamma_2<\infty$ is similar.
\end{proof}

Now we are ready to give the
\begin{proof}[Proof of Theorem \ref{thm:existence}:]
 It is easy to see that if $E^\dav_\lambda=0$ for some $\lambda>0$, then there are minimizing sequences which are not precompact modulo translations. Indeed, assume that $E^\dav_\lambda=0$ and let $g_n$ be the centered Gaussian from \eqref{eq:Gauss} with the choice $\sigma_0=n\in\N$. This gives a sequence of Gaussians which weakly converges to zero and no translates or boosts of $g_n$ converges strongly and, as the proof of Lemma \ref{lem:E-boundedness} shows, we have $H(g_n)\to 0=E^\dav_\lambda$ as $n\to\infty$.  By contrapositive, this is equivalent to that  if every sequence is modulo translations, respectively modulo translations and boosts, then necessarily $E^\dav_\lambda<0$.

 For the converse, assume that $E^\dav_\lambda<0$ and assume further that $\dav>0$.
Let $(f_n)_n\subset H^1(\R)$ be a minimizing sequence of the variational problem \eqref{eq:min}. Since $\|f_n\|^2=\lambda>0$ is fixed we can use Lemma
 \ref{lem:E-boundedness} to see that $f_n$ is bounded in the $H^1$ norm,
 \begin{align*}
  K_1\coloneqq\sup_{n\in\N}\|f_n\|_{H^1}<\infty.
 \end{align*}

 In addition, applying Proposition \ref{prop:tightness for positive average}, there exist shifts $y_n$ such that for the
 shifted sequence $h_n$,  $h_n(x)\coloneqq f_n(x-y_n)$ for $x\in \R$, we have
 \beq\label{eq:tightReal2}
     \lim_{R\rightarrow\infty} \sup_{n\in{\mathbb N}}\int_{|x|>R}|h_n(x)|^2dx=0.
 \eeq
 Clearly, by translation invariance of Lebesgue measure, we still have $\|h_n\|^2=\lambda$.
 On the Fourier side, shifts correspond to modulations
with $e^{iy_n\eta}$, so for the shifted sequence $h_n$ Proposition \ref{prop:tightness for positive average} also yields that there exists $K_2 < \infty$ such that for any $L>0$
  \beq\label{eq:tightFourier2}
    \sup_{n\in\N} \int_{|\eta|>L} |\hat{h}_n(\eta)|^2\, d\eta \le \frac{K_2}{L^2}.
  \eeq
  Thus, by translation invariance of the minimization problem, the shifted sequence
 is a minimizing sequence for $E^\dav_\lambda$  which is \emph{tight} in the sense of Lemma \ref{lem:strong-convergence-L2}.
 The shifted sequence $h_n$ is certainly also bounded in $H^1$, hence also bounded in $L^2$. By the weak sequential compactness of bounded sets in $L^2$ and $H^1$, we can extract a subsequence, which by some slight abuse of notation, we still denote by $h_n$, which converges weakly both to some $f$ in $L^2$ and some $\wti{f}$ in  $H^1$. By uniqueness of weak limits, we must have $f=\wti{f}$ and by the characterization of strong convergence in $L^2$ from  Lemma \ref{lem:strong-convergence-L2}, we know that $h_n$ converges \emph{even strongly} in $L^2$ to $f$. In particular,
 \begin{align*}
   \|f\|^2=\lim_{n\to\infty}\|h_n\|^2 =\lambda>0	
 \end{align*}
 so $f\not=0$.   Since $H^1$ is a Hilbert space, we also have the weak sequential lower semi--continuity of the $H^1$ norm, that is,
  \beq
   \|f\|_{H^1} \le \liminf_{n\to\infty} \|h_n\|_{H^1}
 \eeq
 Since $\|f\|^2=\lambda$ and $\|h_n\|^2=\lambda$ this implies
 \begin{align*}
   \|f'\|^2 \le \liminf_{n\to\infty} \|h_n'\|^2, 	
 \end{align*}
 that is, the kinetic energy is lower semi--continuous.

 Since $h_n$ is bounded in $H^1$ and $h_n$ converges strongly to $f$ in $L^2$, Lemma \ref{lem:continuity} shows that
 \begin{align*}
   \lim_{n\to\infty} N(h_n) = N(f). 	
 \end{align*}
 Together with the lower semi--continuity of the kinetic energy this implies
 \begin{align*}
 	E^\dav_\lambda\le H(f) \le \liminf_{n\to\infty} H(f_n) = E^\dav_\lambda
 \end{align*}
 and since $\|f\|^2=\lambda$, this shows that $f$ is a minimizer for the variational problem \eqref{eq:min}.

It remains to show that the existence of a minimizer of \eqref{eq:min} for $\dav=0$.
Again, the main task is to use translations and boosts to massage an arbitrary minimizing sequence into one having a strongly convergent subsequence.

Let $(f_n)_n\subset L^2(\R)$ be an arbitrary
minimizing sequence of the variational problem \eqref{eq:min} with $\|f_n\|^2=\lambda>0$.
Proposition \ref{prop:tightness for zero average} guarantees the existence of shifts
$y_n\in\R$ and boosts $\xi_n\in\R$ such that \eqref{eq:min-seq-tight-fourier} and
\eqref{eq:min-seq-tight} hold.
Define the shifted and boosted sequence $(h_n)_n=(f_{\xi_n,y_n,n})_n$ by
    \bdm
        h_n(x)= f_{\xi_n,y_n,n}(x)\coloneqq e^{i\xi_nx} f_n(x-y_n) \quad \text{for } x\in\R.
    \edm
Note that $\|h_n\|^2= \|f_n\|^2=\lambda$ since shifts and boost are unitary
operations on $L^2(\R)$ and $N(f_n)=N(h_n)$, see Appendix \ref{sec:Galilei}.
Hence $(h_n)_n$ is also a minimizing sequence.
Certainly $|h_n(x)|= |f_n(x-y_n)|$ for all $n\in\N$. The Fourier transform of $h_n$ is given by
 \beq
    \hatt{h}_n(\eta)= \frac{1}{\sqrt{2\pi}} \int e^{-ix\eta} e^{ix\xi_n} f_n(x-y_n)\, dx
    = e^{-iy_n\eta} \hatt{f}_n(\eta-\xi_n).
 \eeq
Thus also $|\hatt{h}_n(\eta)|= |\hatt{f}_n(\eta-\xi_n)|$.  In particular,
\eqref{eq:min-seq-tight-fourier} and \eqref{eq:min-seq-tight} show that the minimizing
sequence $(h_n)_n$  is tight in the sense of Lemma \ref{lem:strong-convergence-L2}.

Since $(h_n)_n$ is bounded in $L^2(\R)$, the weak compactness of the unit ball,
guarantees the existence of a weakly converging subsequence of $(h_n)_n$, denoted again by $(h_n)_n$.
 Obviously, this subsequence is also tight in the
sense of  Lemma \ref{lem:strong-convergence-L2} and thus hence converges even strongly in $L^2(\R)$.
We set
    \bdm
        f= \lim_{n\to\infty} h_n .
    \edm
By strong convergence $ \|f\|^2 = \lim_{n\to\infty} \|h_n\|^2 = \lambda $.
To conclude that $f$ is the sought after minimizer we note that by Lemma
\ref{lem:continuity} the map $f\mapsto N(f)=\iint_{\R^2}V(|T_r f|)dx\psi dr$ is continuous on $L^2(\R)$.
Hence
    \bdm
        E^0_\lambda\le H(f)= -N(f)= \lim_{n\to\infty}-N(h_n)
        = E_\lambda^0
    \edm
where the last equality follows since $(h_n)_n$ is a minimizing sequence for \eqref{eq:min}.
Thus $f$ is a minimizer for the variational problem \eqref{eq:min}.

To prove that the above minimizer is a weak solution of the associated Euler-Lagrange equation \eqref{eq:GT} is standard in the calculus of variations. One has to be a bit careful here, since we only have the directional derivative of $N$ and hence of the energy functional $H$ at our disposal.
 Let $f$ be a minimizer of \eqref{eq:min}.

 Recall  $H(f)= \f{\dav}{2}\|f'\|^2 -N(f)$. By
Lemma \ref{lem:differentiability} the directional derivative of the functional of $H$ at $f\in H^1$ in direction $h\in H^1$ is given by
\begin{align*}
	D_hH(f) = \dav\re\la h',f' \ra -D_hN(f) .
\end{align*}
Similarly, one can check that the derivative of $\varphi(f) \coloneqq \|f\|^2= \la f,f\ra$ is given by
$D_h \varphi(f)= 2\re\la h,f\ra$.

Now let $f$ be any minimizer of the constraint variational problem \eqref{eq:min} and
$h\in H^1$ arbitrary. Define, for any $(s,t)\in\R^2$,
 \begin{align*}
    F(s,t)&\coloneqq H(f  + sf  +t h) , \\
    G(s,t)&\coloneqq \varphi (f  + sf   +t h ) .
 \end{align*}
Note that
 \begin{align*}
    \nabla F(s,t)&= \left(\begin{array}{c} D_f H(f+ sf +t h) \\ D_hH(f+ sf +t h) \end{array} \right)
 \end{align*}
and
 \bdm
    \nabla G(s,t)= \left(\begin{array}{c} D_f\varphi(f+ sf +t h) \\ D_h\varphi(f+ sf +t h) \end{array} \right)
    = 2\left(\begin{array}{c} \re\la f,f+ sf +t h\ra \\
                              \re \la h,f+ sf +t h\ra  \end{array} \right) .
 \edm
Since $\la f,f\ra = \|f\|^2=\lambda\not= 0$,
 \bdm
    \nabla G(0,0)= 2\left(\begin{array}{c} \la f,f\ra \\
                              \re \la h,f\ra  \end{array} \right)
                 = 2\left(\begin{array}{c} \lambda \\
                              \re \la h,f\ra  \end{array} \right)
 \edm
 is not the zero vector in $\R^2$ and since $\nabla G(s,t)$ depends multi-linearly, in particular
 continuously, on $(s,t)$, the implicit function theorem \cite{Strichartz00} shows that there
 exists an open interval $I\subset\R$ containing $0$ and a differentiable function $\phi$ on $I$
 with $\phi(0)=0$ such that
 \bdm
    \lambda= \|f\|^2 = G(0,0) = G(\phi(t),t)
 \edm
 for all $t\in I$. Consider the function $I\ni t\mapsto F(\phi(t),t)$. Since $f$ is a minimizer for the constraint variational problem \eqref{eq:min}, $F(\phi(t),t)$ has a local minimum at
 $t=0$. Hence, using the chain rule,
 \bdm
    0 = \frac{d F(\phi(t),t)}{dt}\Big\vert_{t=0}
        = \nabla F(0,0)\cdot\left(\begin{array}{c}\phi'(0)\\1\end{array}\right)
        = D_f H(f) \phi'(0) + D_h H(f) .
 \edm
 Since $\lambda= G(\phi(t),t)$, the chain rule also yields
 \bdm
    0= \frac{d G(\phi(t),t)}{dt}\Big\vert_{t=0}
        = \nabla G(0,0)\cdot\left(\begin{array}{c}\phi'(0)\\1\end{array}\right)
        = 2 \la f,f\ra \phi'(0) + 2\re \la h,f\ra.
 \edm
 Solving this for $\phi'(0)$ and plugging it back into the expression for the derivative of $F$,
 we see that
 \bdm
    \frac{D_f H(f) }{\la f,f \ra} \re \la h,f\ra = D_h H(f).
 \edm
 In other words, with
 \begin{align}\label{eq:lagrange multiplier}
   \omega \coloneqq \frac{D_f H(f) }{\lambda}
 \end{align}
 and $f$ any minimizer of
 \eqref{eq:min} we have
 \beq\label{eq:euler-lagrange-real}
     \re(\omega\la h,f\ra) = D_h H(f) = \re\left(\dav\la h',f' \ra
       - \int_\R \left\langle T_r h,V'(|T_rf|)\sgn{T_rf} \right\rangle \,\psi dr\right)
 \eeq
 for any $h\in H^1$, using the formula for $D_hN(f)$ from Lemma \ref{lem:differentiability}. Replacing $h$ by $ih$ in \eqref{eq:euler-lagrange-real}, one gets
 \beq\label{eq:euler-lagrange-imaginary}
         \im(\omega\la h,f\ra) = \im\left(\dav\la h',f' \ra
       - \int_\R \left\langle T_r h,V'(|T_rf|)\sgn{T_rf} \right\rangle \,\psi dr\right)
 \eeq
 for all $h\in H^1$. \eqref{eq:euler-lagrange-real} and \eqref{eq:euler-lagrange-imaginary} together
 show
 \bdm
         \omega\la h,f\ra = \dav\la h',f' \ra
       - \int_\R \left\langle T_r h,V'(|T_rf|)\sgn{T_rf} \right\rangle \,\psi dr
 \edm
 for any $h\in L^2(\R)$, that is, $f$ is a weak solution of the
 dispersion management equation \eqref{eq:GT}.

It remains to prove $\omega< 2E^{\dav}_\lambda/\lambda$. For this, recall that assumption \ref{ass:A2} states that
\begin{align*}
	V'(a)a \ge \gamma_0 V(a) \quad \text {for all } a>0. 	
\end{align*}
Thus
\begin{align*}
	D_f N(f)
	= \iint_{\R^2} V'(|T_rf|)|T_rf| \,dx\psi dr
	\ge \gamma_0 \iint_{\R^2} V(|T_rf|) \,dx\psi dr  =\gamma_0 N(f)
\end{align*}
and since $E^{\dav}_\lambda<0$, we must have $N(f)>0$ for any minimizer
$f$, so \eqref{eq:lagrange multiplier} gives
\begin{align*}
	\omega(f)\lambda= D_f H(f)
	&= \dav\la f',f'\ra -D_f N(f)
	\le \dav\la f', f'\ra -\gamma_0 N(f) \\
	&= 2H(f) -(\gamma_0-2) N(f) < 2H(f) = 2E^{\dav}_\lambda <0
\end{align*}
for every minimizers.
\end{proof}

\section{Threshold phenomena}\label{sec:threshold}
As we showed in the previous section, assumptions \ref{ass:A1} and \ref{ass:A2} guarantee the existence of minimizers for arbitrary $\lambda>0$ and $\dav\ge 0$ as soon as the ground state  energy $E^{\dav}_\lambda$ is \emph{strictly negative}.  In this section we will prove there exists a threshold $0\le\lambda^\dav_{\mathrm{cr}}\le \infty$ such that solutions exist if the power $\lambda= \|f\|^2>\lambda^\dav_{\mathrm{cr}}$.
Furthermore $\lambda^\dav_{\mathrm{cr}}<\infty$ if assumption \ref{ass:A3} holds.

For pure power law nonlinearities and the model case $d_0=\id_{[0,1)}-\id_{[1,2]}$ for the diffraction profile, this had been partly investigated earlier in \cite{Moeser05} for
the diffraction management equation and for pure power nonlinearities in \cite{weinstein} for the discrete nonlinear Schr\"odinger equation. We are not aware of any work which investigates threshold phenomena for general nonlinearities for dispersion management solitons in the continuum.

 In the following we will always assume, without explicitly mentioning it every time,
 that $\psi$ is a probability density on $\R$ with compact support, that is, there exists $0<R<\infty$ such that $\supp(\psi)\subset [-R,R]$ together with further $L^p$ properties, depending on the range of $\gamma_1\le\gamma_2$.
Our main result in this section is
\begin{theorem}[Threshold phenomenon]
 \label{thm:threshold-phenomena}
	Assume that $V$ obeys \ref{ass:A1} for some $2<\gamma_1 \le \gamma_2<10$ if $\dav>0$ and some $2<\gamma_1 \le \gamma_2<6$ if $\dav=0$ and also  \ref{ass:A2}. Then
	\begin{theoremlist}
	\item\label{thm:threshold-phenomena-1} The map $\lambda\mapsto E^{\dav}_\lambda$ is decreasing on $(0,\infty)$, strictly decreasing where $E^\dav_\lambda<0$, and there exists  a critical threshold
	$0\le \lambda_{\mathrm{cr}}^\dav\le \infty$ such that for
	$0 < \lambda < \lambda_{\mathrm{cr}}^\dav$ we have $E^{\dav}_\lambda = 0$ and for $\lambda>\lambda_{\mathrm{cr}}^\dav$ we have $-\infty<E^{\dav}_\lambda < 0$.
	\item\label{thm:threshold-phenomena-2} 	If $\lambda> \lambda_{\mathrm{cr}}^\dav$, then minimizers of \eqref{eq:min} exist and any minimizing sequence is, up to translations, precompact in $L^2(\R)$ if $\dav>0$, respectively, precompact up to translations and boosts in $L^2$ if $\dav=0$. In both cases the suitably translated, respectively translated and boosted, minimizing sequence has  a subsequence which converges to a minimizer.
	\item\label{thm:threshold-phenomena-3} If $0< \lambda< \lambda_{\mathrm{cr}}^\dav$ and $\dav>0$, no minimizers of the variational problem \eqref{eq:min} exist.
	\item\label{thm:threshold-phenomena-4} If $6\le \gamma_1\le\gamma_2<10$, then
		$\lambda_{\mathrm{cr}}^\dav>0$ for all $\dav>0$.
	\item\label{thm:threshold-phenomena-5} If there exists
		$f\in H^1(\R)$ such that $N(f)>0$, then one has  $\lambda_{\mathrm{cr}}^\dav< \infty$ for all $\dav\ge  0$.
	\end{theoremlist}
\end{theorem}

\begin{remarks}
 The precise definition of  $\lambda_{\mathrm{cr}}^\dav$ is given below in Definition \ref{def:threshold}.
 When $\lambda>\lambda_{\mathrm{cr}}^\dav$ we have $E^{\dav}_\lambda<0$ and Theorem \ref{thm:existence} shows that every minimizing sequence is precompact modulo translations if $\dav>0$, respectively, precompact modulo translations and boosts
 if $\dav=0$, and thus minimizers exist yielding solutions of \eqref{eq:GT} for some Lagrange multiplier $\omega < 2E^{\dav}_\lambda/\lambda<0$.

 Since $E^{\dav}_\lambda=0$ when $0<\lambda<\lambda_{\mathrm{cr}}^\dav$, Theorem \ref{thm:existence} also shows that there are minimizing sequences which are not precompact modulo translations and boosts in this case. Nevertheless, it could be that minimizers still exist. At least when $\dav>0$, Theorem \ref{thm:threshold-phenomena} shows that this cannot be the case.
 At the moment, we need $\dav>0$ to conclude nonexistence of minimizers when $0<\lambda<\lambda_{\mathrm{cr}}$.
\end{remarks}

We give the proof of Theorem \ref{thm:threshold-phenomena} at the end of this section after
some preparations.
Recall the definition of the energy functional
\begin{align*}
	H(f) = \frac{\dav}{2}\|f'\|^2  - N(f)
\end{align*}
and
\begin{align*}
	E^{\dav}_\lambda = \inf\{H(f):\, f\in H^1(\R), \|f\|^2=\lambda\} .
\end{align*}
For strictly positive average dispersion, we write  $f\in H^1(\R)$ with $\lambda=\|f\|^2>0$ as $f=\sqrt{\lambda} h$.  Then $h\in H^1(\R)$ with $\|h\|=1$ and
\begin{align}\label{eq:energy-threshold}
	H(f) = \frac{\dav}{2}\|f'\|^2  - N(f)
		= \|f'\|^2  \left( \frac{\dav}{2}- \frac{N(\sqrt{\lambda}h)}{\lambda \|h'\|^2} \right).
\end{align}
In the case of vanishing average dispersion, we simply write, again for $f\in H^1$,
\begin{align*}
	H(f)= -N(f) = - \|f'\|^2 \left(\frac{N(\sqrt{\lambda}h)}{\lambda\|h'\|^2} \right),
\end{align*}
so defining\footnote{Note that the null space of $\partial_x$ on $H^1(\R)$ is trivial, so $R(\lambda,h)$ is well defined for any $h\in H^1\setminus\{0\}$.} for $\dav\ge 0$
\begin{align*}
	R(\lambda,h)\coloneqq \frac{N(\sqrt{\lambda}h)}{\lambda \|h'\|^2}
	\quad \text{ and } \quad R(\lambda)\coloneqq \sup\left\{R(\lambda,h):\, h\in H^1,
	\|h\|=1\right\}
\end{align*}
 we see that the following holds\footnote{We combine the cases $\dav>0$ and $\dav=0$ only so that we do not have to distinguish the two cases in Lemma \ref{lem:R-lower} and  Corollary \ref{cor:properties-threshold}.}
\begin{lemma}\label{lem:threshold-basic}
	For any $\dav\ge0$ and $\lambda> 0$ one has $E^{\dav}_\lambda<0$ if and only if $R(\lambda,h)>\frac{\dav}{2}$ for some $h\in H^1(\R)$ with $\|h\|=1$ and this is the case if and only if $R(\lambda)>\frac{\dav}{2} $.
\end{lemma}
\begin{proof}
  If $\dav>0$ the claims are certainly true by the discussion above, which motivated the
  very definition of $R(\lambda)$.
  Now let $\dav=0$. We  note that $E^0_\lambda<0$ if and only if there exists $f\in L^2$
  with $\|f\|^2=\lambda$ and  $N(f)>0$. For $\dav=0$, we consider only the range
  $2<\gamma_1\le \gamma_2\le 6$ together with suitable further $L^p$ properties of $\psi$
  which guarantee that $L^2\ni f\mapsto N(f)$ is continuous, see Lemma \ref{lem:continuity}.
  Since $H^1$ is dense in $L^2$, we can find
  $\wti{f}\in H^1$ with $\|\wti{f}\|=\|f\|$ such that $N(\wti{f})>0$ if $N(f)>0$.
  Thus we also have $E^0_\lambda<0$ if and only if $R(\lambda)>0$.     	
\end{proof}

This simple observation motivates our definition of the threshold:
\begin{definition}[Threshold]\label{def:threshold}
		For $\dav\ge  0$ we let
	\begin{align*}
		\lambda_{\mathrm{cr}}^\dav
			\coloneqq
				\inf \{ \lambda>0:\, R(\lambda) > \frac{\dav}{2} \}.
	\end{align*}
\end{definition}
\begin{remark} \label{rem:R-finite}
  It is clear from Lemma \ref{lem:threshold-basic} that $E^\dav_\lambda<0$ for arbitrary $\lambda>0$ and all $\dav \ge  0$ if and only if $R(\lambda)=\infty$, in which case $\lambda_{\mathrm{cr}}^\dav=0$ for all $\dav\ge 0$.
  Thus it is important to know when $R$ is finite. Using the bound from Proposition \ref{prop:N-boundedness} with $\kappa_1=\kappa_2=4$, which is allowed if $6\le\gamma_1\le \gamma_2<10$, one sees that
  \begin{align*}
    R(\lambda)\lesssim   \lambda^{\f{\gamma_1-2}{2}} + \lambda^{\f{\gamma_2-2}{2}}<\infty \, .
  \end{align*}
  Thus
  $R(\lambda)<\f{\dav}{2}$ for small enough $\lambda>0$, hence $\lambda_{\mathrm{cr}}^\dav>0$ for all $\dav>0$ in this case. This gives an easy proof of Theorem \ref{thm:threshold-phenomena}.\ref{thm:threshold-phenomena-4}, which shows the naturalness of
  $R$ in the study of threshold phenomena.
  \end{remark}
  For a pure power law nonlinearity, given by $V(a)= a^\gamma$ for some $\gamma>2$,
  one can explicitly calculate
  \begin{align*}
  	R(\lambda) =
  		\sup_{\|h\|=1} \frac{N(\sqrt{\lambda} h)}{\lambda \|h'\|^2}
  		= \lambda^{\frac{\gamma-2}{2}} R_0
  		\quad \text{with } R_0= \sup_{\|h\|=1} \frac{\int_\R \|T_r h\|_{L^\gamma}^\gamma\,\psi dr}{\|h'\|^2} \in (0,\infty]
  \end{align*}
  and Remark \ref{rem:R-finite} shows that $R_0<\infty$ if $\gamma\ge 6$.
  This scaling property for $R$ for pure power nonlinearities shows that $R(\lambda)>\f{\dav}{2}$ for all $\lambda>\lambda_{\mathrm{cr}}^\dav$ and
  $R(\lambda)<\f{\dav}{2}$ for all $0<\lambda<\lambda_{\mathrm{cr}}^\dav$. Thus for pure power nonlinearities one immediately sees that the first claim of Theorem \ref{thm:existence+} and \ref{thm:existence0} holds and $\lambda_{\mathrm{cr}}^\dav>0$ if $\gamma\ge 6$.
  For the general class of nonlinearities one cannot expect a simple scaling property of $R$ to hold. However, condition \ref{ass:A2} ensures a lower bound of the same type which is enough to conlcude all necessary properties of $R$ and the threshold $\lambda^\dav_{\mathrm{cr}}$.
 This is made precise in the following

\begin{lemma}\label{lem:R-lower}
	Assume that $V$ obeys assumption \ref{ass:A2}. 
Then
	\begin{align}\label{eq:R-lower}
		R(\lambda_2) \ge \left( \frac{\lambda_2}{\lambda_1} \right)^{\frac{\gamma_0-2}{2}} R(\lambda_1)
	\end{align}
	for all $0<\lambda_1\le\lambda_2$.
\end{lemma}

Before we give the proof we state and prove an important consequence.
\begin{corollary}\label{cor:properties-threshold}
  Assume that $V$ obeys 
  assumption \ref{ass:A2}. Then
  \begin{theoremlist}
  \item \label{cor:properties-threshold-1}
  	If $\lambda^\dav_{\mathrm{cr}}<\infty$, then $R(\lambda)>\f{\dav}{2}$
  	for all $\lambda>\lambda^\dav_{\mathrm{cr}}$.
  \item\label{cor:properties-threshold-2}
  	If $\lambda^\dav_{\mathrm{cr}}>0$, then $R(\lambda)<\f{\dav}{2}$
  	for all
  	$0<\lambda<\lambda^\dav_{\mathrm{cr}}$.
  \item \label{cor:properties-threshold-3}
  $R(\lambda_0)<\infty$ for some $\lambda_0>0$ if and only if $\limsup_{\lambda\to 0}R(\lambda)\le 0$.
  \item \label{cor:properties-threshold-4}
  	$R(\lambda_0) >0$ for some $\lambda_0>0$ if and only if $\liminf_{\lambda\to \infty}R(\lambda)=\infty$.
  \end{theoremlist}   	
\end{corollary}
\begin{proof}
  Take any $\lambda> \lambda^\dav_{\mathrm{cr}}$. The definition of $\lambda^\dav_{\mathrm{cr}}$ shows that there exists $\lambda_1$ with $\lambda>\lambda_1>\lambda^\dav_{\mathrm{cr}}$ and
  	$R(\lambda_1)>\f{\dav}{2}$. Then Lemma \ref{lem:R-lower} shows
  \begin{align*}
  	  	R(\lambda) \ge \left( \frac{\lambda}{\lambda_1} \right)^{\frac{\gamma_0-2}{2}} R(\lambda_1)
  	  	\ge R(\lambda_1) >\f{\dav}{2}\, ,
  \end{align*}
 which proves the first claim.

 For the second claim, assume that $R(\lambda_1)\ge \f{\dav}{2}$ for some $0<\lambda_1<\lambda^\dav_{\mathrm{cr}}$. Then the bound from Lemma \ref{lem:R-lower} shows that
 for every $\lambda_2$ with $\lambda_1<\lambda_2< \lambda^\dav_{\mathrm{cr}}$ we have
 \begin{align*}
  	  	R(\lambda_2) \ge \left( \frac{\lambda_2}{\lambda_1} \right)^{\frac{\gamma_0-2}{2}} R(\lambda_1)
  	  	> R(\lambda_1) \ge \f{\dav}{2}
  \end{align*}
  which is in conflict with the definition of $\lambda^\dav_{\mathrm{cr}}$.

  For the third claim assume that $R(\lambda_0)<\infty$. Setting $\lambda_2=\lambda_0$ and $\lambda_1=\lambda$ in Lemma \ref{lem:R-lower}, we get
  \begin{align*}
  	\limsup_{\lambda\to 0}R(\lambda)
  		\le \limsup_{\lambda\to 0}\left( \frac{\lambda}{\lambda_0} \right)^{\frac{\gamma_0-2}{2}} R(\lambda_0) \le 0 .
  \end{align*}
  The converse is easy.

  For the last claim assume that $R(\lambda_0)>0$ for some $\lambda_0>0$.
  Arguing similarly as above we see that Lemma \ref{lem:R-lower} implies
   \begin{align*}
  	\liminf_{\lambda\to \infty}R(\lambda)
  		\ge  \liminf_{\lambda\to \infty}\left( \frac{\lambda}{\lambda_0} \right)^{\frac{\gamma_0-2}{2}} R(\lambda_0) =\infty .
  \end{align*}
  Again, the converse is easy.
\end{proof}

It remains to give the
\begin{proof}[Proof of Lemma \ref{lem:R-lower}]
	Fix $h\in H^1(\R)\setminus\{0\}$ and define
	\begin{align*}
		A(s)\coloneqq  s^{-2}N(sh)
	\end{align*}
	for $s>0$.
	Because of Lemma \ref{lem:differentiability} , $A$ is differentiable with derivative
	\begin{align*}
		A'(s) = s^{-3} \Big( sD_hN(sh) -2N(sh) \Big)
	\end{align*}
	where
	\begin{align*}
		sD_hN(sh) - 2N(sh)
			&= \iint_{\R^2} \left[ V'(|T_r(sh)(x)|)|T_r(sh)(x)| - 2V(|T_r(sh)(x)|)\right] \,dx \psi dr \\
			&\ge 	(\gamma_0-2) N(sh)
	\end{align*}
	and  the lower bound follows from assumption \ref{ass:A2}.
	Thus we arrive at the first order differential inequality
	\begin{align}\label{eq:A'lowerbound}
		A'(s) \ge \frac{\gamma_0-2}{s} A(s)
	\end{align}
	for all $s>0$. Using the integrating factor $s^{2-\gamma_0}$, one sees that this implies
	$\frac{d}{ds} (s^{2-\gamma_0}A(s)) \ge 0 $
	and thus
	\begin{align*}
		s^{2-\gamma_0} A(s) \ge s_0^{2-\gamma_0} A(s_0)
	\end{align*}
	for all $0<s_0\le s$. Since $R(\lambda,h) = A(\sqrt{\lambda})/\|h'\|^2$, this proves
	\begin{align*}
		R(\lambda_2,h) \ge \left( \frac{\lambda_2}{\lambda_1} \right)^{\frac{\gamma_0-2}{2}} R(\lambda_1,h)
	\end{align*}
	for all $0<\lambda_1\le \lambda_2$ and taking the supremum over all $h\in H^1(\R)$ with
	$\|h\|=1$ gives  \eqref{eq:R-lower}.
\end{proof}

Now we can give the proof of
\begin{proof}[Proof of Theorem \ref{thm:threshold-phenomena}:]
  By Lemma \ref{lem:E-boundedness} we know that $E^\dav_\lambda\le 0$ for all $\lambda>0$ and $\dav\ge 0$ and Proposition  \ref{prop:strict-subadditivity} shows
  \begin{align*}
  	E^\dav_{\lambda_1}\ge E^\dav_{\lambda_1} + E^\dav_{\lambda_2} \ge E^\dav_{\lambda_1+\lambda_2}
  \end{align*}
  where the last inequality is strict, when $E^\dav_{\lambda_1+\lambda_2}<0$. Thus
  $0<\lambda\mapsto E^\dav_\lambda$ is decreasing and strictly decreasing where $E^\dav_\lambda<0$.
  Corollary \ref{cor:properties-threshold} and Lemma \ref{lem:threshold-basic} show that
  $E^\dav_\lambda<0$ if $\lambda>\lambda^\dav_{\mathrm{cr}}$ and
  if $0<\lambda<\lambda^\dav_{\mathrm{cr}}$, Corollary \ref{cor:properties-threshold} and Lemma \ref{lem:threshold-basic} yields
  $E^\dav_\lambda\ge 0$, which together with Lemma  \ref{lem:E-boundedness} shows
  $E^\dav_\lambda= 0$ in this case. This proves the first claim.

   If $\lambda>\lambda^\dav_{\mathrm{cr}}$, we know by the first part that
   $E^\dav_\lambda<0$. So Theorem \ref{thm:existence} applies. This proves the second part.
   \smallskip

   To prove the third claim, assume that $\dav>0$, $0<\lambda<\lambda^\dav_{\mathrm{cr}}$,
   and  $f\in H^1 $ with $\|f\|^2=\lambda>0$ is a minimizer for $E^\dav_\lambda$.
   Using \eqref{eq:energy-threshold}
   we must have
	\begin{align*}
	  0 = E^\dav_\lambda= H(f) \ge 	\|f'\|^2  \left( \frac{\dav}{2}- R(\lambda) \right)\, .
	\end{align*}
	{}From Corollary \ref{cor:properties-threshold-2} we know that $R(\lambda)<\f{\dav}{2}$.  So the above inequality implies $\|f'\|^2 =0$. On $H^1$ the null--space of $\partial_x$ on $H^1$ is trivial, hence $f=0$, which is a contradiction to $\|f\|>0$.
	\smallskip
	
	Assume that $\dav>0$.  Since the  proof of the fourth claim was already given in Remark \ref{rem:R-finite}, we finish with the proof of the last claim.
		
	Assume that there exists $f\in H^1$ with $N(f)>0$. Then $R(\lambda_0)>0$ where $\lambda_0=\|f\|^2$. Corollary \ref{cor:properties-threshold-4} then shows that $\liminf_{\lambda\to\infty} R(\lambda)=\infty$, which implies that for every $\dav\ge 0$ there exists $\lambda>0$ with $R(\lambda)>\f{\dav}{2}$. By the definition of the threshold, this shows  $\lambda^\dav_{\mathrm{cr}}<\infty$ for all $\dav\ge 0$.
\end{proof}

We can finally give the
\begin{proof}[Proof of Theorems \ref{thm:existence+} and \ref{thm:existence0}: ]
  The first three claims of Theorem \ref{thm:existence+}, respectively the
  first two claims of Theorem \ref{thm:existence0}, follow from Theorem \ref{thm:existence} in tandem with  Theorem \ref{thm:threshold-phenomena}.

  It remains to prove that assumption \ref{ass:A3} guarantees that $\lambda^\dav_{\mathrm{cr}}<\infty$ and $\lambda^\dav_{\mathrm{cr}}=0$, if, additionally, \ref{ass:A4} holds.

  Under assumption \ref{ass:A4} Lemma \ref{lem:E-boundedness} shows that $E^\dav_\lambda<0$ for all $\dav\ge 0$ and the definition of $\lambda^\dav_{\mathrm{cr}}$ then yields  $R(\lambda)=\infty$ for $\lambda>0$, so $\lambda^\dav_{\mathrm{cr}}=0$.

  Now assume that \ref{ass:A1}, \ref{ass:A2}, and  \ref{ass:A3} hold.
    First, note that assumptions \ref{ass:A2} and \ref{ass:A3}, together with Lemma \ref{lem:V-lower-bound} show that there exists $a_0>0$ such that $V(a)\gtrsim a^{\gamma_0}$ for all $a\ge a_0$ and using assumption \ref{ass:A1}, we have $V(a)\gtrsim - a^{\gamma_1}$ for
	$0\le a<a_0 $. Thus, with $\gamma\coloneqq \min(\gamma_0,\gamma_1)$, we see that the lower bound	
	\begin{align}\label{eq:V pointwise-lower-bound}
		V(a) \gtrsim  -a^{\gamma}\id_{[0,a_0)}(a) + a^\gamma \id_{[a_0,\infty)}(a)
	\end{align}
	holds. In particular,  $V$ is bounded from below.

Let $\sigma_0>0$ and use the Gaussian  $g_{\sigma_0}$ from \eqref{eq:Gauss}.  Clearly $g_{\sigma_0}\in H^1$ for all $\sigma_0>0$.
Then
 \begin{align*}
   \left| T_r g_{\sigma_0}(x)\right|=A_0\left(\frac{\sigma_0}{|\sigma(r)|}\right)^{1/2}
   e^{-\frac{\sigma_0 x^2}{|\sigma(r)|^2}}
\end{align*}
with $A_0= ( \tfrac{2\lambda^2}{\pi \sigma_0}  )^{1/4}$. Then $\|g_{\sigma_0}\|^2=\lambda$ and, moreover, since $|\sigma(r)|^2=\sigma_0^2+(4r)^2$, we have, for any $|x| \leq \sqrt{\sigma_0}$,
 \beq \label{eq:large enough for small x}
 \left| T_r g_{\sigma_0}(x)\right| \geq A_0\left(\frac{\sigma_0^2}{\sigma_0^2+(4r)^2}\right)^{1/4}
e^{-\frac{\sigma_0^2}{\sigma_0^2+(4r)^2}} \ge \f{A_0}{3}
 \eeq
for any large enough $\sigma_0$ and all $|r| \leq R$, where $R>0$ is chose such that $\supp(\psi)\subset[-R,R]$.
On the other hand, for a large enough $\sigma_0$ and any $|x| \geq \sigma_0$ we also have
 \beq \label{eq:small enough for large x}
 \left| T_r g_{\sigma_0}(x)\right| \leq A_0 \, e^{-\frac{|x|}{2}}.
 \eeq

	By \eqref{eq:large enough for small x}, we can choose $\sigma_0$ and $A_0$
	large enough such that
	$| T_r g_{\sigma_0}(x)|\ge \tfrac{A_0}{3}\ge a_0$ for all $|x|\le \sqrt{\sigma_0}$ and
	$|r|\le R$. Then \eqref{eq:V pointwise-lower-bound} yields
	\begin{align*}
		I \coloneqq \int_{|x|\le \sqrt{\sigma_0}} V(|T_rg_{\sigma_0}(x)|)\, dx \gtrsim  \sqrt{\sigma_0} A_0^\gamma .
	\end{align*}
	Since $V$ is bounded from below, we also have
	\begin{align*}
		II\coloneqq \int_{\sqrt{\sigma_0} \leq |x|\le \sigma_0}
		V(|T_r g_{\sigma_0}(x)|)\, dx  \gtrsim - \sigma_0 ,
	\end{align*}
	and \eqref{eq:V pointwise-lower-bound} together with
	\eqref{eq:small enough for large x} gives
	\begin{align*}
		III\coloneqq \int_{|x|\ge \sigma_0} V(|T_r g_{\sigma_0}(x)|)\, dx
		\gtrsim -  A_0^\gamma \int_{|x|\geq \sigma_0} e^{-\gamma |x|/2}\, dx
		\gtrsim - \frac{A_0^\gamma}{\gamma} e^{-\gamma \sigma_0/2}
	\end{align*}
	for all $|r|\le R$. Thus, since $\psi$ is integrable, this gives the lower bound
	\begin{align*}
		N( g_{\sigma_0}) = \int_\R (I + II + III) \psi\, dr
		\gtrsim \sqrt{\sigma_0} A_0^\gamma - \sigma_0
		-  \frac{A_0^\gamma}{\gamma} e^{-\gamma \sigma_0/2}
	\end{align*}
	for all large enough $A_0$ and $\sigma_0$. Setting $\lambda=\sigma_0$, that is,
	$A_0=( 2\sigma_0/\pi  )^{1/4}$ shows
	$N(g_{{\sigma_0}})>0$ for large enough $\sigma_0>0$.
\end{proof}

\section{Nonexistence}\label{sec:non-existence}
In this section, we will make the standing assumption that $V$ is a
power--law nonlinearity given by
  \begin{align}\label{eq:power-law}
  	V(a) = ca^\gamma \text{ for } a\ge 0
  \end{align}
and  some $c>0, \gamma \ge 6$.
\begin{proof}[Proof of Theorem \ref{thm:non-existence}: ]
 For the proof of the first part of Theorem \ref{thm:non-existence} assume first that
 $\gamma>10$, $c>0$, and fix $\lambda>0$. Let $g_{\sigma_0}$ be the Gaussian from \eqref{eq:Gauss} with $\sigma_0>0$.  Since $V$ is a power--law, Lemma \ref{lem:Gaussians} shows that the nonlinearity $N(g_{\sigma_0})$ is given by
 \begin{equation}
 \begin{split}\label{eq:N on Gauss}
   N(g_{\sigma_0})
   &=  c \left( \f{\pi}{\gamma} \right)^{1/2} \left( \f{2\lambda^2}{\pi} \right)^{\gamma/4}
  			 \sigma_0^{\f{2-\gamma}{4}}
  			 \int_\R \left( \f{\sigma_0}{|\sigma(r)|} \right)^{\f{\gamma-2}{2}} \, \psi(r) dr \\
   &= c \left( \f{\pi}{\gamma} \right)^{1/2} \left( \f{2\lambda^2}{\pi} \right)^{\gamma/4}
  			 \,  \sigma_0^{\f{2-\gamma}{4}}
  			 \int_\R \left( \f{1}{1+ (4r/\sigma_0)^2} \right)^{\f{\gamma-2}{4}} \, \psi(r) dr \\
   &= c \left( \f{\pi}{\gamma} \right)^{1/2} \left( \f{2\lambda^2}{\pi} \right)^{\gamma/4}
  			 \,  \sigma_0^{\f{6-\gamma}{4}}
  			 \int_\R \left( \f{1}{1+ (4s)^2} \right)^{\f{\gamma-2}{4}} \, \psi(\sigma_0 s) ds
 \end{split}
 \end{equation}
 where we also did a change of variables $r=\sigma_0 s$.
 Since $\psi$ is bounded below by $m$, say, in a possibly one-sided neighborhood of zero and $\psi$ has compact support, we have
 \begin{align*}
   \liminf_{\sigma_0\to 0} 	\int_\R \left( \f{1}{1+ (4s)^2} \right)^{\f{\gamma-2}{4}} \, \psi(\sigma_0 s) ds
   \ge m \int_0^\infty \left( \f{1}{1+ (4s)^2} \right)^{\f{\gamma-2}{4}} \, ds >0
 \end{align*}
 so with $C_{\gamma,\lambda} = \tfrac{1}{2} c \left( \f{\pi}{\gamma} \right)^{1/2} \left( \f{2\lambda^2}{\pi} \right)^{\gamma/4} m \int_0^\infty \left( \f{1}{1+ (4s)^2} \right)^{\f{\gamma-2}{4}}\, ds >0 $
 we get from \eqref{eq:N on Gauss} the lower bound
 \begin{align}\label{eq:N lower}
 	N(g_{\sigma_0}) \ge C_{\gamma,\lambda} \, \sigma_0^{\f{6-\gamma}{4}}
 \end{align}
 for all small enough $\sigma_0>0$. Lemma \ref{lem:Gaussians} also yields
 $\|g_{\sigma_0}\|^2=\lambda$ and $ \|g_{\sigma_0}'\|^2=\lambda/\sigma_0$, so
 \begin{align*}
   H(g_{\sigma_0}) = \f{\dav}{2}\|g_{\sigma_0}'\|^2 - N(g_{\sigma_0})
   	\le \f{1}{\sigma_0}\left( \f{\dav \lambda}{2} -  C_{\gamma,\lambda} \, \sigma_0^{\f{10-\gamma}{4}} \right) \to -\infty \text{ as } \sigma_0\downarrow 0
 \end{align*}
 since $\gamma>10$. If $\gamma=10$, we can still conclude that
 $H(g_{\sigma_0})\to-\infty$ as $\sigma_0\downarrow 0$, as long as $C_{\gamma,\lambda}> \f{\dav \lambda}{2}$, which is the case if $c>0$ is large enough.
 This proves that $f\mapsto H(f)$ is unbounded from below on $H^1$ even for fixed $L^2$ norm of $f$.
\smallskip

 If $\dav=0$ and $\gamma>6$, then \eqref{eq:N lower} shows
 \begin{align*}
 	H(g_{\sigma_0})= -N(g_{\sigma_0}) \le  -C_{\gamma,\lambda} \, \sigma_0^{\f{6-\gamma}{4}}
 	  \to -\infty \text{ as } \sigma_0\downarrow 0,
 \end{align*}
 so in this case the energy functional $f\mapsto H(f)$ is again unbounded from 
 below on $L^2$ even for fixed $L^2$ norm of $f$. 
 
 If $\gamma=6$ and $\psi=\id_{[0,1]}$, we modify an argument of \cite{Stanislavova05}. Set
 \begin{align}
 	C_s(\lambda)\coloneqq
 	\sup\left\{ \int_0^s\int_\R |T_rf(x)|^6\, dx dr:\, \|f\|^2=\lambda \right\}
 \end{align}
 and note that $E^0_\lambda$ has a minimizer for $\psi=\id_{[0,1]}$ if there
 is a maximizer for $C_1(\lambda)$.  The main point for the argument is that $C_{s}(\lambda)$ is \emph{independent} of $s>0$: To see this, note that if $u:\R^2\to\C$ solves the free Schr\"odinger equation,
 $i\partial_r u= -\partial_x^2 u$, $u(0,\cdot) = f\in L^2$, then $u_\delta$ defined by
 $u_\delta(r,x)\coloneqq \delta^{1/2} u(\delta^2 r, \delta x )$ solves again the free
 Schr\"odinger equation with initial condition
 $u_\delta(0,x) = f_\delta(x)\coloneqq \delta^{1/2}f(\delta x)$, $x\in\R $. Then
 \begin{align*}
   \int_0^s\int_\R |T_rf_\delta (x)|^6\, dx dr
   	&= \int_0^s\int_\R |u_\delta(r,x)|^6\, dx dr	
   		= \int_0^s\int_\R \delta^3|u(\delta^2r,\delta x)|^6\, dx dr \\
   	&= 	\int_0^{\delta^2s}\int_\R |u(r, x)|^6\, dx dr
   		= \int_0^{\delta^2s}\int_\R |T_r f(x))|^6\, dx dr
 \end{align*}
 and noting $\|f_\delta\|^2=\|f\|^2=\lambda $ we get
 \begin{align*}
   C_s(\lambda) = C_{\delta^2 s} (\lambda)	
 \end{align*}
 for all $s,\delta,\lambda>0$, in particular, $C_s(\lambda)= C_{1}(\lambda)$
 for all $s > 0$ and $\lambda>0$. Assume that $f$ is a minimizer for $E^0_\lambda$ for
 $\psi=\id_{[0,1]}$, that is, $f$ is a maximizer for $ C_1(\lambda)$:
 \begin{align*}
   \|f\|^2=\lambda>0 \text{ and } C_1(\lambda)
   	= \int_0^1\int_\R |T_rf (x)|^6\, dx dr .	
 \end{align*}
 Then
 \begin{align*}
   0= C_2(\lambda)-C_1(\lambda) &\ge \int_0^2\int_\R |T_rf (x)|^6\, dx dr -\int_0^1\int_\R |T_rf (x)|^6\, dx dr \\
   	&= \int_1^2\int_\R |T_rf (x)|^6\, dx dr  	\ge 0\, .
 \end{align*}
 So $|T_rf(x)|=0$ for Lebesque almost all pairs $1\le r\le 2$ and $x\in\R $ and hence, since $T_r$ is unitary on $L^2$,
 \begin{align*}
   0 =\int_1^2 \int_\R |T_rf(x)|^2\, dx dr = \|f\|^2	,
 \end{align*}
 which contradicts $\|f\|^2=\lambda>0$. So no minimizer of \eqref{eq:min} exists if
 $\gamma=6$ in the model case where $\psi= \id_{[0,1]}$.

\end{proof}

\vspace{5mm}

\appendix
\setcounter{section}{0}
\renewcommand{\thesection}{\Alph{section}}
\renewcommand{\theequation}{\thesection.\arabic{equation}}
\renewcommand{\thetheorem}{\thesection.\arabic{theorem}}

\section{Tightness and strong convergence in $L^2$} \label{sec:strong-convergence}
A key step in our existence proof of minimizers of the variational problems \eqref{eq:min} is the following characterization of strong convergence in $L^2(\R)$ which is given in \cite{HuLee2012}.

\begin{lemma}\label{lem:strong-convergence-L2}
A sequence $(f_n)_n\subset L^2(\R)$ is strongly converging to $f$ in $L^2(\R)$
if and only if it is weakly convergent to $f$ and
 \begin{align}
    \lim_{L\to\infty} \limsup_{n\to\infty} \int_{|\eta|>L} |\hatt{f}_n(\eta)|^2 \, d\eta &= 0,
    \label{eq:tight-fourier}\\
    \lim_{R\to\infty} \limsup_{n\to\infty} \int_{|x|>R} |f_n(x)|^2 \, dx &= 0,
    \label{eq:tight-real}
 \end{align}
 where $\hatt{f}$ is the Fourier transform of $f$.
\end{lemma}

\section{Galilei transformations and space-time localization properties of Gaussian coherent states} \label{sec:Galilei}
\setcounter{theorem}{0}
\setcounter{equation}{0}

We will only discuss the one-dimensional case which is somewhat easier since we do not have
to deal with rotations in one dimension.
The unitary operator implementing the shift $S_y :L^2(\R)\to L^2(\R)$,
$(S_y f)(x)= f(x-y)$ is given by
 \beq
    S_y= e^{-iy P}
 \eeq
where $P= -i\partial_x$ is the momentum operator. Indeed, since $e^{-iy P}$ corresponds
to multiplication by $e^{-iy k}$ in Fourier space, we have
 \bdm
    (e^{-iy P}f)(x) = \frac{1}{\sqrt{2\pi}} \int_\R e^{i(x-y)k}\hatt{f}(k)\, dk
    = f(x-y).
 \edm
Boosts, i.e., shifts in momentum space are given by $e^{iv\cdot}:L^2(\R)\to L^2(\R)$, i.e.,
multiplication by $e^{ivx}$, since
 \beq
    \hatt{e^{iv\cdot}f}(k)= \frac{1}{\sqrt{2\pi}} \int_\R e^{-ix(k-v)}f(x)\, dx
    = \hatt{f}(k-v).
 \eeq
Finally, if $G$ is a bounded (measurable) function then $G(P)$ 
is defined by
 \bdm
    \hatt{G(P)f}(k)= G(k) \hatt{f}(k) .
 \edm

Of course, for any $y \in\R$, the operators $G(P)$ and $e^{-iy P}$ commute,
$G(P)e^{-iy P}= e^{-iy P}G(P)$.  Moreover, for any $v\in\R$ the commutation relation
 \beq
    G(P)e^{iv\cdot} = e^{iv\cdot} G(P+v)
 \eeq
holds. Indeed, Computing the Fourier transform $\calF$ yields
 \bdm
 \begin{split}
    \calF \big(G(P) e^{iv\cdot} f \big) (k) 
    &= G(k) \hatt{e^{iv\cdot} f}(k) = G(k) \hatt{f}(k-v) \\
    &= ( G(\cdot+ v) \hatt{f} ) (k-v) = \calF \big( G(P + v) f \big) (k-v) \\
    &= \calF \big(e^{iv\cdot}G(P + v) f \big) (k).
 \end{split}
 \edm
In particular, choosing $G(P)= e^{-irP^2}$, we arrive at the commutation relation
 \beq\label{eq:commutation}
 \begin{split}
    e^{-irP^2} e^{iv\cdot} e^{-iy P}
    & = e^{iv\cdot} e^{-iy P} e^{-ir(P+v)^2}
        = e^{iv\cdot} e^{-iy P} e^{-ir(P^2+2vP + v^2)} \\
    &= e^{-irv^2} e^{iv\cdot} e^{-i(y +2rv)P} e^{-irP^2} .
 \end{split}
 \eeq
Now let $f\in L^2(\R)$. Then $u(r)= T_rf= e^{-irP^2}f$ is the solution of the
(one-dimensional) Schr\"odinger equation $-i\partial_r u = P^2 u = -\partial_x^2 u$
with initial condition $u(0)= f$. Using \eqref{eq:commutation}, the solution of the
free Schr\"odinger equation for the translated and boosted initial condition
$f_{y,v}= e^{iv\cdot}e^{-iy P} f$ is given by
 \beq\label{eq:Galilei-trafo}
 \begin{split}
    u_{y,v}(r,x)
    &\coloneqq  T_r f_{y,v} (x)
        = \big(e^{-irP^2} e^{iv\cdot}e^{-iy P} f\big)(x) \\
    &= \big(e^{-irv^2} e^{iv\cdot} e^{-i(y +2rv)P} e^{-irP^2} f\big)(x) \\
    &= e^{-irv^2} e^{iv x} \big(e^{-i(y +2rv)P} e^{-irP^2} f\big)(x) \\
    &   = e^{-irv^2} e^{iv x} \big( e^{-irP^2} f\big)(x - y -2rv) \\
    &= e^{-irv^2} e^{iv x} (T_r f)(x-y-2rv),
 \end{split}
 \eeq
that is, on the level of the solutions of the free time-dependent Schr\"odinger equation,
translations and boosts of the initial condition are implemented by the Galilei transformations $\calG_{y,v}$ given by
$(\calG_{y,v}u)(r,x) \coloneqq u_{y,v}(r,x)= e^{-irv^2} e^{iv x} u(r,x - y -2rv)$.
Except for the time-dependent phase factor $e^{-irv^2}$, formula \eqref{eq:Galilei-trafo} is
exactly what one would have guessed from classical mechanics

A simple calculation now shows that any functional of the form
 \bdm
        f\mapsto N(f)=\iint_{\R^2} V(|T_rf(x)|)\, dx \psi dr
 \edm
is invariant under translations and boosts of $f$ in $L^2(\R)$.


Now, we come to one of the major tools for our analysis, the so-called coherent states.
\begin{definition}[Coherent states]
  Let $h\in L^2$, $\|h\|=1$, $y, \, v \in \R$  and $h_{y,v}\coloneqq e^{iv\cdot} e^{-iyP}h$, i.e.,
   \beq
     h_{y,v}(x) = e^{ivx} h(x-y)
   \eeq
   for $x\in\R$  and define the coherent rank-one projection $P_{y,v}\coloneqq |h_{y,v}\ra\la h_{y,v}|$ in Dirac's notation, i.e., given by
 \beq
   f\mapsto P_{y,v}f \coloneqq   h_{y,v} \la h_{y,v},f\ra.
 \eeq
\end{definition}
A well-known property of coherent states  is their completeness expressed in

 \begin{lemma}[Completeness of coherent states]\label{lem:completeness}
  Let $h\in L^2(\R)$ with $\|h\|=1$ and $h_{y,v}$ the shifted and boosted $h$ as above. Then, in a weak sense,
  \begin{align}\label{eq:completeness}
    \frac{1}{2\pi} \iint_{\R^2} dydv  P_{y,v} =  \frac{1}{2\pi} \iint_{\R^2} dydv  |h_{y,v}\ra\la h_{y,v}|  = \id
  \end{align}
  on $L^2$. Moreover,
  \begin{align}\label{eq:completeness-time}
  	 \frac{1}{2\pi} \int_\R dv\, \la f, P_{y,v}f\ra
  	 = \int |h(x-y)|^2|f(x)|^2 \, dx
  \end{align}
  and
    \begin{align}\label{eq:completeness-fourier}
  	 \frac{1}{2\pi} \int_\R dy\, \la f, P_{y,v}f\ra
  	 = \int |\hat{h}(\eta-v)|^2|\hat{f}(\eta)|^2\, d\eta  .
  \end{align}
 \end{lemma}
 \bpf
 The completeness expressed in \eqref{eq:completeness}
 is well-known, see \cite{Perry, Simon}, the other two are less known.  We give a short proof for the convenience of the reader:
 In order to see that the operator $A$ given by its matrix elements
 \bdm
   \la f_1, Af_2\ra \coloneqq \frac{1}{2\pi}\int_{\R}\int_{\R} dydv  \la f_1, h_{y,v}\ra\la h_{y,v},f_2\ra
 \edm
 is the identity on $L^2$ it is enough, by polarization, to take $f_1=f_2=f$ and to check
 $\la f, Af\ra = \la f,f\ra$ for all $f\in L^2$.  Note
 \bdm
  \la h_{y,v},f\ra   = \int_{\R} e^{-ivx} \ol{h(x-y)} f(x)\, dx
  = (2\pi)^{1/2}\hatt{(\ol{h_{y,0}}f)}(v).
 \edm
 and thus by Plancherel,
 \bdm
  \frac{1}{2\pi} \int_\R dv\, \la f, P_{y,v}f\ra
   =
   \frac{1}{2\pi} \int_{\R} dv |\la h_{y,v},f \ra |^2
     = \int_{\R} dx \, |h_{y,0}(x)f(x)|^{2}
     = \int_{\R} dx |h(x-y)f(x)|^{2},
 \edm
 so \eqref{eq:completeness-time} follows and we also see
  \bdm
   \la f, Af \ra
     = \frac{1}{2\pi} \int_{\R} dy \int_{\R} dv |\la h_{y,v},f \ra |^2
     = \int_{\R} dy \int_{\R} dx |h(x-y)f(x)|^{2}
     = \int_{\R} |f(x)|^2\, dx
 \edm
 thus, in addition, \eqref{eq:completeness} follows. For \eqref{eq:completeness-fourier} we note that a short calculation reveals
 \begin{align*}
 	\widehat{h_{y,v}}(\eta) = e^{-iy(\eta-v)} \hat{h}(\eta-v)= e^{iyv}\hat{h}_{v,-y}(\eta).
 \end{align*}
 By Plancherel
 \begin{align*}
	\la h_{y,v},f \ra
 	= \la \widehat{h_{y,v}},\widehat{f} \ra
  	= \int_{\R} e^{iy(\eta-v)} \ol{\hat{h}(\eta-v)}\widehat{f}(\eta)\, d\eta
  = (2\pi)^{1/2}e^{-iyv} \mathcal{F}^{-1}\left[\ol{\hatt{h}_{v,0}}\hatt{f}\right](y)
 \end{align*}
 where $\mathcal{F}^{-1}$ denotes the inverse Fourier transform.
 Again by Plancherel, we thus have
 \begin{align*}
 	\frac{1}{2\pi} \int_\R dy\, \la f, P_{y,v}f\ra  	
 	=
 	  \frac{1}{2\pi} \int_\R dy\, |\la \widehat{h_{y,v}},\widehat{f} \ra |^2
 	=
 	\int_{\R} d\eta\, \left|\ol{ \hatt{h}_{v,0}(\eta) } \hatt{f}(\eta) \right|^2
 	=
 	\int_{\R} d\eta\, \left|\ol{ \hatt{h}(\eta-v) } \hatt{f}(\eta) \right|^2
  \end{align*}
  and  \eqref{eq:completeness-fourier} follows.
\epf

 We use coherent states in order to localize a wave function simultaneously in real and Fourier spaces and since Gaussians have nice localization properties simultaneously in real and Fourier spaces, it is natural to use Gaussian coherent states for this.

 First we note some important properties of Gaussians, which are needed in several places of this work.

\begin{lemma}[Properties of Gaussians] \label{lem:Gaussians}
 	Let $\lambda>0$, $\sigma_0\in\C$ with $\re(\sigma_0)>0$, and
 	\begin{align} \label{eq:Gauss}
 		g_{\sigma_{0}}(x) = \left(\frac{2\re(\sigma_0)\, \lambda^2}{\pi |\sigma_0|^2}\right)^{1/4}
 				e^{-\f{x^2}{\sigma_0}} .
 	\end{align}
  Then   $\|g_{\sigma_0}\|^2 =\lambda$, $\|g_{\sigma_0}'\|^2 =\f{\lambda}{\re\, \sigma_0}$, and its time evolution is given by
  \begin{align}\label{eq:centered Gaussian time evolution}
  	T_rg_{\sigma_0}(x) =  \left(\frac{2\re(\sigma_0)\, \lambda^2}{\pi |\sigma_0|^2}\right)^{1/4} \left( \f{\sigma_0}{\sigma(r)} \right)^{1/2} e^{-\f{x^2}{\sigma(r)}}
  \end{align}
  with $\sigma(r)= \sigma_0+4ir$. In particular, for all $\gamma\ge 1$,
  \begin{align}\label{eq:Gaussian-Lp-norm}
  	\|T_rg_{\sigma_0}\|_{L^\gamma(\R,dx)}^\gamma
  		=	\left( \f{\pi}{\gamma} \right)^{1/2} \left( \f{2\lambda^2}{\pi} \right)^{\gamma/4}
  			 \left( \f{\re(\sigma_0)}{|\sigma_0|^2} \right)^{\f{\gamma-2}{4}}
  			 \left( \f{|\sigma_0|}{|\sigma(r)|} \right)^{\f{\gamma-2}{2}}
  \end{align}
\end{lemma}
\begin{proof}	
  Write $g_{\sigma_0}$ as  $g(x)= A_0e^{-x^2/\sigma_0}$ with $A_0,\sigma_0\in \C$ with $\re(\sigma_0)>0$. Then
  \begin{align*}
    |g(x)| = |A_0| e^{-\f{\re(\sigma_0) x^2}{|\sigma_0|^2}}	
  \end{align*}
  and thus
  \begin{align*}
    \|g\|^2 = |A_0|^2 \int_\R 	e^{-\f{2\re(\sigma_0) x^2}{|\sigma_0|^2}}\, dx
    	= |A_0|^2  \left( \f{\pi |\sigma_0|^2}{2\re(\sigma_0)} \right)^{1/2}
  \end{align*}
  using the Gaussian integral $\int_\R e^{-\beta x^2}\, dx = (\tfrac{\pi}{\beta})^{1/2}$. Thus with the choice
  \begin{align}\label{eq:choice of A_0}
  	A_0= (\tfrac{2\re(\sigma_0)\, \lambda^2}{\pi |\sigma_0|^2})^{1/4}
  	\end{align}
  we have  $\|g\|^2 =\lambda$.
  In addition,
  \begin{align*}
  	\|g'\|^2 &=  |A_0|^2 4|\sigma_0|^{-2}\int_\R x^2	e^{-\f{2\re(\sigma_0) x^2}{|\sigma_0|^2}}\, dx
  		= |A_0|^2 4|\sigma_0|^{-2}
  		\left. \left(- \f{\partial}{\partial\beta}\int_\R e^{-\beta x^2}\, dx \right)\right|_{\beta= \f{2\re(\sigma_0)}{|\sigma_0|^2}} \\
  	&=  2|A_0|^2 |\sigma_0|^{-2} \pi^{1/2}
  		\left( \f{|\sigma_0|^2}{2\re(\sigma_0)} \right)^{3/2}
  		=
  		\f{\lambda}{\re(\sigma_0)}\, .
  \end{align*}

  To prove formula \eqref{eq:centered Gaussian time evolution} note that for a centered Gaussian the time evolution $T_rg$ can be found by making the ansatz
 \begin{align}\label{eq:ansatz}
   (T_rg)(x) = A(r) e^{-x^2/\sigma(r)}=:u(r,x).
 \end{align}
 A short calculation,  using that $u(r,x)$ solves $i \partial_r u = -\partial_x^2 u$, reveals that $A$ and $\sigma$ solve
 \begin{align*}
 	iA' = \frac{2A}{\sigma} \quad\text{and } \sigma'= 4i,
 \end{align*}
 thus $A(r)$ and $\sigma(r)$ are given by
 \beq\label{eq:A(r)}
  A(r)= A_0\left( \frac{\sigma_0}{\sigma(r)} \right)^{1/2}  \quad \text{and} \quad \sigma(r)= \sigma_0+4ir
 \eeq
 which proves \eqref{eq:centered Gaussian time evolution}.

 Using  \eqref{eq:ansatz}, \eqref{eq:A(r)}, and $\re(\sigma(r))= \re(\sigma_0)$  we get
 \begin{align*}
 	\|T_rg_{\sigma_0}\|_{L^\gamma(\R,dx)}^\gamma
 	&= |A_0|^\gamma \left|\f{\sigma_0}{\sigma(r)}\right|^{\gamma/2}
 		\int_\R e^{-\f{\gamma\re(\sigma_0)x^2}{|\sigma(r)|^2}}\, dx
 	=	|A_0|^\gamma \left|\f{\sigma_0}{\sigma(r)}\right|^{\gamma/2}
 		\left( \f{\pi |\sigma(r)|^2}{\gamma\re(\sigma_0)} \right)^{1/2}
 \end{align*}
 and with the choice \eqref{eq:choice of A_0} for $A_0$ and rearranging the terms this shows \eqref{eq:Gaussian-Lp-norm}.

\end{proof}

The localization properties of Gaussian coherent states are the content of

\begin{lemma}[Space-time localization properties of Gaussian coherent states] \label{lem:gaussian coherent states}
Let $g(x)= \pi^{-1/4} e^{-x^2/2}$ be the standard $L^2$ normalized Gaussian and
 \beq
   g_{y,v}(x)\coloneqq e^{ivx} g(x-y)
 \eeq
 its shifted and boosted version. Let
 \beq
   P^{\scriptscriptstyle{\le}}_L \coloneqq \frac{1}{2\pi}\int _{\R}dy\int_{|v|\le L}dv  |g_{y,v}\ra\la g_{y,v}|
 \eeq
 and
 \beq
 P^{\scriptscriptstyle{>}}_L \coloneqq \frac{1}{2\pi}\int_{\R} dy\int_{|v|> L}dv  |g_{y,v}\ra\la g_{y,v}|.
 \eeq
 Then $P^{\scriptscriptstyle{\le}}_L+P^{\scriptscriptstyle{>}}_L=\id$, $0\le P^{\scriptscriptstyle{\le}}_L\le \id$, and $0\le P^{\scriptscriptstyle{>}}_L\le \id$ as operators.
 Moreover $P^{\scriptscriptstyle{>}}_L$ localizes a wave function in the region of large frequencies $|\eta|\gtrsim L$ in
 the sense that for any $f\in H^\alpha$ we have
 \beq\label{eq:high momenta}
  \| P^{\scriptscriptstyle{>}}_Lf \| \lesssim L^{-\alpha} \| f \|_{H^\alpha}
 \eeq
 where the implicit constant does not depend on $f$ nor $L$.

 Moreover, the time-evolution of the shifted and boosted Gaussian $g_{y,v}$ is given by
 \beq\label{eq:gaussian time evolution}
   (T_rg_{y,v})(x) = \frac{1}{\pi^{1/4}\sqrt{1+2ir}} e^{-irv^2} e^{iv x}
   		e^{-\f{(x-y-2rv)^2}{2(1+2ir)}}
 \eeq
 and for any $f_1,f_2\in L^2$ which have separated supports we have the bilinear estimate
  \beq\label{eq:low momenta bilinear}
      \sup_{|r|\le R}\|T_r P^{\scriptscriptstyle{\le}}_L f_1 T_r P^{\scriptscriptstyle{\le}}_L f_2\|_{L^p_x}
        \lesssim A_R L^2 e^{L^2/p -B_{p,R} s^2} \|f_1\| \|f_2\|,~1 \le p <\infty,
  \eeq
  where $A_R\coloneqq \sqrt{1+4R^2}$,
   $B_{p,R}\coloneqq 2^{-4}(\sqrt{p(1+4R^2)}+1)^{-2}$, and $s\coloneqq \dist(\supp f_1, \supp f_2)$.
\end{lemma}

\bpf
The first assertions are clear, since by Lemma \ref{lem:completeness} we have $P^{\scriptscriptstyle{\le}}_L+ P^{\scriptscriptstyle{>}}_L=\id $ and certainly
$P^{\scriptscriptstyle{\le}}_L$ and $P^{\scriptscriptstyle{>}}_L\ge 0$ in the sense of operators. So also $P^{\scriptscriptstyle{\le}}_L = \id - P^{\scriptscriptstyle{>}}_L\le \id $ and similarly $P^{\scriptscriptstyle{>}}_L\le \id $.

To prove \eqref{eq:high momenta}, we first note that because of $0\le P^{\scriptscriptstyle{>}}_L\le \id$, one has
 \begin{align*}
 	\|P^{\scriptscriptstyle{>}}_L f\|^2
 	= \la {P^{\scriptscriptstyle{>}}_L}^{1/2}f, {P^{\scriptscriptstyle{>}}_L}{P^{\scriptscriptstyle{>}}_L}^{1/2} f \ra
 	\le \la f, P^{\scriptscriptstyle{>}}_L f \ra .
 \end{align*}
Let $P_{y,v}\coloneqq |g_{y,v}\ra \la g_{y,v}| $, then
\begin{align}
	\la f, P^{\scriptscriptstyle{>}}_L f\ra
	&= \frac{1}{2\pi} \int_{\R } dy \int_{|v|>L} dv\, \la f, P_{y,v} f\ra
	=  \int_{|v|>L} \int_{\R } |\hat{g}(\eta-v)|^2|\hat{f}(\eta)|^2\, d\eta dv   \nonumber\\
	&= \frac{1}{\sqrt{\pi}}
	    \int_{|v|>L} \int_{\R } e^{-(\eta-v)^2}|\hat{f}(\eta)|^2\, d\eta dv
	  = \int_{\R } H_L(\eta) |\hat{f}(\eta)|^2\, d\eta
	    \label{eq:loc-fourier-1}
\end{align}
due to \eqref{eq:completeness-fourier} and $\hat{g}=g$ where we set
 \begin{align*}
 	H_L(\eta)\coloneqq \frac{1}{\sqrt{\pi}}
	    \int_{|v|>L}  e^{-(\eta-v)^2}\, dv .
 \end{align*}
Note that $H_L$ is even, $0<H_L\le 1$, increasing on $[0,\infty)$, and
 $\lim_{\eta\to\infty} H_L(\eta)=1$.
 A short calculation reveals
\begin{align*}
	H_L(L) =
	    \frac{1}{2} + \frac{1}{\sqrt{\pi}}
	    \int_{2L}^\infty  e^{-v^2}\, dv
\end{align*}
so $H_L(L)$ is extremely close to $1/2$ for large $L$. For $|\eta|\le L/2$ and $|v|\ge L$, one has $|v-\eta|\ge |v|-|\eta| \ge |v| -L/2\ge L/2$, hence
 \begin{align*}
 	H_L(\eta) \le \frac{2}{\sqrt{\pi}}
	    \int_L^\infty  e^{-\f{L}{2}(v-\f{L}{2})}\, dv
	    = \frac{4}{\sqrt{\pi} L } e^{-\f{L^2}{4}}
	    \quad \text{for all } |\eta|\le \frac{L}{2}.
 \end{align*}
So
 \begin{align*}
 	\int_{\R } H_L(\eta) |\hat{f}(\eta)|^2\, d\eta
 	&= \int_{|\eta|\le L/2} H_L(\eta) |\hat{f}(\eta)|^2\, d\eta
 	   + \int_{|\eta|> L/2} H_L(\eta) |\hat{f}(\eta)|^2\, d\eta \\
 	&\le
		\frac{4}{\sqrt{\pi} L } e^{-\f{L^2}{4}} \|f\|^2
		+ \int_{|\eta|> L/2}  |\hat{f}(\eta)|^2\, d\eta .
 \end{align*}
Using
 \begin{align*}
 	\int_{|\eta|> L/2}  |\hat{f}(\eta)|^2\, d\eta
 	\le
 	  (L/2)^{-2\alpha} \int_{|\eta|> L/2}  |\eta|^{2\alpha}| \hat{f}(\eta)|^2\, d\eta
 	  \le
 	    (L/2)^{-2\alpha} \|f\|_{H^\alpha}^2
 \end{align*}
 completes the proof of \eqref{eq:high momenta}.

To prove formula \eqref{eq:gaussian time evolution} first note that for the
centered Gaussian from \eqref{eq:Gauss} with $\sigma_0=2$ and $\lambda=1$
Lemma \ref{lem:Gaussians} gives the time evolution as
  \begin{align*}
    (T_rg_{0,0})(x) = \pi^{-1/4} \frac{1}{\sqrt{1+2ir}} e^{-\f{x^2}{2(1+2ir)}}.
  \end{align*}
 Now we use the Galilei transformation formula \eqref{eq:Galilei-trafo} to arrive at
 \bdm
   (T_rg_{y,v})(x) = \pi^{-1/4}\frac{e^{-irv^2} e^{iv x}}{\sqrt{1+2ir}}
   	e^{-\f{(x-y-2rv)^2}{2(1+2ir)}}
 \edm
which is \eqref{eq:gaussian time evolution}.

To prove \eqref{eq:low momenta bilinear},  fix $|r|\le R$ and note that
 $$
    (T_r P^{\scriptscriptstyle{\le}}_L f)(x)= \frac{1}{2\pi}\int_\R dy  \int _{|v| \le L} \!\!\!\!\!\! dv\,(T_rg_{y,v})(x) \la g_{y,v},f \ra .
 $$
Thus using \eqref{eq:gaussian time evolution} and the triangle inequality
  \bdm
    |(T_rP^{\scriptscriptstyle{\le}}_Lf)(x)|\le \frac{1}{2\pi(\pi(1+4r^2))^{1/4}} \int_{\R} dy \int_{|v|\le L} \!\!\!\!\!\! dv\,  e^{-\frac{(x-y-2rv)^2}{2(1+4r^2)}} | \la g_{y,v},f \ra |
  \edm
together with
  \bdm
    A(r,L)\coloneqq\int_\R dy  \int _{|v| \le L} dv \, e^{-\f{(x-y-2rv)^2}{2(1+4r^2)}} = 2L (2\pi(1+4r^2))^{1/2},
  \edm
 which is independent of $x$, by translation invariance of Lebesgue measure we can thus bound
 \bdm
   |(T_rP^{\scriptscriptstyle{\le}}_Lf)(x)|
     \le \frac{A(r,L)}{2\pi(\pi(1+4r^2))^{1/4}} \int_{\R}  \int_{\R} \nu_x(dy,dv)  | \la g_{y,v},f \ra |
  \edm
  with the probability measure $\nu_x(dy,dv)\coloneqq \tfrac{1}{A(r,L)} e^{-\f{(x-y-2rv)^2}{2(1+4r^2)}} \id_{|v|\le L} \, dydv$. Hence
 Jensen's inequality \cite{LiebLoss} for the convex function $r\to |r|^p,\ 1 \le p < \infty ,$ shows
 \begin{align*}
     \left| (T_r P^{\scriptscriptstyle{\le}}_L f)(x)\right|^p &
     \le
      \frac{A(r,L)^p}{(2\pi)^{p}(\pi(1+4r^2))^{p/4}} \int_{\R}  \int _{\R}\nu_x(dy,dv)  | \la g_{y,v},f \ra |^p \\
     & \lesssim L^{p-1} (1+4r^2)^{\f{p-2}{4}} \int_\R dy  \int _{|v| \le L} \!\!\!\!\!\! dv\, e^{-\f{(x-y-2rv)^2}{2(1+4r^2)}}|\la g_{y,v},f \ra|^p.
 \end{align*}
Therefore,
 \beq\label{eq:PL0}
 \begin{split}
     \|(T_rP^{\scriptscriptstyle{\le}}_Lf_1) (T_rP^{\scriptscriptstyle{\le}}_Lf_2)\|_{L_x^p}^p
     \lesssim L^{2(p-1)} (1+4r^2)^{\f{p-2}{2}}
     \int_\R dy_1  \int _{|v_1| \le L} \!\!\!\!\!\! dv_1 \int_\R dy_2  \int _{|v_2| \le L} \!\!\!\!\!\! dv_2 \hskip 5cm \\
     |\la g_{y_1,v_1},f_1 \ra|^p |\la g_{y_2,v_2},f_2 \ra|^p \int_\R dx \, e^{-\f{(x-y_1-2rv_1)^2 + (x-y_2-2rv_2)^2}{2(1+4r^2)}}  \hskip 4cm\\
      \lesssim L^{2(p-1)}(1+4r^2)^{\f{p-1}{2}}
     \int_\R dy_1  \int _{|v_1| \le L} \!\!\!\!\!\! dv_1 \int_\R dy_2  \int _{|v_2| \le L} \!\!\!\!\!\! dv_2 \
     |\la g_{y_1,v_1},f_1 \ra|^p |\la g_{y_2,v_2},f_2 \ra|^p \ e^{-\f{[(y_1-y_2)+2r(v_1-v_2)]^2}{4(1+4r^2)}}
      \hskip 2.2cm
 \end{split}
 \eeq
 where we used
 \bdm
   \int_\R dx \, e^{-\f{(x-y_1-2rv_1)^2 + (x-y_2-2rv_2)^2}{2(1+4r^2)}}  = (\pi(1+4r^2))^{1/2}e^{-\f{((y_1-y_2)+2r(v_1-v_2))^2}{4(1+4r^2)}}
 \edm
 by a simple convolution of Gaussians.
Since $(a+b)^2\ge \tfrac{1}{2}a^2- b^2$ for any $a,b\in\R$, the lower bound
  \bdm
    [(y_1-y_2)+2r(v_1-v_2)]^2 \ge \frac{1}{2}(y_1-y_2)^2 - 16r^2L^2
  \edm
 holds for all $y_1,y_2$, and $|v_1|, |v_2|\le L$. Moreover,
 \bdm
   |\la g_{y,v},f \ra| \le \int_\R  |g_{y,v}(x)||f(x)|\, dx= \pi^{-1/4} \int_\R e^{-\f{1}{2}(x-y)^2}|f(x)|\, dx = (g_{0,0}*|f|)(y),
 \edm
and thus \eqref{eq:PL0} gives the upper bound
  \beq\label{eq:PL}
 \begin{split}
     \|T_rP^{\scriptscriptstyle{\le}}_Lf_1 T_rP^{\scriptscriptstyle{\le}}_Lf_2\|_{L_x^p}^p
     \lesssim L^{2p}\ e^{L^2}(1+4r^2)^{\f{p-1}{2}}
     \int_\R dy_1  \int_\R dy_2\,
       e^{-\f{(y_1-y_2)^2}{8(1+4r^2)}} [ g_{0,0}*|f_1|(y_1) ]^p [ g_{0,0}*|f_2|(y_2) ]^p.
 \end{split}
 \eeq

Let $K_j\coloneqq \supp \, f_j,\  j=1,2$ be the support of $f_j$. Recall that we assume $s\coloneqq \dist(K_1,K_2)>0$. Given $0<\tilde{s}<s/2$, we will enlarge $K_j$ a little bit,
 \bdm
   \wti{K}_j\coloneqq \{ y\in\R|\, \dist(y,K_j) \le \tilde{s} \}.
 \edm
 Note that $\dist(\wti{K}_1,\wti{K}_2)= s-2\wti{s}>0$
 and we will split the integral in \eqref{eq:PL} according to the splitting
 $\R\times\R= (\wti{K}_1^c\times\R)\cup(\wti{K}_1\times \R )=  (\wti{K}_1^c\times\R)\cup(\wti{K}_1\times \wti{K}_2^c)\cup (\wti{K}_1\times \wti{K}_2)$.
As a further  preparation,
note that the Cauchy-Schwartz inequality implies
 \beq\label{eq:double-integral}
 \begin{aligned}
    &\left|\int\!\!\!\int_{\R^2} e^{-\f{1}{c}(y_1-y_2)^2}h_1(y_1)h_2(y_2)\,dy_1dy_2  \right| \\
     \le &\left[\int\!\!\!\int_{\R^2}e^{-\f{1}{c}(y_1-y_2)^2} |h_1(y_1)|^2 dy_1dy_2 \right]^{1/2} \left[\int\!\!\!\int_{\R^2}e^{-\f{1}{c}(y_1-y_2)^2} |h_2(y_2)|^2 dy_1dy_2 \right]^{1/2}\\
      = & \sqrt{c\pi} \|h_1\|\|h_2\|.
 \end{aligned}
 \eeq
 for any $h_1,h_2\in L^2(\R)$ and $c>0$. Using this, we can bound
 \begin{align}\label{eq:I1}
    I_1 \coloneqq & \int_{\wti{K_1}^c} dy_1 \int_\R dy_2\  e^{-\f{(y_1-y_2)^2}{8(1+4r^2)}} \bigl[(g_{0,0}*|f_1|)(y_1)\bigr]^p\bigl[(g_{0,0}*|f_2|)(y_2)\bigr]^p \notag\\
    & \lesssim (1+4r^2)^{1/2} \left[\int_{\wti{K_1}^c} \Bigl[(g_{0,0}*|f_1|)(y_1)\Bigr]^{2p} \, dy_1  \right]^{1/2}
    \left[\int_\R \Bigl[(g_{0,0}*|f_2|)(y_2)\Bigr]^{2p} \,  dy_2\right]^{1/2}.
 \end{align}
Moreover, by Young's inequality,
 \beq\label{eq:I1-1}
     \int_{\R} \Bigl[(g_{0,0}*|f_2|)(y_2)\Bigr]^{2p} \,  dy_2 \lesssim \|f_2\|^{2p}
 \eeq
and, on the other hand,
 \begin{align}\label{eq:I1-2}
    & \int_{\wti{K_1}^c}  \Bigl[(g_{0,0}*|f_1|)(y)\Bigr]^{2p}\, dy
    =\f{1}{(2\pi)^p} \int_{\wti{K_1}^c} dy \left[ \int_{K_1} e^{-\f{1}{2}(y-z)^2} |f_1(z)|\, dz \right]^{2p} \notag \\
    & \lesssim e^{-\f{p}{2} [\text{dist}(K_1,\wti{K_1}^c)]^2} \|e^{-\f{1}{4}| \cdot |^2}*|f_1| \|_{L^{2p}}^{2p} \notag\\
    & \lesssim e^{-\f{p}{2} [\text{dist}(K_1,\wti{K_1}^c)]^2} \|f_1\|^{2p},
 \end{align}
where again Young's inequality, similar as for \eqref{eq:I1-1}, has been used in the last inequality.
Plugging \eqref{eq:I1-1} and \eqref{eq:I1-2} into \eqref{eq:I1}, we obtain
 \beq\label{eq:I1-final}
     I_1 \lesssim (1+4r^2)^{1/2} e^{-\f{p}{4} [\text{dist}(K_1,\wti{K_1}^c)]^2} \|f_1\|^p \|f_2\|^p.
 \eeq
Furthermore, the bound
 \beq\label{eq:I2}
 \begin{aligned}
     I_{2} &\coloneqq  \int_{\wti{K_1}} dy_1 \int_{\wti{K_2}^c} dy_2\
     e^{-\f{(y_1-y_2)^2}{8(1+4r^2)}} \Bigl[(g_{0,0}*|f_1|)(y_1)\Bigr]^{p}\Bigl[(g_{0,0}*|f_2|)(y_2)\Bigr]^{p} \\
     &\lesssim (1+4r^2)^{1/2} e^{-\f{p}{4} [\text{dist}(K_2,\wti{K_2}^c)]^2} \|f_1\|^{p} \|f_2\|^{p}
 \end{aligned}
 \eeq
 follows as the one for $I_{1}$, by symmetry.

It remains to get a bound on
 \begin{align}\label{eq:I3}
     I_{3} & \coloneqq\int_{\wti{K_1}} dy_1 \int_{\wti{K_2}} dy_2\
     e^{-\f{(y_1-y_2)^2}{8(1+4r^2)}} \Bigl[(g_{0,0}*|f_1|)(y_1)\Bigr]^p\Bigl[(g_{0,0}*|f_2|)(y_2)\Bigr]^p.
 \end{align}
 Since $(y_1-y_2)^2\ge (y_1-y_2)^2/2+ [\dist(\wti{K}_1,\wti{K}_2)]^2/2$ in the integral in \eqref{eq:I3}, we get
  \begin{align}\label{eq:I3-1}
     I_{3} & \le e^{-\f{1}{16(1+4r^2)} [\text{dist}(\wti{K_1},\wti{K_2})]^2} \int_{\wti{K_1}} dy_1 \int_{\wti{K_2}} dy_2\
         e^{-\f{(y_1-y_2)^2}{16(1+4r^2)}} \Bigl[(g_{0,0}*|f_1|)(y_1)\Bigr]^p
     \Bigl[(g_{0,0}*|f_2|)(y_2)\Bigr]^p  \notag\\
     & \lesssim (1+4r^2)^{1/2}\,e^{-\f{1}{16(1+4r^2)} [\text{dist}(\wti{K_1},\wti{K_2})]^2} \|g_{0,0}*|f_1|\, \|_{L^{2p}}^p \|g_{0,0}*|f_2|\, \|_{L^{2p}}^p \notag\\
     & \lesssim (1+4r^2)^{1/2}\,e^{-\f{1}{16(1+4r^2)} [\text{dist}(\wti{K_1},\wti{K_2})]^2} \|f_1\|^p \|f_2\|^p
 \end{align}
 using again \eqref{eq:I1-1}.
 Combining
  \bdm
     \|T_rP^{\scriptscriptstyle{\le}}_Lf_1 T_rP^{\scriptscriptstyle{\le}}_Lf_2\|_{L_x^p}^p
     \lesssim L^{2p}\ e^{L^2}(1+4r^2)^{\f{p-1}{2}} \Big( I_1+I_2+I_3 \Big)
  \edm
with \eqref{eq:I1-final}, \eqref{eq:I2}, \eqref{eq:I3-1}, $\dist({K}_j,\wti{K}_j^c)=\tilde{s}$ for $j=1,2$, and  $\dist(\wti{K}_1,\wti{K}_2)=s-2\tilde{s}$, we obtain
 \begin{align*}
      \|T_r P^{\scriptscriptstyle{\le}}_L f_1 T_r P^{\scriptscriptstyle{\le}}_L f_2\|^p_{L^p_x}
      \lesssim L^{2p}\ e^{L^2}(1+4r^2)^{\f{p}{2}}
           \Big[
               e^{-\f{p\wti{s}^2}{4}} + e^{- \f{(s-2\wti{s})^2}{16(1+4r^2)}}
            \Big]
            \|f_1\|^p \|f_2\|^p
 \end{align*}
choosing $\tilde{s}=s/(2\sqrt{p(1+4r^2)} + 2)$, which makes $p\tilde{s}^2/4= (s-2\tilde{s})^2/(16(1+4r^2))$, gives the upper bound
 \bdm
  \begin{split}
   \|T_r P^{\scriptscriptstyle{\le}}_L f_1 T_r P^{\scriptscriptstyle{\le}}_L f_2\|_{L^p_x}
    &\lesssim
    (1+4r^2)^{1/2} L^{2} e^{L^2/p}
               e^{-\f{s^2}{16(\sqrt{p(1+4r^2)}+1)^2}}
            \|f_1\| \|f_2\| \\
     & \le (1+4R^2)^{1/2} L^{2} e^{L^2/p}
               e^{-\f{s^2}{16(\sqrt{p(1+4R^2)}+1)^2}}
            \|f_1\| \|f_2\|
  \end{split}
 \edm
 for all $|r|\le R$,  which proves \eqref{eq:low momenta bilinear}.
\epf

\noindent
\textbf{Acknowledgements: } Mi-Ran Choi and Young-Ran~Lee thank the Department of Mathematics at KIT and  Dirk Hundertmark thanks the Department of Mathematics at Sogang University for their warm hospitality.
	We would also like to thank the anonymous referees for constructive criticism on an earlier version of this work, making us rethink some of our results, leading to, in part,  strong improvements of our previous results.
	Dirk Hundertmark gratefully acknowledges financial support by the Deutsche Forschungsgemeinschaft (DFG) through CRC 1173. He also thanks the Alfried Krupp von Bohlen und Halbach Foundation for financial support. Young-Ran Lee thanks the National Research Foundation of Korea(NRF) for financial support funded by the Korea government under grants (MSIP) No. 2011-0013073 and (MOE) No. 2014R1A1A2058848.
\renewcommand{\thesection}{\arabic{chapter}.\arabic{section}}
\renewcommand{\theequation}{\arabic{chapter}.\arabic{section}.\arabic{equation}}
\renewcommand{\thetheorem}{\arabic{chapter}.\arabic{section}.\arabic{theorem}}

\def\cprime{$'$}


\begin{thebibliography}{100}
\small{


%

\bibitem{Bourgain}
    J.~Bourgain, \textit{Global solutions of nonlinear Schr\"odinger equations.}
    American Mathematical Society Colloquium Publications, \textbf{46}.
    AMS, Providence, RI, 1999.


\bibitem{Christ-Shao-1} M.~Christ and S.~Shao,
	\textit{On the extremizers of an adjoint Fourier restriction inequality.}
	Adv.\ Math.\ \textbf{230} (2012), no.\ 3, 957–-977.

\bibitem{Christ-Shao-2} M.~Christ and S.~Shao,
	\textit{Existence of extremals for a Fourier restriction inequality.}
	Anal.\ PDE \textbf{5} (2012), no.\ 2, 261-–312.



\bibitem{EHL2009} M.~B.~Erdo\smash{\u g}an, D.~Hundertmark, and Y.-R.~Lee,
    \textit{Exponential decay of dispersion managed solitons for vanishing average dispersion.}
    Math.\ Res.\ Lett.\ \textbf{18} (2011), no.\ 1, 13--26.



\bibitem{Fanelli-Vega-Visciglia} L.~Fanelli, L.~Vega, and N.~Visciglia,
	\textit{Existence of maximizers for Sobolev–Strichartz inequalities}
	 Adv.\ Math.\ \textbf{229} (2012), no.\ 3, 1912-–1923.
	

\bibitem{Foschi} D.~Foschi, \textit{Maximizers for the Strichartz inequality.}
	J.\ Eur.\ Math.\ Soc.\ (JEMS) \textbf{9} (2007), no.\ 4, 739–-774.


\bibitem{GT96a} I.~Gabitov and S.~K.~Turitsyn,
    \textit{Averaged pulse dynamics in a cascaded transmission system
    with passive dispersion compensation.}
    Opt.\ Lett.\  \textbf{21} (1996), 327--329.

\bibitem{GT96b} I.~Gabitov and S.~K.~Turitsyn,
    \textit{Breathing solitons in optical fiber links.}
    JETP Lett.\ \textbf{63} (1996) 861.



\bibitem{GV85} J.~Ginibre and G.~Velo,
    \textit{The global Cauchy problem for the nonlinear Schr\"odinger equation.}
    Ann.\ Inst.\  H.\ Poincar\'{e} Anal.\ Non Lin\'{e}aire
    \textbf{2} (1985), 309--327.


\bibitem{GH} W.~Green and D.~Hundertmark, \textit{Exponential Decay for dispersion managed solitons for general dispersion profiles}.
	Lett.\ Math.\ Phys.\ \textbf{106} (2016), no.\ 2, 221–-249.


\bibitem{HuLee2009} D.~Hundertmark, and Y.-R.~Lee,
     \textit{Decay estimates and smoothness for solutions of the dispersion managed   non-linear Schr\"odinger equation.}
     Comm.\ Math.\ Phys.\ \textbf{286} (2009), no.\ 3, 851--873.

\bibitem{HuLee2012}
D.~Hundertmark and Y.-R.~Lee, \textit{ On non-local variational problems with lack of compactness related to non-linear optics.} J. Nonlinear Sci.\ \textbf{22} (2012),  1--38.


\bibitem{HuShao} D.~Hundertmark and S.~Shao,
	\textit{Analyticity of extremisers to an Airy Strichartz inequality.}
	Bulletin of London Mathematical Society, \textbf{44} (2012), 336-–352.


\bibitem{HZ06} D.~Hundertmark and V.~Zharnitsky,
    \textit{On sharp Strichartz inequalities for low dimensions.}
    Int. Math. Res. Notices (2006), Art. ID 34080, 18 pp.


\bibitem{LandauLifshitz}
    L.~D.~Landau and E.~M.~Lifshitz,
    \textit{Course of theoretical physics. Vol. 1. Mechanics.}
    Third edition. Pergamon Press, Oxford-New York-Toronto, Ont., 1976.


\bibitem{LiebLoss}
    E.~H.~Lieb and M.~Loss, \textit{Analysis.} Second edition.
    Graduate Studies in Mathematics, \textbf{14}.
    AMS, Providence, RI, 2001.


\bibitem{LKC80} C.~Lin, H.~Kogelnik, and L.~G.~Cohen,
    \textit{Optical pulse equalization and low dispersion transmission in singlemode
            fibers in the 1.3--1.7 $\mu$m spectral region.}
    Opt.\ Lett.\ \textbf{5} (1980), 476--478.




\bibitem{Moeser05} J.\ Moeser,
	\textit{ Diffraction managed solitons: asymptotic validity and excitation thresholds},
Nonlinearity \textbf{18} (2005), 2275--2297.


\bibitem{Perry} P.~Perry, \textit{Scattering theory by the Enss method.}
	Harwood Academic Publishers, New York, 1983.
	
	
\bibitem{Shao1} S.~Shao, \textit{Maximizers for the Strichartz and the Sobolev-Strichartz inequalities for the Schr\"odinger equation.}
	Electron.\ J.\ Differential Equations (2009), no.\ 3, 13 pp.
	
	
\bibitem{Simon} B.~Simon, \textit{The classical limit of quantum partition functions.}
	Comm.~Math.~Phys.~ \textbf{71} (1980), 247--276.


\bibitem{Stanislavova05} M.~Stanislavova,
    \textit{Regularity of ground state solutions of DMNLS equations.}
    J.\ Diff.\ Eq.\ \textbf{210} (2005), no.\ 1, 87--105.

\bibitem{Strichartz}
    R.~S.~Strichartz,
    \textit{Restrictions of Fourier transforms to quadratic surfaces and
        decay of solutions of wave equations.}
    Duke Math.\ J.\ \textbf{44} (1977), 705--714.


\bibitem{Strichartz00} R.\ S.\ Strichartz,
    \textit{The Way of Analysis.} Revised Edition.
    Jones and Bartlett publications, Sudbury, 2000. 


\bibitem{SulemSulem}
    C.~Sulem and P.-L.~Sulem,
    \textit{The non-linear Schr\"odinger equation. Self-focusing and wave collapse.}
    Applied Mathematical Sciences, \textbf{139}. Springer-Verlag, New York, 1999.

\bibitem{TDNMSF99}
    S.~K.~Turitsyn, N.~J.~Doran, J.~H.~B.~Nijhof, V.~K.~Mezentsev,
    T.\ Sch\"afer, and W.\ Forysiak,
     \textit{Optical Solitons: Theoretical challenges and industrial perspectives.}
    V.~E.\ Zakharov and S.\ Wabnitz, eds.\ Springer Verlag, Berlin, 1999.

\bibitem{Tetal03} S.~K.~Turitsyn, E.~G.~Shapiro, S.~B.~Medvedev, M.~P.~Fedoruk,
    and V.~K.~Mezentsev,
    \textit{Physics and mathematics of dispersion-managed optical solitons.}
    Comptes Rendus Physique, Acad\'emie des sciences/\'Editions scientifiques
    et m\'edicales \textbf{4} (2003), 145--161.


\bibitem{weinstein} M.\ Weinstein,
	\textit{Excitation thresholds for nonlinear localized modes on lattices},
	Nonlinearity \textbf{12} (1999), 673–-691.



\bibitem{ZGJT01} V.~Zharnitsky, E.~Grenier, C.~K.~R.~T.~Jones, and S.~K.~Turitsyn,
\textit{Stabilizing effects of dispersion management.}
    Physica D.\ \textbf{152-153} (2001), 794--817.

}
\end{thebibliography}
\end{document}